\documentclass[11pt,reqno,tbtags,a4paper]{amsart}
\usepackage{amssymb}
\usepackage{amsxtra}
\usepackage{xpunctuate}
\usepackage{url}
\usepackage{cancel}
\usepackage[square,numbers]{natbib}
\bibpunct[, ]{[}{]}{;}{n}{,}{,}

\date{25 November, 2021}

\usepackage{hyperref}

\usepackage{xcolor}
\usepackage[top=2.5cm, bottom=2.5cm, left=3cm, right=3cm]{geometry}

\usepackage{nicefrac}
\usepackage{soul}

\usepackage{enumerate}

\usepackage{comment}

\newcommand{\cec}{\color{violet}}

\newcommand{\opnorm}[1]{{\left\vert\kern-0.25ex\left\vert\kern-0.25ex\left\vert #1 
		\right\vert\kern-0.25ex\right\vert\kern-0.25ex\right\vert}}
\newcommand{\triplenorm}[1]{\|#1\|_{B(V)}}

\title[Fluctuations of balanced urns with infinitely many colours]
{Fluctuations of balanced urns with infinitely many colours}

\subjclass[2020]{60F05, 60F25, 60J80, 62L20}

\numberwithin{equation}{section}

\allowdisplaybreaks

\theoremstyle{plain}
\newtheorem{theorem}{Theorem}[section]
\newtheorem{lemma}[theorem]{Lemma}
\newtheorem{proposition}[theorem]{Proposition}
\newtheorem{corollary}[theorem]{Corollary}

\theoremstyle{definition}

\newtheorem{exampleqqq}[theorem]{Example}
\newenvironment{example}{\begin{exampleqqq}}
  {\hfill\qedsymbol\end{exampleqqq}}

\newtheorem{remarkqqq}[theorem]{Remark}
\newenvironment{remark}{\begin{remarkqqq}}
  {\hfill\qedsymbol\end{remarkqqq}}

\newtheorem{definition}[theorem]{Definition}
\newtheorem{problem}[theorem]{Problem}

\theoremstyle{remark}

\newenvironment{romenumerate}[1][-10pt]{
\addtolength{\leftmargini}{#1}\begin{enumerate}
 \renewcommand{\labelenumi}{\textup{(\roman{enumi})}}%
 }{\end{enumerate}}

\newenvironment{iiromenumerate}[1][-10pt]{
\addtolength{\leftmargini}{#1}\begin{enumerate}
 }{\end{enumerate}}

\newenvironment{PXenumerate}[1]{
\begin{enumerate}
 \renewcommand{\labelenumi}{\textup{(#1\arabic{enumi})}}%
 }{\end{enumerate}}

\newenvironment{PQenumerate}[1]{
\medskip
\begin{enumerate}
 \renewcommand{\labelenumi}{\textup{(#1)}}%
 }{\end{enumerate}
\medskip}

\newcounter{oldenumi}
\newenvironment{romenumerateq}
{\setcounter{oldenumi}{\value{enumi}}
\begin{romenumerate} \setcounter{enumi}{\value{oldenumi}}}
{\end{romenumerate}}

\newcounter{thmenumerate}
\newenvironment{thmenumerate}
{\setcounter{thmenumerate}{0}%
 \def\item{\par
 \refstepcounter{thmenumerate}\textup{(\roman{thmenumerate})\enspace}}
}
{}

\newcounter{xenumerate}   

\newcommand\pfitemx[1]{\par#1:}
\newcommand\pfitemref[1]{\pfitemx{\ref{#1}}}

\newcommand{\refT}[1]{Theorem~\ref{#1}}
\newcommand{\refTs}[1]{Theorems~\ref{#1}}
\newcommand{\refC}[1]{Corollary~\ref{#1}}

\newcommand{\refL}[1]{Lemma~\ref{#1}}
\newcommand{\refLs}[1]{Lemmas~\ref{#1}}
\newcommand{\refR}[1]{Remark~\ref{#1}}

\newcommand{\refS}[1]{Section~\ref{#1}}

\newcommand{\refSS}[1]{Section~\ref{#1}}
\newcommand{\refSSs}[1]{Sections~\ref{#1}}
\newcommand{\refP}[1]{Proposition~\ref{#1}}
\newcommand{\refD}[1]{Definition~\ref{#1}}
\newcommand{\refE}[1]{Example~\ref{#1}}

\newcommand{\refApp}[1]{Appendix~\ref{#1}}

\begingroup
  \divide\count255 by 60
  \multiply\count255 by -60
  \advance\count255 by \time
  \ifnum \count255 < 10 \xdef\klockan{\the\count1.0\the\count255}
  \else\xdef\klockan{\the\count1.\the\count255}\fi
\endgroup

\newcommand{\sumko}{\sum_{k=0}^\infty}

\newcommand{\sumi}{\sum_{i=1}^\infty}

\newcommand{\sumk}{\sum_{k=1}^\infty}

\newcommand{\sumin}{\sum_{i=1}^n}
\newcommand{\sumion}{\sum_{i=0}^n}
\newcommand{\sumjn}{\sum_{j=1}^n}
\newcommand{\sumjp}{\sum_{j=1}^p}

\newcommand{\sumlL}{\sum_{\ell=0}^L}

\newcommand\set[1]{\ensuremath{\{#1\}}}
\newcommand\bigset[1]{\ensuremath{\bigl\{#1\bigr\}}}
\newcommand\Bigset[1]{\ensuremath{\Bigl\{#1\Bigr\}}}

\newcommand\xpar[1]{(#1)}
\newcommand\bigpar[1]{\bigl(#1\bigr)}
\newcommand\Bigpar[1]{\Bigl(#1\Bigr)}
\newcommand\biggpar[1]{\biggl(#1\biggr)}
\newcommand\lrpar[1]{\left(#1\right)}

\newcommand\bigsqpar[1]{\bigl[#1\bigr]}
\newcommand\Bigsqpar[1]{\Bigl[#1\Bigr]}

\newcommand\lrsqpar[1]{\left[#1\right]}
\newcommand\xcpar[1]{\{#1\}}

\newcommand\abs[1]{\lvert#1\rvert}
\newcommand\bigabs[1]{\bigl\lvert#1\bigr\rvert}
\newcommand\Bigabs[1]{\Bigl\lvert#1\Bigr\rvert}
\newcommand\biggabs[1]{\biggl\lvert#1\biggr\rvert}
\newcommand\lrabs[1]{\left\lvert#1\right\rvert}
\def\rompar(#1){\textup(#1\textup)}    
\newcommand\xfrac[2]{#1/#2}

\newcommand\parfrac[2]{\lrpar{\frac{#1}{#2}}}
\newcommand\bigparfrac[2]{\bigpar{\frac{#1}{#2}}}
\newcommand\Bigparfrac[2]{\Bigpar{\frac{#1}{#2}}}
\newcommand\biggparfrac[2]{\biggpar{\frac{#1}{#2}}}

\newcommand\innprod[1]{\langle#1\rangle}

\def\xexp(#1){\ee^{#1}}
\newcommand\ceil[1]{\lceil#1\rceil}

\newcommand\floor[1]{\lfloor#1\rfloor}

\newcommand\ntoo{\ensuremath{{n\to\infty}}}
\newcommand\Ntoo{\ensuremath{{N\to\infty}}}

\newcommand\ktoo{\ensuremath{{k\to\infty}}}

\newcommand\bmin{\wedge}
\newcommand\norm[1]{\lVert#1\rVert}
\newcommand\bignorm[1]{\bigl\lVert#1\bigr\rVert}

\newcommand\upto{\nearrow}
\newcommand\half{{\nicefrac12}}

\newcommand\punkt{\xperiod}    
\newcommand\iid{i.i.d\punkt}    
\newcommand\ie{i.e\punkt}
\newcommand\eg{e.g\punkt}

\newcommand\cf{cf\punkt}
\newcommand{\as}{a.s\punkt}
\newcommand{\aex}{a.e\punkt}

\newcommand\ii{\mathrm{i}}
\newcommand\ee{\mathrm{e}}

\newcommand{\tend}{\longrightarrow}
\newcommand\dto{\overset{\sss\mathrm{d}}{\tend}}
\newcommand\pto{\overset{\sss\mathrm{p}}{\tend}}
\newcommand\asto{\overset{\sss\mathrm{a.s.}}{\tend}}

\newcommand\eqd{\overset{\mathrm{d}}{=}}

\newcommand\bbR{\mathbb R}
\newcommand\bbC{\mathbb C}
\newcommand\bbN{\mathbb N}
\newcommand\bbT{\mathbb T}

\newcounter{CC}
\newcommand{\CC}{\stepcounter{CC}\CCx} 
\newcommand{\CCx}{C_{\arabic{CC}}}     
\newcommand{\CCdef}[1]{\xdef#1{\CCx}}     
\newcommand{\CCname}[1]{\CC\CCdef{#1}}    
\newcounter{cc}

\renewcommand\Re{\operatorname{Re}}
\renewcommand\Im{\operatorname{Im}}

\newcommand\E{\operatorname{\mathbb E{}}}
\renewcommand\P{\operatorname{\mathbb P{}}}

\newcommand\Ge{\operatorname{Ge}}

\newcommand\ga{\alpha}
\newcommand\gb{\beta}
\newcommand\gd{\delta}
\newcommand\gD{\Delta}
\newcommand\gf{\varphi}
\newcommand\gam{\gamma}
\newcommand\gG{\Gamma}

\newcommand\kk{\kappa}
\newcommand\gl{\lambda}
\newcommand\gL{\Lambda}

\newcommand\gs{\sigma}

\newcommand\gss{\sigma^2}
\newcommand\gth{\theta}

\newcommand\gY{\Upsilon}
\newcommand\gthv{\vartheta}
\newcommand\eps{\varepsilon}

\newcommand\cA{\mathcal A}
\newcommand\cB{\mathcal B}

\newcommand\cE{\mathcal E}
\newcommand\cF{\mathcal F}

\newcommand\cK{\mathcal K}
\newcommand\cL{{\mathcal L}}
\newcommand\cM{\mathcal M}
\newcommand\cN{\mathcal N}

\newcommand\cP{\mathcal P}

\newcommand\cR{{\mathcal R}}
\newcommand\cS{{\mathcal S}}
\newcommand\cT{{\mathcal T}}

\newcommand\cX{{\mathcal X}}
\newcommand\cY{{\mathcal Y}}
\newcommand\cZ{{\mathcal Z}}

\newcommand\tB{\tilde B}

\newcommand\indic[1]{\boldsymbol1\xcpar{#1}}

\newcommand\etta{\boldsymbol1}

\newcommand\qw{^{-1}}
\newcommand\qww{^{-2}}
\newcommand\qq{^{\nicefrac12}}
\newcommand\qqw{^{-\nicefrac12}}

\newcommand\intoi{\int_0^1}
\newcommand\intoo{\int_0^\infty}

\newcommand\oi{\ensuremath{[0,1]}}

\newcommand\ooo{[0,\infty)}

\newcommand\setoi{\set{0,1}}

\newcommand\dd{\,\mathrm{d}}
\newcommand\ddx{\mathrm{d}}

\newcommand{\gsf}{$\gs$-field}
\newcommand{\ui}{uniformly integrable}

\newcommand\lhs{left-hand side}
\newcommand\rhs{right-hand side}

\newcommand{\sss}{\relax}

\newcommand\xoo{_1^\infty}
\newcommand\zz[1]{^{(#1)}}
\newcommand\kkk{\zz{k}}

\newcommand\nn{^{\sss (n)}}
\newcommand\nnx[1]{^{\sss (#1)}}
\newcommand\nnz{\nnx{n+1}}
\newcommand\nni{\nnx{i}}

\newcommand\tm{\widetilde m}
\newcommand\Rx{\RR^*}
\newcommand\cME{{\cM(E)}}
\newcommand\BE{{B(E)}}
\newcommand\zzeta{\hat\zeta}
\newcommand\fin{f_{i,n}}
\newcommand\xx{^{**}}
\newcommand\rhoT{\rho(T)}
\newcommand\restr{|}
\newcommand\BB{\mathbf{B}}
\newcommand\RR{\mathbf{R}}

\newcommand\II{\mathbf{I}}
\newcommand\BBB{\widetilde{\mathbf{B}}}
\newcommand\BC{\mathbf{C}}
\newcommand\bbNo{\bbN_0}
\newcommand\Rq{\hat \RR}
\newcommand\gDM{\gD M}
\newcommand\cmc{\cM} 
\newcommand\cmr{\cM_{\bbR}}
\newcommand\cmp{\cM_{+}}
\newcommand\cmpp{\cM_{>0}}
\newcommand\mvpp{{\sc mvpp}}

\newcommand\RYx[1]{R\nnx{#1}_{Y_{#1}}}
\newcommand\RYi{\RYx{i}}
\newcommand\RYn{\RYx{n}}
\newcommand\RYni{\RYx{n+1}}
\newcommand\rrr{^{\sss(1)}} 
\newcommand\RA{R\rrr}
\newcommand\RAn{R^{(n)}}
\newcommand\RAni{R^{(n+1)}} 

\newcommand\raa{\fr\rrr}
\newcommand\ER{\overline R}
\newcommand\Eq{\E\,}
\newcommand\BHN{\ref{as:balance}, \ref{as:BW}, and \ref{as:nu2}}
\newcommand\BHNa{\ref{as:balance}, \ref{as:BW}, and \ref{as:nu2}}

\newcommand\BW{B(W)}

\newcommand\ww{w}
\newcommand\gthd{\gth_D}
\newcommand\nnii{n}
\newcommand\Ceps{C_{\eps}}
\newcommand\fj{\Pi_{\gl_j}f}
\renewcommand\opnorm{\norm}
\newcommand\gox{\chi}
\newcommand\goy{\gox}

\newcommand\fm{\mathfrak m}
\newcommand\fr{\mathfrak r}
\newcommand\fv{\mathfrak{v}}
\newcommand\tfm{\widetilde{\mathfrak m}}
\newcommand\normbwqq[1]{\norm{#1}_{\BW}}

\newcommand\gthh{\gth'}
\newcommand\gthhh{\gth''}
\newcommand\Dx{{D'}}

\newcommand\Ciii{C}  
\newcommand\floorx[1]{\ceil{#1}-1}
\newcommand\ba{a}
\newcommand\cd{d}
\newcommand\iia{\ii\ga}
\newcommand\rhooo{\rho_\infty}
\newcommand\sphatx{\sphat\;}
\newcommand\bigdot{{\hbox{\LARGE$\cdot$}}}
\newcommand\tf{\tilde f}
\newcommand\hD{\widehat D}
\newcommand\nuae{$\nu$-a.e.}
\newcommand\Sigmaq{\Sigma}
\newcommand\hG{\widehat G}
\newcommand\hmu{\widehat{\mu}}
\newcommand\hf{\widehat{f}}
\newcommand\Xcase[1]{(The case #1.)}

\newcommand\Xslqc{simply logarithmically quasi-compact}
\newcommand\slqc{slqc}
\newcommand\aslqc{an \slqc}
\newcommand\xq{^{\nicefrac 1q}}
\newcommand\PPP{\mathbf{P}}
\newcommand\bmu{\boldsymbol{\mu}}

\newcommand{\Holder}{H\"older}
\newcommand{\Polya}{P\'olya}

\newcommand{\bs}{\boldsymbol}

\begin{document}

\author{Svante Janson}
\address{Department of Mathematics, Uppsala University, PO Box 480, SE-751 06 Uppsala,
    Sweden}
\email{svante.janson@math.uu.se}
\thanks{SJ is supported by the Knut and Alice Wallenberg Foundation}

\author{C\'ecile Mailler} 
\address{University of Bath, Claverton Down, Bath BA2 7AY, UK}
\email{c.mailler@bath.ac.uk}
\thanks{CM is grateful to EPSRC for support through the fellowship EP/R022186/1.}

\author{Denis Villemonais}
\address{Université de Lorraine, CNRS, Inria, IECL, F-54000 Nancy, France}
\email{denis.villemonais@univ-lorraine.fr}


\begin{abstract} 
In this paper, we prove convergence and fluctuation results 
for measure-valued P\'olya processes (MVPPs, 
also known as P\'olya urns with infinitely-many colours).
Our convergence results hold almost surely and in $L^2$, 
under assumptions that are different from that of other convergence results in the literature. 
Our fluctuation results are the first second-order results in the literature on MVPPs;
they generalise classical fluctuation results from the literature on finitely-many-colour P\'olya urns.
As in the finitely-many-colour case, the order and shape of the fluctuations depend on whether the ``spectral gap is small or large''.

To prove these results, we show that MVPPs are stochastic approximations taking values in the set of measures on a measurable space $E$ (the colour space). We then use martingale methods and standard operator theory to prove convergence and fluctuation results for these stochastic approximations.
\end{abstract}

\maketitle

\section{Introduction}\label{Sintro}
\subsection{A brief overview of the theory of P\'olya urns}\label{sub:intro_urns}
A $d$-colour P\'olya urn is a stochastic process that 
describes the evolution of an urn containing balls of $d$ different colours. 
It is a Markov process that depends on two parameters: the initial
composition of the urn $\frak u_0 \in \mathbb N^d$ and a replacement
matrix $\frak r = (\frak r_{x,y})_{1\leq x, y\leq d}$, which has integer
entries.  
At time zero, the urn contains $\frak u_{n,x}$ balls of colour $x$, for all $1\leq x\leq d$.
At every discrete time-step, we pick a ball uniformly at random in the urn, and if it is of colour $x$, 
we replace it in the urn together with an additional $\frak r_{x,y}$ balls of colour $y$, for all $1\leq y\leq d$. 
The quantity of interest is the process $(\frak u_n)_{n\geq 0}$, where, for all $n\geq 0$, the vector $\frak u_n = (\frak u_{n,1}, \ldots, \frak u_{n,d})$ is the composition of the urn at time $n$.

As expected, the behaviour of the composition vector at large times depends on the replacement matrix. The case when the replacement matrix is the identity was studied by Markov~\cite{Markov} 
and then P\'olya and Eggenberger~\cite{EP23}. It is well-known that, in this case, $\frak u_n/n$ converges almost surely to a $d$-dimensional Dirichlet random variable of parameter~$\frak u_0$.
The fluctuations around this limit are Gaussian, conditioned on the limit.
(See~\cite[Section~2.3.1]{Noela}.) 

P\'olya urns whose replacement matrix is irreducible 
(the irreducibility assumption can be weakened, 
see Janson~\cite{SJ154})
exhibit a drastically different behaviour, 
see e.g.\ Athreya and Karlin~\cite{AK}: 
in that case, 
if for simplicity all replacements $\frak r_{x,y}$ are non-negative
(this too can be relaxed), 
the
Perron--Frobenius theorem implies that 
the spectral radius $\frak s$ of $\frak r$ is also a simple eigenvalue of $\frak r$, 
and that there exists a unit left-eigenvector~$v$ 
associated to $\frak s$ whose coordinates are all non-negative. 
Then, as $n$ goes to infinity, $\frak u_n/n$ converges almost surely 
to~$\frak s v$.
Interestingly, the fluctuations around this limit are either Gaussian and of order $\sqrt n$, or non-Gaussian and of higher order, depending on the spectral gap of~$\frak r$ (see, e.g. Janson~\cite{SJ154} or Pouyanne~\cite{Pouyanne2008}).

The main differences between the identity and the irreducible cases are that 
(1) the limit of $\frak u_n/n$ is random in the identity case, and deterministic in the irreducible case,  (2) it depends on the initial composition in the identity case, while it does not in the irreducible case, and
(3) the irreducible case sometimes exhibits non-Gaussian fluctuations.

Since these seminal results, the model of P\'olya urns 
has been extended and more precise asymptotic results have been proved.
The most natural extension is to allow balls to be removed from the urn:
It is standard to allow the diagonal coefficients of the replacement matrix to equal $-1$,
meaning that the ball that was drawn is removed from the urn.
One can also allow other coefficients of the replacement matrix to be negative and work conditionally on ``tenability'', which is the event that all coefficients of the composition vector stay non-negative at all times.
The model can also be extended by 
allowing the replacement matrix to be random 
(at each time step, we use a new realisation of this matrix), 
different colours to have different weights or activities 
(a ball is drawn with probability proportional to its weight).
These three generalisations are for example considered in~\cite{SJ154} 
(see~Remark~4.2 therein for ball substractions).

\subsection{Measure-valued P\'olya processes}
Measure-valued P\'olya processes were introduced 
by Bandyopadhyay and Thacker~\cite{BT}, and shortly after by Mailler and Marckert~\cite{MM17}, 
as a generalisation of P\'olya urns to infinitely many colours. 
They both considered cases that can be seen as corresponding to the irreducible case in Section~\ref{sub:intro_urns}.
In fact, the generalisation to infinitely many colours in the diagonal case is much older and dates back to~\citet{BMQ}.

In the analogue of the irreducible case, 
the theory is very recent and, as far as we know, there are only five papers on the subject:
Bandyopadhyay and Thacker~\cite{BT}, Mailler and Marckert~\cite{MM17}, Janson~\cite{SJ327},
Mailler and Villemonais~\cite{MV19}, and Bandyopadhyay, Janson and Thacker~\cite{SJ340}.
The main difficulty is that the linear algebra tools used in the 
study of P\'olya urns are replaced by operator theory in an infinite dimensional space.

In the model introduced by \cite{BT} and \cite{MM17}, 
a measure-valued P\'olya process ({\sc mvpp}) is defined as a Markov process $(\frak m_n)_{n\geq 0}$ 
taking values in the set of positive measures on a measurable space $E$ of
colours. 
The process depends on two parameters again: the initial composition measure $\frak m_0$
and the replacement kernel $(R_x)_{x\in E}$ 
(a family of positive measures on $E$;
see \refApp{Akernels} for measurability issues).

At every discrete time-step $n\geq 1$, a random colour $Y_n$ 
is drawn at random in $E$ with probability distribution 
$\frak m_{n-1}/\frak m_{n-1}(E)$, 
and then $\frak m_n$ is defined as $\frak m_{n-1} + R_{Y_n}$
(see \refS{Smodel} for details).

The authors of \cite{BT} and \cite{MM17} see the {\sc mvpp} 
as a branching version of the $E$-valued Markov chain $(\frak w_n)_{n\geq 0}$
with transition kernel $(R_x)_{x\in E}$. 
They assume that the {\sc mvpp} is ``balanced'', i.e., that the $R_x$'s are
all probability measures, which makes the Markov chain well defined.  
They use this representation to prove that, if $(\frak w_n)_{n\geq 0}$
is ``ergodic'' (in a general sense that allows renormalisations),
then a renormalised version of $\frak m_n/\frak m_n(E)$ converges in probability 
to the limiting distribution of $(\frak w_n)_{n\geq 0}$.
The ``ergodicity'' assumption in this {\sc mvpp} case 
can be seen as the equivalent of the ``irreducibility'' assumption 
in the finitely-many-colour case.
This result is improved by Janson~\cite{SJ327}, who allows the replacement kernel to be random.

Bandyopadhyay, Janson and Thacker~\cite{SJ340} later built on these methods to prove that the convergence results of~\cite{BT} and~\cite{MM17} hold almost surely, under a condition that they call ``uniform ergodicity'' on the underlying Markov chain $(\frak w_n)$, and if the set of colours is countable.

Using a different approach, Mailler and Villemonais~\cite{MV19}
were able to consider non-balanced, weighted {\sc mvpp}s, 
also with random replacements; these are three
generalisations that are classical in the finitely-many-colour case 
and that extend the range of applications.
In the non-balanced case, $R_x$ may be a defective measure, so
the underlying Markov 
chain $(\frak w_n)_{n\geq 0}$ has an absorbing ``cemetery'' state.
The authors show that, if the continuous-time version of the underlying
Markov chain 
admits a quasi-stationary distribution (and under other important assumptions), 
then $\frak m_n/n$ converges almost surely to this quasi-stationary distribution.
They use stochastic approximation methods, 
which is difficult since the stochastic approximation takes values in
a non-compact space as soon as the space of colours is non-compact 
(which is desirable for many applications), 
but which gives almost sure convergence instead of 
the convergence in probability of \cite{BT} and \cite{MM17}. 
The difficulty coming from the fact that the stochastic approximation takes values in a non-compact space is overcome by a Lyapunov-type assumption.
The main drawback of this method is that the Markov chain needs 
to be ``quasi-ergodic'' without any renormalisations, 
whereas renormalisations were allowed in~ \cite{BT} and \cite{MM17}.

\subsection{Our contribution}
In this paper, we prove  limit theorems 
for the fluctuations of an {\sc mvpp} around its almost sure limit:
we are able to generalise the fluctuations results of~\cite{SJ154}
to the infinitely-many-colour case.
Our framework is close to that of~\cite{MV19}, 
although we restrict ourselves to the balanced case; 
we expect the non-balanced case to be more challenging and leave it open for now.

Interestingly, our results do not use the results of~\cite{MV19}: 
they are totally self-contained, 
and our methods also give almost sure convergence of $\frak m_n/\frak m_n(E)$ to its limit, 
under a set of assumptions that are different from those of~\cite{MV19}.
Similarly to~\cite{MV19}, we use a Lyapunov-type assumption to deal with the
fact that, in general, $\frak m_n/\frak m_n(E)$ 
takes values in a non-compact space.

To prove these results, we use stochastic approximation and thus martingale
methods, together with standard operator theory (in particular, we refer
several times to the book of Conway~\cite{Conway}  on the subject).

\subsection{Some notation and conventions}\label{S:not}
``Positive'' is used in the weak sense, i.e., non-negative.

The notation $1$ stands for the usual number, and also for
the function that is constant equal to~$1$ on $E$. Indicator functions are denoted by
$\etta$. 

 $\II$ stands for the identity operator. As usual, for any
complex number $z\in\mathbb{C}$ and for any operator ${T}$, 
the operator ${T}+z$ stands for
${T}+z\II$. 

If $T$ is a bounded operator in a Banach space $\cX$, 
and $\gD$ is a clopen (closed and open) subset of its spectrum $\gs(T)$,
let $\Pi_\gD=\Pi_\gD(T)$ 
denote the corresponding spectral projection in $\cX$.
(See \eg{} 
\cite[VII.3.17--20]{Dunford-Schwartz} or
\cite[Exercise VII.4.9 and VII.(6.9)]{Conway}.) 
In particular, if $\gl$ is an isolated point in $\gs(T)$,
$\Pi_{\lambda}:=\Pi_{\set{\gl}}$ is a projection onto the corresponding generalized
eigenspace.
Note that $T$ commutes with $\Pi_\gD$, and thus
$T$ maps the range $\Pi_\gD (\cX)$ into itself (\ie, $\Pi_\gD(\cX)$ is an
invariant subspace);
moreover the spectrum of the restriction  
of $T$ to $\Pi_\gD\cX$ equals $\gD$ 
\cite[after Equation VII.6.9]{Conway}.

For any non-negative integer $n\geq 1$, $\E_n$ is the conditional
expectation with respect to $\cF_n$, the  
\gsf{} generated by all events up to time $n$, \ie, by $Y_i$ and
$R\nni_{Y_i}$ for $1\le i\le n$.

Let $\cmc(E)$ be the space of complex
measures on $E$ (recall that these are finite by definition),
and let $\cmr(E), \cmp(E), \cmpp(E), \cP(E)$
denote the subsets of finite signed (\ie, real-valued) measures, 
finite positive measures, finite positive non-zero measures, 
and probability measures, respectively.
These sets can all be regarded as measurable spaces, with the $\gs$-fields
generated by the mappings $\mu\mapsto\mu(A)$, $A\in\cE$.

If $\mu$ is a (possibly signed or complex) measure on $E$ and $f$ is a
measurable function, 
then $\mu f:=\int f\dd\mu$ (whenever this is defined).

For a complex measure $\mu$ on $E$, let
$|\mu|$ denotes its total variation measure, and $\norm{\mu}=|\mu|(E)$
its total variation.
If $\ww$ is a positive function on $E$,
then $\cM(\ww)$ is the Banach space of complex measures $\mu$
on $E$, such that the
norm $\|\mu\|_{\ww}:=|\mu|\ww$ is finite.
$\cP(\ww):=\cP(E)\cap\cM(\ww)$ is the subset
of probability measures in $\cM(\ww)$.  

For any positive function $\ww$ on $E$, we define the complex Banach space
\begin{align}
  B(\ww):= \Big\{g:E\to \mathbb C \ \big|\  
\text{$g$ is measurable and }
\|g\|_{B(\ww)}:=\sup_{x\in E} \frac{|g(x)|}{\ww(x)}<+\infty\Big\}
.\end{align}
In the special case $\ww=1$ we write $B(E)$, the space of bounded measurable
functions on~$E$.
Note that $\cM(\ww)$ can be regarded (isometrically)
as a subspace of the dual space $B(\ww)^*$ in the obvious way.

  If $\gb=(\gb_x)_x$ is a  kernel from $E$ to a
measurable space $F$ (see \refApp{Akernels} for definition), 
then $\mu \gb$ denotes the measure on $F$ given by
\begin{align}\label{varin}
  \mu \gb(A):=\int_E \gb_x(A)\dd\mu(x).
\end{align}
(This is the projection onto $F$ of the measure $\mu\otimes\gb$ defined in
\eqref{ker}.)
We extend \eqref{varin} to complex measures $\mu$ and signed kernels $\gb$
such that $\int_E\norm{\gb_s}\dd|\mu|(s)<\infty$.

If ${T}$ is a  bounded operator on $B(w)$ 
such that its adjoint maps $\cM(w)$ into itself,
then we write the adjoint as
${T}$  acting on the right on measures; 
we then have the associativity
\begin{align}\label{ass}
  (\mu{T})f=\mu({T} f)
\end{align}
for (suitable) measures $\mu$ and functions $f$ on $E$. 

For a Banach space $D$, we use $\norm{\cdot}_D$ both for the norm of
elements of $D$, and for the operator norm of operators $D\to D$.

We also make use of the following usual notations and conventions: 
$x\vee y:=\max\set{x,y}$;  $x\land y:=\min\set{x,y}$; $xy\wedge z=(xy)\wedge z$;
empty sums are $0$ and empty products are $1$;
$\inf\emptyset:=+\infty$
and $\sup\emptyset:=-\infty$.

We let $C$ and $\frak C$ denote unspecified constants whose meaning may change
from one occurrence to the next. 
We use $\frak C$ for constants that may depend on $\frak m_0$ 
while $C$ denotes constants that do not depend on $\frak m_0$.
Subscripts may be used to identify specific constants.


\subsection{Plan of the paper}
In Section~\ref{Smodel}, 
we define our model, state and discuss our assumptions and our main results.
Our main results are two main theorems: Theorem~\ref{T1} states convergence of the MVPP, Theorem~\ref{T2} gives the fluctuations of the MVPP around its limit.
In Section~\ref{Spf}, we prove Theorem~\ref{T1}, and in Section~\ref{Spf2}, we prove Theorem~\ref{T2}.
In Section~\ref{Spf3}, we prove Theorems~\ref{T210}-\ref{T230}, 
which give conditions for the limits in Theorem~\ref{T2} to be non-degenerate.
In Section~\ref{Sex}, we apply our main result to four examples:
the out-degree profile of the random recursive tree, 
the heat kernel on the square,
a branching random walk, and reinforced processes on a countable state-space.

Finally, we have three appendices.
In Appendix~\ref{Akernels}, 
we discuss the construction of the MVPP and measurability issues.
In Appendix~\ref{AppFA}, we state some general results on the spectra
of operators on Banach spaces, which are useful for our proofs.
Appendix~\ref{AppC}, we prove a technical lemma that is used in the proof of
Theorem~\ref{T2}.

\section{Model and main results}\label{Smodel}

Let $(E,\mathcal E)$ be a measurable space,
$\RA = (\RA_x)_{x\in E}$
be a set of finite (possibly signed)
random measures on $E$ indexed by $x\in E$,
and let $\fm_0$ be a (non-random) finite measure on $E$.
($E$ may be called the colour space.)
We define the 
measure-valued \Polya{} process ({\sc mvpp}) 
$(\fm_n)_{n\geq 0}$ of initial composition $\fm_0$ and random replacement kernel
$\RA$ as the Markov process given by the following recursion.
See \refApp{Akernels} for 
some technical details, including measurability assumptions.

Given $\fm_n$ with $n\geq 0$,
first sample $Y_{n+1}\in E$ such that $Y_{n+1}$ is a random variable  
whose conditional distribution on $E$,
given $\fm_n$ and  the previous history, 
is 
\begin{align}\label{tmn}
\tfm_n := \fm_n/\fm_n(E).  
\end{align}
Then, let
\begin{align}\label{mn}
\fm_{n+1} := \fm_n + R^{\sss  (n+1)}_{Y_{n+1}},   
\end{align}
where 
$R^{\sss  (n+1)}_{Y_{n+1}}$,
conditioned on $\fm_n$, $Y_{n+1}=y$ and the previous history, 
has the distribution $\cR_y:=\cL(\RA_y)$.

We assume that $\RA_x$ is positive on $E\setminus\{x\}$ but allow 
$\RA_x(\set{x})\in(-\infty,\infty)$. 
We assume that the urn is \emph{tenable}, i.e.\ that almost surely, 
$\fm_n$ is a non-zero positive measure for all $n\geq 0$, so $\tfm_n$ and
$Y_{n+1}$ are well defined.
This is the case if, for example, $\fm_0$  is a non-zero positive measure 
and each $\RA_x$ \as{} is a positive measure.

\begin{remark}\label{Rm0}
We assume for convenience that $\fm_0$ is non-random (except when we
explicitly say otherwise).
Extensions to random $\fm_0$ follow by conditioning on $\fm_0$, 
see Remark~\ref{rem:m0rand} for details.
To enable such extensions, some of the results 
are stated with constants that do not depend on $\fm_0$
(they may depend on the distribution of the replacement kernel $\RA$),
so that the dependence on $\fm_0$ is explicit.
Recall that, by convention, $C$ does not depend on $\fm_0$ while $\mathfrak C$ may depend on $\fm_0$.
The reader who is interested only in a non-random $\fm_0$ may simplify some
expressions and arguments by allowing all constants to depend on $\fm_0$.
\end{remark}

Throughout the paper, we also make the following 
assumptions \BHN.

\medskip
We assume that the urn is
{\it balanced}:
\begin{PQenumerate}{B}
	\item\label{as:balance} For all $x\in E$, $\RA_x(E) = 1$ almost surely. 
\end{PQenumerate}
Note that \ref{as:balance} implies that the total mass is deterministic  a.s.:
\begin{align}\label{mass}
  \fm_n(E)=\fm_0(E)+n.
\end{align}

As said above, $\RA_x$ does not have to be a positive measure. Nevertheless,
we will see that \ref{as:balance} and  our assumption \ref{as:BWii} below imply
that, for every $x\in E$, 
\begin{align}
  \label{fin2}
\Eq \norm{\RA_x}<+\infty
\end{align}
Hence, we can define the expectation $\E \RA_x$ of the random signed measure
$\RA_x$, 
which we denote by $\ER_x$, i.e.,
\begin{align}\label{ER}
  \ER_x(A):=\E \bigsqpar{\RA_x(A)},
\qquad A\in \cE.
\end{align}
 It follows from \eqref{fin2} that $\ER_x$ is a finite signed measure on $E$
 and from \ref{as:balance} that $\ER_x(E)=1$.  
Moreover, $\ER_x$ is positive on $E\setminus\set{x}$, and
it will follow from Assumption~\ref{as:BW} below that
\begin{align}
    \label{eq:ER2}
\sup_{x\in E} |\ER_x(\{x\})|<+\infty.
\end{align}
In particular, $\ER_x f$ is well defined for all non-negative measurable
functions $f:E\to[0,+\infty)$. 
Note also that $\ER$ is a signed kernel from $E$ to $E$ (see \refR{RER}),

\medskip
Let $W:E\to [1,+\infty)$ be a fixed function and 
let $V:=W^q:E\to [1,+\infty)$ for some fixed 
$q>2$. 
We assume that $V$ and $W$ satisfy the following.
\begin{PQenumerate}{H}
	\item\label{as:BW} 
      \begin{iiromenumerate}
        
      \item \label{as:BWi}
There exists $\vartheta\in(0,1)$ and $\CCname{\CCHi}\geq 0$ 
such that, 
for all $x\in E$,
	\begin{align}
	\label{lyapv}
	\ER_xV &\leq \vartheta V(x)+\CCHi
.	\end{align}
      \item \label{as:BWii}
There exists $\CCname{\CCHii}>0$ such
 that, for 
    every $x\in E$, 
    \begin{align}\label{hg}
    \E\left[\left(|\RA_x|W\right)^{q}\right]\leq C_2 W(x)^q=C_2\,V(x)
.    \end{align}
      \item \label{as:BWiii}
 In addition, $\fm_0 V<+\infty$.
      \end{iiromenumerate}
\end{PQenumerate}


\begin{remark}\label{RW=1}
An important case is simply to choose $W=1$, and thus $V=1$.
Note that for $W=1$, \ref{as:BW} is equivalent to assuming that
there exists a constant $C_2>0$ such that 
$\E[\|\RA_x\|^{q}]\leq C_2$.
In particular, if $W=1$ and $\RA_x$ is positive (\as{}) for every $x\in E$,
and thus $\norm{\RA_x}=1$ by \ref{as:balance}, 
then \ref{as:BW} holds automatically.
\end{remark}

\begin{remark}\label{RBWii}
If 
$\RA_x$ a.s.\ is a positive measure, 
and thus a probability measure by \ref{as:balance},
then Jensen's inequality yields
\begin{align}
  \bigpar{|\RA_x|W}^q = \bigpar{\RA_x W}^q \le \RA_x W^q =\RA_x V.
\end{align}
Hence, if also \eqref{lyapv} holds, then
\begin{align}
\E\bigsqpar{ \bigpar{|\RA_x|W}^q} \le \E\bigsqpar{\RA_x V}
=\ER_x V \le \vartheta V(x)+ C_1
\le (\vartheta+C_1)V(x),
\end{align}
i.e., \eqref{hg} holds with $C_2=\vartheta+C_1$.
Consequently, if $\RA_x$ \as{} is a positive measure (for every $x\in E$), then
\ref{as:BWii} follows from \ref{as:BWi} and \ref{as:balance}.

More generally, if we assume that for some constant $C$ and every $x\in E$,
\begin{align}\label{subC}
  |\RA_x\set{x}|\le C \qquad\text{a.s.}
\end{align}
(in other words, subtractions are uniformly bounded),
then \ref{as:balance} implies
\begin{align}\label{subC2}
  \norm{\RA_x}=|\RA_x|(E)\le \RA_x(E)+2 |\RA_x\set{x}|\le C \qquad\text{a.s.}
\end{align}
and \Holder's inequality yields
\begin{align}
  |\RA_x|W \le \bigpar{|\RA_x|W^q}^{\nicefrac 1q} \bigpar{|\RA_x|1}^{1-\nicefrac 1q} 
\le C \bigpar{|\RA_x|V}^{\nicefrac1q}.
\end{align}
Hence, using \eqref{subC} again,
\begin{align}
\bigpar{|\RA_x|W}^q 
\le C |\RA_x| V
\le C \RA_x V + 2C |\RA_x\set{x}|V(x)
\le C \RA_x V + C V(x),
\end{align}
and 
taking expectations, we obtain
\begin{align}
\E\bigsqpar{\bigpar{|\RA_x|W}^q }
\le C \ER_x V + C V(x).
\end{align}
Consequently, if \eqref{subC} holds, then \ref{as:BWii} follows from
\ref{as:BWi} and \ref{as:balance}.
\end{remark}

\begin{remark}\label{RH}
The example in \refS{SSRRT} shows that the assumption \ref{as:BW} is 
important for our results and cannot be weakened much. In particular, it is
not enough to take $q<2$ above, see \refR{Rbest}.
We do not know whether our results hold with $q=2$, and leave that as an
open problem.
\end{remark}

\begin{remark}\label{rk:bounded_H}
    By Jensen's inequality, it follows from \ref{as:BWii} that
 \begin{align}
        \label{eq:dom}
\E[|\RA_x|W]
        \leq \E\big[\big(|\RA_x|W\big)^{q}\big]^{\nicefrac1q}\leq
      C_2^{\nicefrac{1}{q}}\,V(x)^{\nicefrac{1}{q}}=\Ciii W(x)
.   \end{align}
    In particular, this
    implies \eqref{fin2} above. Moreover, it also implies that 
    \begin{align}  
|\ER_x(\{x\})|W(x)
\leq |\ER_x|W
\leq \E[|\RA_x|W]
\leq \Ciii W(x),
    \end{align}
    so that $|\ER_x(\{x\})|\leq \Ciii$, which entails~\eqref{eq:ER2}.
\end{remark}


\medskip
Finally, we assume that, with notation as in \refSS{S:not},
\begin{PQenumerate}{N}
	\item\label{as:nu2} There exists a probability measure $\nu$ such that
   $\nu \ER = \nu$ and $\nu V <+\infty$. 
\end{PQenumerate}

Let
\begin{align}\label{RR}
\RR : f\mapsto (x\in E\mapsto \ER_xf)
\end{align}
be the operator corresponding to 
$\ER$.
Since $\ER$ is a signed kernel from $E$ to $E$, 
$\RR$ maps suitable
(\eg{} bounded) measurable functions on $E$ to measurable functions on $E$.
As remarked above, the balance assumption \ref{as:balance}
implies that $\ER_x(E)=1$ for every $x\in E$, \ie, 
\begin{align}\label{R1}
    \RR 1=1,
\end{align}
and Assumption \ref{as:nu2} yields
\begin{align}\label{nurr}
  \nu\RR=\nu.
\end{align}
We also see that~\eqref{lyapv} can be written
$\RR V \le \vartheta V + \CCHi$.

It follows from~\eqref{eq:dom} that $\RR$ defines a bounded operator on $B(W)$;
by default, we regard $\RR$ as an operator on $B(W)$ unless we say otherwise.
In particular, we let $\sigma(\RR)$ denote the spectrum of $\RR$ on $B(W)$, 
i.e.\ the set of all
$\lambda\in\mathbb C$ such that $\RR-\lambda \II$ is not invertible. 

In the following theorems, which are our main results, 
we increase the generality by considering
$\RR$ as a bounded operator on a closed 
subspace~$D$ of the Banach space~$B(W)$
such that $\RR(D)\subseteq D$ (\ie, $D$ is \emph{stable}, or \emph{invariant});
the most important case is simply $D=\BW$.
 We denote by~$\RR_D$ the
restriction of~$\RR$ to~$D$, and denote its spectrum by $\sigma(\RR_D)$.

To state our main results, we use the following definitions.
\begin{definition}\label{Dsmall}
We say that a bounded operator $T$ on a Banach space $\cX$ is 
\emph{\Xslqc} (\emph{\slqc}) if 
	\begin{PXenumerate}{QC}
		\item \label{ergo3}	
1 is an isolated point in $\gs(T)$, and
the corresponding spectral projection $\Pi_1=\Pi_1(T)$ has rank 1.
		\item \label{ergo2} We have 
		$
		\gs(T)\setminus\set1\subset \set{\gl:\Re\gl<1}$.
	\end{PXenumerate}
\end{definition}
The reason for our name is that the conditions say that the operator $\mathrm e^T$ is
quasi-compact (see \refR{Rqc}) with a single dominating eigenvalue that has a
one-dimensional generalized eigenspace; for convenience, we assume 
that $T$ is normalised such that its dominating eigenvalue  is 1.

By \ref{ergo3}, 
$T$ maps the one-dimensional subspace $\Pi_1(\cX)$ into itself.
Since the restriction of $T$ to this subspace has spectrum \set{1},
it follows that 1 is an eigenvalue of $T$;
moreover, the corresponding eigenvectors are precisely
the non-zero elements of $\Pi_1\cX$; thus the eigenvector 
 is unique up to a scalar factor.
(We can regard \ref{ergo3} as a generalisation of the
finite-dimensional condition that the eigenvalue 1 has algebraic
multiplicity 1.)

\begin{definition}
	We say that an
operator $T$ is \emph{small} if it is \slqc{} and in addition
		\begin{PQenumerate}{S}
\item \label{small1} 
\quad$\gs(T)\setminus\set1\subset \set{\gl:\Re\gl<\frac12}$.
		\end{PQenumerate}
\end{definition}

\begin{remark}\label{Rsmall}
	This definition of \emph{small} operator is analogous to the terminology used in the context of finitely-many-colour urns: 
	a P\'olya urn whose spectral gap is at least half of 
its spectral radius is called a \emph{small} P\'olya urn 
(see, e.g.,~\cite{Pouyanne2008} where this vocabulary is first used).
	We comment later on the similarities and differences between our results 
	and the fluctuation results of~\cite{SJ154} for P\'olya urns.
\end{remark}

We define, for a closed invariant subspace $D\subseteq\BW$,
\begin{align}\label{gthd}
\theta_D:=\sup \Re\bigpar{\sigma(\RR_D)\setminus\{1\}},  
\end{align}
and, in particular, 
\begin{align}\label{gth}
\theta:=\theta_{\BW}=\sup \Re\bigpar{\sigma(\RR)\setminus\{1\}}.
\end{align}
Note that if $T$ is a bounded operator on a complex Banach
space,  
then its spectrum $\gs(T)$ is compact \cite[Theorem~VII.3.6]{Conway}.
This gives immediately: 
\begin{lemma}\label{Lth}
  \begin{thmenumerate}
    
  \item\label{Ltha}   If the operator $\RR_D$ is \slqc, then
$\gthd<1$.

\item\label{Lthb}
If\/ $\RR_D$ is \slqc,
then $\RR_D$ is small if and only if $\gthd<\frac12$.
  \end{thmenumerate}
\qed
\end{lemma}

The first theorem gives several versions of a law of large numbers
for $\tfm_n$.

\begin{theorem}\label{T1}
	Let $(\fm_n)_{n\geq 0}$ be a {\sc mvpp} with initial composition $\fm_0$ 
	and random replacement kernel $\RA$.
	Suppose that $\RA$ satisfies  
	Assumptions \BHN.
	Let~$D$ be a closed invariant subspace of~$B(W)$ such that
	$1\in D$ and the restriction $\RR_D$ of $\RR$ to $D$
	is  \slqc.
    \begin{romenumerate}[-20pt]
        \item \label{T1a} 
Then $\theta_D<1$ and, for every $ \delta\in (0,1-\theta_D)$, 
        there exists a constant $C_\delta$ such that, for any $f\in D$,
        \begin{align}
        \label{eqT1first}
        \E |\tfm_n f- \nu f|^2 
        \le C_\delta  \,\tfm_0 V\,\left(\frac{\fm_0(E)+1}{\fm_0(E)+n}\right)^{2\delta\wedge 1}\normbwqq{f}^2,\qquad \forall n\geq 1.
        \end{align}
        If, in addition, $\gd<\nicefrac12$, then
        \begin{align}
        \label{eqT1firstbis}
        n^{ \delta}|\tfm_n f- \nu f|\xrightarrow[n\to+\infty]{a.s.} 0.
        \end{align}
        \item \label{T1b}
If in addition $\RR$ is \aslqc{} operator on $\BW$, 
        then $\gth<1$
        and, 
        for all $\delta\in(0,{ 1-\theta})$, 
        for all $f\in B(W^2)$, 
         \begin{align}\label{eqT1sec}
        \E\left|\tfm_n f- \nu f\right|&
        \leq C_\gd \,\tfm_0 V\,\left(\frac{\fm_0(E)+1}{\fm_0(E)+n}\right)^{({  2\delta}\wedge 1)\frac{q/2-1}{q-1}}\,\|f\|_{B(W^2)},\qquad \forall n\geq 1.
        \end{align}
    \end{romenumerate}
    \end{theorem}
    
    \begin{remark}
%
        In the case of a metric space $E$ and $D=\BW$,
        \eqref{eqT1firstbis} implies $\tfm_n\asto\nu$ in the usual
        weak topology, but it is stronger since it also implies
        $\tfm_n(A)\asto \nu(A)$ for every measurable set $A$. 
 In particular, we recover and improve on the results obtained in~\cite{MV19} 
        in the balanced non-weighted case.
    \end{remark}

\begin{remark}\label{RcomplexG}
  In the following theorem, we consider the asymptotic distribution of
$\tfm_n f$ for a general complex-valued function $f\in D$. 
In parts (1) and (2) below, the limit is a complex normal distribution, which we
describe by identifying $\bbC$ with $\bbR^2$; we thus give the covariance
matrix of its real and imaginary parts in \eqref{t11} and \eqref{t12}.
Note  that this complex normal distribution 
in \eqref{t11} and \eqref{t12} 
can equivalently be characterised as the distribution of
a complex normal variable
$\zeta$ with 
\begin{align}\label{zeta}
\E\zeta=0,&& \E\zeta^2=\chi(f), && \E|\zeta|^2=\gss(f).&
\end{align}
In applications, we
usually consider real $f$, and then the results simplify since the imaginary
parts disappear; 
in fact, in this case, $\gox(f)=\gss(f)$ is always real, 
and the limit distributions in \eqref{t11} and \eqref{t12} are just
$\mathcal N\bigpar{0,\gss(f)}$.

If $D$ is closed under complex conjugation, for example if $D=\BW$, then
the results for complex $f$ follow easily from the results for real $f$ by
considering real and imaginary parts (and the Cram\'er--Wold device).
Our formulation allows for other interesting domains $D$, for example
$D=\Pi_{\gl}\BW+\bbC1$ where $\gl$ is a non-real isolated point in the spectrum
$\gs(\RR)$.  (See also \refE{Egl}.)
\end{remark}

The second theorem treats the fluctuations around the limit.
As in the finite colour case (see \eg{} \cite{SJ154,Pouyanne2008}), 
there are (under some additional hypotheses)
three cases depending on the size of the spectral gap
(or, equivalently, on $\gthd$); 
in the theorem below we indicate 
the range of $\gthd$ for each case.
Recall that we regard $\RR$ as an operator on $B(W)$.

\begin{theorem}\label{T2}
    	Let $(\fm_n)_{n\geq 0}$ be a {\sc mvpp} with initial composition $\fm_0$ 
    and random replacement kernel $\RA$.
    We assume that $\RA$ satisfies  
    Assumptions \BHN.
    Let~$D$ be a closed invariant subspace of~$B(W)$ such that
    $1\in D$ and the restriction $\RR_D$ of $\RR$ to $D$
    is  \slqc. Then,
 the following hold.
	\begin{enumerate}[{\rm (1)}]
		\item \label{T2.1}\Xcase{$\gthd<1/2$}
If\/ $\RR_D$ is small and $\RR$ is  \slqc, then
for any $f\in D$,
		\begin{align}\label{t11}
		n^{\nicefrac12} (\tfm_n f - \nu f) \dto \mathcal N\left({0,\frac12
	\begin{pmatrix}
	\sigma^2(f)+\Re(\gox(f))&\Im(\gox(f))\\
	\Im(\gox(f))&\sigma^2(f)-\Re(\gox(f))
	\end{pmatrix}	
	}\right),
		\end{align}
		where
		\begin{equation}\label{goxf}
		\gox(f):=\intoo \nu{ \BB\bigpar{\ee^{s\RR}(f-\nu f)}}\ee^{-s}\dd s
		= \int_E\intoo  \BB_x\bigpar{\ee^{s\RR}(f-\nu f)}\ee^{-s}\dd s \dd\nu(x)
		\end{equation}
		and
		\begin{equation}\label{gssf}
		\sigma^2(f):=\intoo \nu{ \mathbf{C}\bigpar{\ee^{s\RR}(f-\nu f)}}\ee^{-s}\dd s
		= \int_E\intoo  \mathbf{C}_x\bigpar{\ee^{s\RR}(f-\nu f)}\ee^{-s}\dd s \dd\nu(x)
		\end{equation}
		with
		\begin{equation}\label{BB}
		\BB(f):  x\mapsto \BB_x(f):=\E \big[\bigpar{\RA_xf}^2\big]
\text{\quad and\quad }
\mathbf{C}(f):  x\mapsto \mathbf{C}_x(f):=\E \big[\bigabs{\RA_xf}^2\big]
		\end{equation}
		and with absolutely convergent integrals.

		\item\label{T2.2} \Xcase{$\gthd=1/2$}
If\/ $\RR_D$ and $\RR$ are \slqc{} and the spectrum of $\RR_D$ is given by  
		\begin{align}
\sigma(\RR_D)=\{1,\lambda_1,\ldots,\lambda_p\}\cup \Delta
\end{align} 
for some $p\ge1$,
where 
	$\Re(\lambda_1)=\cdots=\Re(\lambda_p)=1/2$ and
	$\sup\Re(\Delta) < 1/2$, 
let
\begin{align}\label{kkj}
  \kk_j:=\min\bigset{k\ge1:(\RR_D-\gl_j\II)^k =0 \text{ on }\Pi_{\gl_j}D},
\qquad 1\le j\le p,
\end{align}
and $\kk:=\max_{j\le p}\kk_j$.
Assume that $\kk<\infty$.
Then, for any $f\in D$,
	\begin{align}\label{t12}
	\frac{n^{\nicefrac12}}{(\log n)^{\kappa-\nicefrac12}}  \big(\tfm_n f - \nu f\big)\dto \mathcal N\left({0,\frac12
		\begin{pmatrix}
		\gs^2(f)+\Re(\goy(f))&\Im(\goy(f))\\
		\Im(\goy(f))&\gs^2(f)-\Re(\goy(f))
		\end{pmatrix}	
	}\right),
	\end{align}
	where
	\begin{align}\label{goxf2}
\goy(f):= \sum_{j,j'=1}^p  
\frac{\etta_{\kappa_j=\kappa_{j'}=\kappa,\
      \overline{\lambda_j}=\lambda_{j'}}}{(2\kk-1)((\kappa-1)!)^2}
\nu\BBB\left((\RR-\lambda_j \II)^{\kappa-1}\Pi_{\lambda_j} f,(\RR-\lambda_{j'} \II)^{\kappa-1}\Pi_{\lambda_{j'}} f\right)		
	\end{align}
	and
	\begin{align}\label{gssf2}
	\gs^2(f):= \sum_{j=1}^p  \frac{\etta_{\kappa_j=\kappa}}
{(2\kk-1)((\kappa-1)!)^2}{\nu\mathbf{C}}\left((\RR-\lambda_j \II)^{\kappa-1}\Pi_{\lambda_j} f\right)
	\end{align}
with
$\BC$ as in \eqref{BB} and
\begin{equation}\label{BBB}
\BBB(f,g):  x\mapsto \BBB_x(f,g):=\E [\RA_x f\cdot \RA_xg],
\end{equation}

	\item\label{T2.3} \Xcase{$\gthd>1/2$}
 If\/ $\RR_D$ 
is \slqc{} and the spectrum
      of $\RR_D$ is given by  
\begin{align}\label{ingemar}
\sigma(\RR_D)=\{1,\lambda_1,\ldots,\lambda_p\}\cup \Delta
\end{align} 
for some $ p\geq 1$,
where 
$\Re(\lambda_1)=\cdots=\Re(\lambda_p)\in(1/2,1)$, 
and
$\sup\Re(\Delta) < \Re(\lambda_1)$, 
let $\kk_j$ $(1\le j\le p)$ be defined by \eqref{kkj} and 
let $\kk:=\max_{j\le p}\kk_j$.
Assume that $\kk<\infty$.
Then, for any $f\in D$, there exist 
complex random variables
$\gL_1,\dots,\gL_p\in L^2$ such that
	\begin{align}\label{hildegran}
\frac{n^{1-\Re\lambda_1}}{(\log n)^{\kappa-1}}\bigpar{\tfm_n f - \nu f}
- 	\sum_{j=1}^p  n^{\ii \Im\gl_j} \gL_j
\to 0
	\end{align}
\as{} and in $L^2$.
Furthermore, 
\begin{align}
  \label{solveig}
\E\gL_j
=\frac{\gG(\fm_0(E)+1)}{(\kk-1)!\,\gG(\fm_0(E)+\gl_j)}
 \tfm_0(\RR-\gl_j)^{\kk-1}\Pi_{\gl_j}f.
\end{align}
	\end{enumerate}
\end{theorem}

\begin{remark}
        \label{rem:m0rand}
To adapt our results to a random~$\fm_0$, 
we make the same assumptions as in~Theorems~\ref{T1} and~\ref{T2}, 
and we assume that~\ref{as:BWiii} holds almost surely.
Under these assumptions, Theorems~\ref{T1} and~\ref{T2} apply conditionally on $\fm_0$.
This implies that, under these assumptions, 
the almost-sure convergences in Theorems~\ref{T1} and~\ref{T2} 
still hold for random~$\fm_0$.
Furthermore, since the limiting distributions in Theorem~\ref{T2}~\eqref{T2.1} and~\eqref{T2.2} 
do not depend on $\fm_0$, it also implies
that the convergences in distribution in Theorem~\ref{T2}(\ref{T2.1}) and~(\ref{T2.2}) 
hold if $\fm_0$ is random.
If in addition $\E\tfm_0 V<+\infty$, then, by dominated convergence, 
the left-hand-side terms of~\eqref{eqT1first} and~\eqref{eqT1sec} also converge to~$0$ 
when $n\to+\infty$.

Under the assumptions of Theorem~\ref{T2}\eqref{T2.3}, 
conditioning on $\fm_0$ shows that \eqref{hildegran} still
holds \as{} for a random $\fm_0$,
with \eqref{solveig} replaced by
    \begin{align}
    \label{solveigm0rand}
    \E\gL_j
    =\E\left[\frac{\gG(\fm_0(E)+1)}{(\kk-1)!\,\gG(\fm_0(E)+\gl_j)}
    \tfm_0(\RR-\gl_j)^{\kk-1}\Pi_{\gl_j}f\right].
    \end{align}
Moreover, it follows from the proof that 
under the additional assumption that 
\begin{align}\label{m0rand+}
\E\left[(\fm_0(E)+1)^{2(1-\Re\gl_1)}\tfm_0 V\right]<\infty,   
\end{align}
\eqref{hildegran} holds also in $L^2$,
see \refR{rem:m0rand2}.
(Note that $\E\left[(\fm_0(E)+1)\tfm_0 V\right]<\infty$ suffices
for~\eqref{m0rand+}.)
\end{remark}

\begin{remark}\label{RBB}
	The operator $\BB$ defined in~\eqref{BB} is the quadratic operator
    corresponding to the bilinear operator $\BBB$ in \eqref{BBB}, \ie,
    $\BB(f)=\BBB(f,f)$. Similarly, $\BC(f)=\BBB(f,\overline f)$.
It follows from \eqref{hg} that the bilinear map
$\BBB$ is bounded, and thus continuous, as a mapping $\BW\times\BW\to B(W^2)$. 
Indeed, for all $f,g\in B(W)$ such that $\|f\|_{B(W)}=\|g\|_{B(W)}=1$, we have, for all $x\in E$,
\begin{align}
\E [|\RA_x f\cdot \RA_xg|]\leq \E\left[(|\RA_x|W)^2\right]\leq \E\left[(|\RA_x|W)^q\right]^{\nicefrac2q}\leq C\,W(x)^2,
\end{align}
where we used Jensen's inequality and~\eqref{hg}.
Hence, $\BB$ and $\BC$ are continuous maps $\BW\to B(W^2)$, and
\begin{align}\label{rbb}
  \norm{\BB(f)}_{B(W^2)}\le C \norm{f}_{\BW}^2,
\qquad
  \norm{\BC(f)}_{B(W^2)}\le C \norm{f}_{\BW}^2
.\end{align}
\end{remark}

\begin{remark}\label{Rsub}
Just as in the finitely-many-colour case,
the limit result \eqref{hildegran} implies convergence in distribution 
of 
$\frac{n^{1-\Re\lambda_1}}{(\log n)^{\kappa-1}}\bigpar{\tfm_n f - \nu f}$
for  suitable subsequences, but in general not for the full sequence. 
\end{remark}

\begin{remark}\label{Rjoint}
The asymptotic normality in parts (1) and (2) extends immediately to joint
convergence for any  number of $f\in D$, 
by the Cram\'er--Wold device \cite[Theorem 5.10.5]{Gut};
the asymptotic covariances are given by obvious bilinear analogues of the
variance formulas in the theorem (\cf{} \refR{RBB}).

In part (3), joint (subsequence) convergence in distribution 
for several $f\in D$ is immediate from the \as{}
convergence in \eqref{hildegran}.
\end{remark}

\begin{remark}\label{Rkk}
  The assumption $\kk<\infty$ in parts (2) and (3) holds,
in particular, 
if we have $\dim (\Pi_{\lambda_j}D)<\infty$ for each $j\le p$.
To see this,
let $D_j:=\Pi_{\gl_j}D$, and note that if $\dim(D_j)<\infty$, then 
the restriction $R_{D_j}$ is an operator in the finite-dimensional vector
space $D_j$ with spectrum $\gs(R_{D_j})=\set{\gl_j}$; hence the operator
$R_{D_j}-\gl_j\II$ is nilpotent (as is shown by the Jordan decomposition),
and thus $\kk_j$ in \eqref{kkj} is finite; in fact,
\begin{align}\label{kkd}
\kk_j\le \dim(D_j).   
\end{align}
Hence, $\kk\le\max_j \dim(D_j)<\infty$ if all $D_j=\Pi_{\gl_j}D$
have finite dimensions.
\end{remark}

\begin{remark}\label{RD}
	Note that allowing a domain $D\subset B(W)$ leads to a
    more complete result. 
    For instance, if $\Delta$ is a
    clopen subset of $\sigma(\RR)$, then one can
    consider the operator $\RR_D$ acting on
    $D=\Pi_{\Delta}B(W)+ \mathbb{C}1$, whose spectrum is
    $\{1\}\cup\{\Delta\}$ which may be 
    strictly included in~$\sigma(\RR)$. In that case, the assumptions
 in Theorem~\ref{T1}(1)--(3) on $\RR_D$ become assumptions on $\gD$,
and then the theorem yields results for $f\in D$, even if the assumptions are
not satisfied for $\gs(\RR)$.

For another example where subspaces are useful, see \refR{Rhomo}.
\end{remark}

\begin{example}\label{Egl}
  We give a simple example; further examples are given in \refS{Sex}.
Suppose that $\RR$ is \slqc{} in $\BW$, and that $f\in\BW$ 
is an eigenfunction: $\RR f=\gl f$ with $\gl\neq1$.
Then \refT{T1} applies to the two-dimensional space $D$ spanned by $f$ and~$1$.

If $\Re\gl<\half$, then (1) yields the asymptotic normality \eqref{t11}.
We have $\nu f=0$ and $\ee^{s\RR}f=\ee^{s\gl}f$, and thus \eqref{goxf} and
\eqref{gssf} yield $\gox(f)=(1-2\gl)\qw\nu\BB(f)$ and 
$\gss(f)=(1-2\Re\gl)\qw\nu\BC(f)$.

If $\Re\gl=\half$, then (2) applies instead, with $p=1$ and $\kk=1$;
\eqref{goxf2} and \eqref{gssf2} yield $\goy(f)=\nu\BB(f)$ and
$\gs^2(f)=\nu\BC(f)$. 

Finally, if $\Re\gl>\half$, then (3) applies, with $\sigma(\RR_D) = \{1,
\lambda\}$ and $\kk=1$;
\eqref{hildegran} shows that there exists a complex random variable
$\Lambda$ such that
$n^{1-\mathrm{Re}(\lambda)}(\tilde{\frak m}_n f - \nu f) - n^{\mathrm{iIm}(\lambda)}\Lambda \to 0$, {and hence $n^{1-\lambda}(\tilde{\frak m}_n f - \nu f) \to \Lambda$,}
almost surely and in $L^2$ when $n\to+\infty$.
\end{example}

\begin{remark}
    In Theorem~\ref{T2}, the assumption that $1\in D$ is in fact not
    necessary. We make this assumption for convenience and because, in
    practice, as one can see in Example~\ref{Egl}, $1$ can always be
    added to~$D$ to enter the setting of our results. 
\end{remark}

\begin{example}\label{ex:finite_nb_colours}
The classical generalised P\'olya urn model with finitely-many colours 
is given by $E = \{1, \ldots, d\}$ and
$\RA_x = (\raa_{x,1}\delta_1 + \cdots + \raa_{x,d}\delta_d)/S$, 
where $\raa = (\raa_{x,y})_{1\leq x, y\leq d}$
is a (possibly random) matrix of  integers,  
with $\raa_{x,y}\ge0$ when $x\neq y$,  
$\raa_{x,x}\geq -1$ for all $1\leq x\leq d$,
$\gd_x$ is a point mass (Dirac measure) at $x$,
and $S$ is a scaling factor (for convenience).
We apply our results to that case and compare the outcome to results from
the literature. 
This model satisfies 
\begin{itemize}
\item[\ref{as:balance}] if and only if the replacement matrix $\raa$ 
is ``balanced'', i.e.\ if all row sums are equal to $S$ (a.s.);
\item[\ref{as:BW}] 
always when \ref{as:balance} holds, since then $-1\le \raa_{x,y}\le S+1$ a.s.
(We take $V\equiv 1$.)
\item[\ref{as:nu2}]
always when \ref{as:balance} holds,
since  then the non-negative matrix
$\E\raa+\II=(\E[\raa_{x,y}+\gd_{x,y}])_{x,y}$ 
has a positive right eigenvector with eigenvalue $S+1$ (viz.\ $1$),
and therefore it follows from the 
the Perron--Frobenius theorem  that it also
has a non-negative left eigenvector $u=(u_x)_1^d$ with this eigenvalue;
we may assume that $\sum_x u_x=1$ and then take $\nu=\sum_x u_x\gd_x$.
\end{itemize}
Furthermore, $B(1)=B(E)$ is the space of all functions from $\{1, \ldots,
d\}$ to $\mathbb C$,  
i.e.\ $\mathbb C^d$.
Under~\ref{as:balance}, 
the operator $\RR$
defined by $\bar{\fr}/S = \mathbb E[\raa]/S$ on $\mathbb C^d$ is \slqc{} if
and only if the eigenvalue~1 has (algebraic) multiplicity~1. 
Under these assumptions, \eqref{eqT1first}--\eqref{eqT1firstbis} 
imply that, if $\frak u_n$
is the composition vector of the urn at time~$n$, i.e. the vector whose
$i$-th coordinate is the number of balls of colour $i$ in the urn at time
$n$, then
\begin{align}
\|\frak u_n - v\| 
= o(n^{-{ \delta}})\quad\text{a.s.\ and in $L^2$ as }n\to\infty,
\end{align}
for all $\delta \in (0, (1-\theta)\land{\nicefrac12})$, 
where $\theta$ is the maximum of the real parts of the eigenvalues of $\frak
r$ excluding~1.
Furthermore, Theorem~\ref{T2}(1) and (2) allow us to recover versions of the 
limit theorems \cite[Theorems~3.22, 3.23 and~3.24]{SJ154}: 
the only caveat is that we make the additional 
assumption that the replacement matrix is balanced.
\end{example}

\begin{remark}\label{Rweights}
As mentioned in the introduction, it is standard in the theory of P\'olya 
urns to associate different weights (or activities) to the balls of
different colours. 
In this generalisation, when picking a ball at random at the $n$-th step, 
one pick each of the balls with probability proportional to its weight, 
and then applies the replacement rule as normal depending on the colour of the drawn ball.

In~\cite{MV19}, the authors generalise this concept of weight in the infinitely-many-colour case: 
for all positive kernel ${\bf P} = (P_x)_{x\in E}$, they define the {\sc
  mvpp} $\frak m_n$ as in~\eqref{mn}, except that, conditionally on $\frak
m_n$, $Y_{n+1}$ is drawn according to the distribution 
${\frak m}_n {\bf P}/{\frak m}_n {\bf P}(E)$.

One can apply our main results (Theorems~\ref{T1} and~\ref{T2}) to $\frak m'_n := \frak m_n{\bf P}$, 
which is an {\sc mvpp} of replacement kernel~${\bf RP}$. 
Our assumptions require in particular that ${\bf RP}$ satisfies the balance assumption \ref{as:balance}. 
From our main results applied to $\frak m'_n$, if the operator induced by ${\bf P}$ is invertible, 
one can deduce a fluctuation result for the original weighted {\sc mvpp} $\frak m_n$.

Even if ${\bf P}$ were non-invertible, it would be straightforward to 
generalise our proofs to the weighted case under the assumption that ${\bf RP}$ is balanced;
since our proofs are already technical, and since the balance assumption restricts greatly the set of weighted kernels one could use, we do not extend our framework to include this case. 
\end{remark}

\begin{remark}\label{Rspaces}
In  the theorems above, we regard $\RR$ as an operator on $B(W)$, where the
possibility to choose a suitable $W$ gives additional flexibility. 
(Warning: the spectrum $\gs(\RR)$, and thus \eg{} $\gth$, may change if we
change $W$, see the example in \refSS{SSRRT}.)
The space $B(W)$ seems natural and convenient for applications, but it is
not the only reasonable choice of a function space.

First,
in typical cases, we may ignore functions that are 0 \nuae\ and it is then
equivalent to consider $\RR$ as an operator on the quotient space of $B(W)$
modulo functions that are 0 \nuae, which we denote by
$L^\infty(W;\nu)$; see \refL{Las} which implies that $\RR$ always is well
defined on $L^\infty(W;\nu)$. However, \refE{Enull} shows that there are
(exceptional) cases when null sets and functions cannot be completely ignored.

More importantly, the examples in \refSSs{SSheat} and \ref{SSBRW}
use Fourier analysis and it is then convenient to consider $\RR$ as an
operator on $L^2(E,\nu)$. In these examples we transfer spectral properties
of $\RR$ from $L^2(E,\nu)$ to $B(E)$ and then apply the theorems above.
However, for these and other similar examples, it would be desirable to have
extensions of the theorems above 
where $B(W)$ is replaced by a more general function space on $E$, including 
$L^2(E,\nu)$ as a possible choice. (Other choices might also be useful in other
applications.) 
In the present paper, however, we consider only the theorems as stated
above, with $\RR$ acting on $\BW$ (and invariant subspaces thereof).
\end{remark}

\subsection{Degenerate limits?}\label{Snull}

The limit results in \refT{T2} are less interesting when the limit
distribution is identically 0.
We  characterize here these degenerate cases, 
and begin by showing that in part (1) of
\refT{T2}, the limit is non-degenerate except in trivial cases.
Proofs are given in \refS{Spf3}.

\begin{theorem}\label{T210}
  Suppose that the conditions of \refT{T2}(1) hold, and let $f\in D$.
Let $\Sigmaq(f)$ be the covariance matrix in \eqref{t11}. 
Then the following are equivalent:
\begin{romenumerate}
  
\item \label{T210a}$\Sigmaq(f)=0$.
\item \label{T210b}$\gss(f)=0$.
\item \label{T210c} $\nu \BC(f-\nu f)=0$.
\item \label{T210d} $\RA_x f = \nu f$ \as, for \nuae~$x$.
\end{romenumerate}
\end{theorem}

On the contrary, in \refT{T2}(2), the asymptotic distribution depends only
on $\Pi_{\gl_1}f,\dots,\Pi_{\gl_p}f$, and thus it degenerates to 0 for many
$f$. (For such $f$, it might be possible to apply the theorem with a smaller
space $D$.)
In fact, in typical applications, the projections $\Pi_{\gl_j}$ project onto
finite-dimensional subspaces, and thus their kernels are very large.
We have the following characterization.

\begin{theorem}\label{T220}
  Suppose that the conditions of \refT{T2}(2) hold, and let $f\in D$.
Let $\Sigmaq(f)$ be the covariance matrix in \eqref{t12}. 
Then the following are equivalent:
\begin{romenumerate}
  
\item \label{T220a}$\Sigmaq(f)=0$.
\item \label{T220b}$\gss(f)=0$.
\item \label{T220c} $(\RR-\gl_j\II)^{\kk-1}\Pi_{\gl_j} f =0$ \nuae, for
every  $j=1,\dots,p$.
\item \label{T220d} $(\RR-\gl_j\II)^{\kk-1}\Pi_{\gl_j} f =0$ \nuae, for
every  $j=1,\dots,p$ such that $\kk_j=\kk$.
\end{romenumerate}
\end{theorem}

 In \refT{T2}(3), the situation is similar to
\refT{T2}(2). The characterization is more technical,
partly because the limit distribution now also depends on the initial values
$\fm_0$;
we give several equivalent conditions.
Note that the sum $\sum_j n^{\ii \Im \gl_j}\gL_j$ in \eqref{hildegran}
vanishes for all $n$ if and only if $\gL_j=0$ for each $j=1,\dots,p$.
In typical cases, the conditions below are satisfied only for $g_j=0$,
but \refE{Enull} gives an example where $g_j$ is non-zero and 
the conditions are satisfied for some, but not all, $\fm_0$.

\begin{theorem}\label{T230}
  Suppose that the conditions of \refT{T2}(3) hold, and let $f\in D$.
Let $\gL_j$ be as in \eqref{hildegran},
and let $g_j:=(\RR-\gl_j)^{\kk-1}\Pi_{\gl_j}f$.
%
Then the following are equivalent, for each $j=1,\dots,p$:
\begin{romenumerate}
  
\item \label{T230a}
$\gL_j$ is (a.s.)\ non-random.

\item \label{T230b}
$\gL_j=\E\gL_j$ a.s. 

\item \label{T230o}
$\gL_j=0$ a.s.

\item \label{T230g}
$\fm_n g_j =0$ a.s., for every $n\ge0$.

\item \label{T230h}
$\fm_0 g_j =0$ and 
$\RYni g_j =0$ a.s., for every $n\ge0$.

\item \label{T230p}
$\fm_n |g_j| =0$ a.s., for every $n\ge0$.

\item \label{T230q}
$\fm_n\set{x:|g_j(x)|\neq0}=0$ a.s., for every $n\ge0$.

\item \label{T230r}
$\fm_0 \RR^n\abs{g_j}=0$, for every $n\ge0$.
\end{romenumerate}
Moreover, 
if\/ $\RR$ is \slqc{} on $\BW$, then
\ref{T230a}--\ref{T230r} imply 
\begin{romenumerateq}
\item \label{T230nu}
$g_j=0$ \nuae
\end{romenumerateq}
Conversely, if $\fm_0$ is absolutely continuous w.r.t.\ $\nu$, then
\ref{T230nu} implies
\ref{T230a}--\ref{T230r}.
\end{theorem}

\begin{remark}\label{Rjoint2}
It follows from \refTs{T220} and \ref{T230} that
when considering joint limits
for several $f\in D$ in 
parts (2) and (3) of \refT{T2}
(see \refR{Rjoint}), 
the limit distribution is supported on 
a subspace of dimension at most 
\begin{align}
    \sum_{j=1}^p\dim\bigsqpar{(\RR-\gl_j)^{\kk-1}\Pi_{\gl_j}(D)}
\end{align}
with equality in typical cases (we leave the details to the reader).
Note that this is always at most $\sum_{j=1}^p\dim\bigsqpar{\Pi_{\gl_j}(D)}$.
\end{remark}

\section{Proof of Theorem~\ref{T1}}\label{Spf}

We assume throughout this section that Assumptions \BHN{} hold. 
Recall that $\fm_0$ is non-random
(unless we explicitly say otherwise),
and that constants $C$
do not depend on $\fm_0$. 
The claims about $\theta$ and $\gth_D$ in Theorem~\ref{T1}
follow by \refL{Lth}.

\subsection{Preliminary results}\label{SSprel}

Define,	for $n\ge0$, the random signed measure
    \begin{align}
      \label{vn}
\fv_n:= \tfm_n - \nu.
    \end{align}

\begin{lemma}\label{lem:algo_sto}
 For all $n\geq 0$,
	\begin{align}\label{a1}
		\fv_n=\fv_0B_{0,\nnii}+\sumin\gamma_{i-1}\gD M_i B_{i,\nnii},
	\end{align}
	where, for all $n\geq 0$ and $0\le i\le n$,
	\begin{align}
		B_{i,n}&:=\prod_{j=i}^{n-1} (\II+\gamma_j(\RR-\II)),\label{Bmn}\\
		\gD M_{n+1} &
		:= \RYni-\E_{n} \RYni
		= \RYni-\E_{n} \ER_{Y_{n+1}}
		= \RYni-\tfm_{n} \RR,
\label{a2}\\
		\gam_n&
		:=\frac{1}{n+1+\fm_{0}(E)}.\label{gamma}
	\end{align}
\end{lemma}

\begin{proof}
	By definition, for any $n\geq 0$, we have
\eqref{mn},
	where the conditional distribution of $Y_{n+1}$ given $\fm_n$ is $\tfm_n$.
Furthermore, \ref{as:balance} implies, see \eqref{mass},
\begin{align}
  \label{baln+1}
\fm_{n+1}(E) = \fm_0(E) + n+1 = \nicefrac1{\gamma_n}.
\end{align}
 Together with \eqref{a2}, this implies
	\begin{align}\label{tm1}
		\tfm_{n+1}&
		=\frac{\fm_{n+1}}{\fm_{n+1}(E)}
		=\frac{\fm_n}{\fm_{n+1}(E)}+\gamma_nR_{Y_{n+1}}\nnx{n+1}
		\notag\\&
		=\frac{\fm_n(E)}{\fm_{n+1}(E)}\cdot \tfm_n+\gam_n R_{Y_{n+1}}\nnx{n+1}
		=(1-\gam_{n}) \tfm_n+\gam_n R_{Y_{n+1}}\nnx{n+1}
		\notag\\&
		= \tfm_n+\gam_n\tfm_n(\RR-\II)+\gam_n \gD M_{n+1},
	\end{align}
 By definition, $\fv_n=\tfm_n-\nu$, and by 
\eqref{nurr}, we have
$\nu(\RR-\II)=\nu\RR-\nu=0$; therefore, \eqref{tm1} implies
	\begin{align}\label{vn+1}
		\fv_{n+1}&= \fv_n+\gam_n \fv_n(\RR-\II)+\gam_n \gD M_{n+1}
= \fv_n\bigpar{\II+ \gamma_n(\RR-\II)}+\gam_n \gD M_{n+1},
	\end{align}
and \eqref{a1} follows by induction.
\end{proof}

As noted above,
it follows from \eqref{eq:dom}
that $\RR$ is a bounded operator on
$\BW$; hence every $B_{i,n}$ is too.  
Dually, $\RR$ and $B_{n,i}$ are bounded operators on $\mathcal M(W)$ 
(acting on the right). Moreover:

\begin{lemma}\label{Lmwqq}
A.s., for every $n\ge0$, we have $\fm_n, \tfm_n, \fv_n, \RYni, \gDM_{n+1}\in\mathcal M(W)$.
Moreover,  
\begin{equation}\label{lmwqq}    
\E \bigsqpar{\fm_n W} <\infty
\qquad\text{and}\qquad
\E \bigsqpar{|\fv_n| W} <\infty,
\end{equation}
and 
\begin{equation}
\sup_{n\geq 0} \E \tfm_n V  \leq C\tfm_0 V
\quad\text{ and }\quad
\sup_{n\geq 1} \E V(Y_n) \leq C\tfm_0 V. \label{lmv4a}
\end{equation}
Finally, there exists a constant $C$ such that for every $g\in\BW$
\begin{align}
  \label{lg0}
 |\fv_0g|^2 &\le  C\normbwqq{g}^2\,(\tfm_0 W)^2
 \le  C\normbwqq{g}^2\,\tfm_0 V,
\\
  \label{lg}
\E\bigabs{\gD M_i g}^2
&\le C\normbwqq{g}^2\,\tfm_0 V,
\qquad i \ge 1,\\
\label{lgnew}
\E_{i-1}\bigabs{\gD M_i g}^q 
& \leq  C\normbwqq{g}^q\,\tfm_{i-1}(V),
\qquad i \ge 1.
\end{align}
\end{lemma}

\begin{proof}
{\bf We start by proving $\fm_n, \tfm_n, \fv_n, \RYni, \gDM_{n+1}\in\mathcal
  M(W)$ and~\eqref{lmwqq}.}
By construction,
the conditional distribution of $R^{\sss (n+1)}_{Y_{n+1}}$ 
given $\cF_n$ and $Y_{n+1}=y$, for some $y\in E$,
equals the distribution of $\RA_y$. 
Hence, using \ref{as:BWii} through its consequence \eqref{eq:dom},
\begin{align}\label{palma}
  \E\bigsqpar{\abs{\RYni W} \bigm| \cF_n,Y_{n+1}=y}& 
=  \E\bigabs{\RA_y W}
\le  \E\bigsqpar{|\RA_y| W}
\le C W(y).
\end{align}
In other words, a.s.,
\begin{align}\label{palmb}
 \E\bigsqpar{\abs{\RYni W} \bigm| \cF_n,Y_{n+1}}& 
\le C  W(Y_{n+1}).  
\end{align}
Furthermore,
$Y_{n+1}$ has the conditional distribution $\tfm_{n}$ 
given $\cF_{n}$, 
and thus, by taking the conditional expectation $\E_n$ in \eqref{palmb},
\begin{align}\label{palmc}
 \E_n{\bigabs{\RYni W}}&
= \E\bigsqpar{\abs{\RYni W} \bigm| \cF_n}
\le C \E\bigsqpar{ W(Y_{n+1})\bigm|\cF_n}
=C\int_E W(y) \dd\tfm_n(y)
\notag\\&
=C \tfm_n W.
\end{align}
Hence, by \eqref{mn}, \eqref{tmn} and \eqref{mass},
\begin{align}
\E \fm_{n+1} W&
=\E \fm_n W+
 \E{\RYni W}
\le
\E \fm_n W+ C\E\tfm_n W
\notag\\&
=\Bigpar{1+\frac{C}{\fm_0(E)+n}}\E \fm_n W
.\end{align}
Hence, the first part of \eqref{lmwqq} 
follows by induction,
since $\fm_0W\leq \fm_0 V<\infty$  by \ref{as:BWiii} (recall that $W =
V^{\nicefrac1q}$, $q>2$, and $V\ge1$).

Consequently, for every $n$, 
\as,  $\fm_n W<\infty$ 
and thus $\fm_n\in\mathcal M(W)$,
recalling that $\fm_n$ is a positive measure.
Hence, also
$\tfm_n\in\mathcal M(W)$ and, by \eqref{mn}, $\RYni\in\mathcal M(W)$.

Since $\RR$
acts on $\mathcal M(W)$ as noted above, we further obtain
$\tfm_n\RR\in\mathcal M(W)$ and thus
\eqref{a2} yields $\gD M_{n+1}\in\mathcal M(W)$ a.s.

Finally, 
\eqref{vn} implies that  $|\fv_n|\le\tfm_n+\nu$
and
\ref{as:nu2} shows that
$\nu\in\cM(V)\subseteq\mathcal M(W)$;
hence, $\fv_n W\in \mathcal M(W)$ \as{} and $\E \bigsqpar{|\fv_n| W} <\infty$ 
follow from the results for $\fm_n$ and $\tfm_n$.

\medskip
{\bf We now prove~\eqref{lmv4a}.}
Recall that, by Assumption \ref{as:BW}, $\RR V \le \vartheta V + \CCHi$.
Taking expectations in \eqref{tm1}, since $\E\gD M_{n+1}=0$, we obtain
\begin{align}\label{annika}
  \E \tfm_{n+1} = \E \tfm_n + \gam_n\E\tfm_n(\RR-\II)
= (1-\gam_n)\E \tfm_n + \gam_n\E\tfm_n\RR
\end{align}
and thus 
\begin{align}\label{bertil}
  \E \tfm_{n+1}V &
= (1-\gam_n)\E \tfm_nV + \gam_n\E\tfm_n\RR V
\le  (1-\gam_n)\E \tfm_nV + \gam_n\E\tfm_n(\vartheta V + \CCHi)
\notag\\&
=
 (1-\gam_n+\gam_n\vartheta)\E \tfm_nV + \gam_n \CCHi.
\end{align}
Recall that $\tfm_0 V<\infty$  by \ref{as:BWiii}.
 Let $C_0:= \tfm_0 V\vee \frac{\CCHi}{1-\vartheta}$.
Then $\E \tfm_n V \le C_0$ by induction; indeed, the induction hypothesis and
\eqref{bertil} yield
\begin{align}\label{berth}
  \E \tfm_{n+1} V \le (1-(1-\vartheta)\gam_n) C_0 + \gam_n\CCHi
= C_0 + \gam_n(\CCHi - (1-\vartheta)C_0) \le C_0.
\end{align}
 Because $\tfm_0 V\geq 1$, we have that $C_0\leq (1+\frac{\CCHi}{1-\vartheta}) \tfm_0 V$, 
which proves the first part of \eqref{lmv4a}.
Finally, since $Y_n$ has the conditional distribution $\tfm_{n-1}$ given
$\cF_{n-1}$, this implies
\begin{align}
  \E V(Y_n) = \E [\E_{n-1} V(Y_n)] 
= \E [\tfm_{n-1} V] 
 \le C\tfm_0 V.
\end{align}

{\bf It only remains to prove~\eqref{lg0},~\eqref{lg}  and~\eqref{lgnew}.}
For \eqref{lg0} note that $\fv_0=\tfm_0-\nu$  and thus
\begin{align}
|\fv_0 g|^2&\leq 2\left(\tfm_0|g|\right)^2+2(\nu|g|)^2\leq 2\left((\tfm_0 W)^2+(\nu W)^2\right) \normbwqq{g}^2
\notag\\
&\leq C (\tfm_0 W)^2 \normbwqq{g}^2
 \le  C\tfm_0 V\,\normbwqq{g}^2
,\end{align}
where we used the fact that $(\tfm_0 W)^2\geq 1$ and  
$(\nu W)^2\leq \nu(W^2)\leq \nu V<+\infty$ by~\ref{as:nu2},
 and similarly
$(\tfm_0 W)^2\leq \tfm_0(W^2)\leq \tfm_0 V$. 

For $i\ge1$,
by the definition \eqref{a2},
  \begin{align}
  \label{lg1a}
\E_{i-1}\bigabs{\gD M_i g}^q
&=\E_{i-1}\bigabs{\RYx{i} g -\E_{i-1}\RYx{i}g}^q
\le
C\,\E_{i-1}\bigabs{\RYx{i} g}^q
\leq C \normbwqq{g}^q \E_{i-1}\bigpar{|\RYx{i}| W}^q
.\end{align}
Furthermore, arguing as in \eqref{palma}--\eqref{palmc}, 
now using \eqref{hg}, gives
\begin{align}\label{lg1d}
  \E_{i-1}\bigpar{|\RYx{i}| W}^q
\le 
C_2\int_E V(x) \dd \tfm_{i-1}(x) = 
C_2\tfm_{i-1} V.
\end{align}
%
Therefore, \eqref{lgnew} follows by \eqref{lg1a} and \eqref{lg1d}.
By taking the expectation in \eqref{lgnew} and 
using \eqref{lmv4a}, we obtain
\begin{align}\label{lg2}
  \E\bigabs{\gD M_i g}^q \le C\normbwqq{g}^q \,\tfm_0 V.
\end{align}
Since $q>2$, \eqref{lg} follows from \eqref{lg2} by 
Jensen's inequality
 and since $\tfm_0V\ge1$. 
\end{proof}

Lemma~\ref{lem:algo_sto} implies 
that if $f\in B(W)$, then
(with all terms \as{} finite and integrable by \refL{Lmwqq})
\begin{align}\label{alf}
	\fv_n f=\fv_0B_{0,\nnii}f+\sumin\gamma_{i-1}\gD M_i B_{i,\nnii}f.
\end{align}
Note that the sequence of partial sums of \eqref{alf} is a martingale, since
$\E_{i-1}\gD M_i=0$. 
For later use, note that \eqref{R1} and \eqref{nurr} imply
\begin{align}
	B_{i,n}1 &= 1, \label{B1}
	\\
	\nu B_{i,n}&=\nu.\label{nuB}
\end{align}
We write \eqref{alf} as
\begin{align}\label{ba}
	\fv_nf=\zeta_{n,0}+\sumin \zeta_{n,i},
\end{align}
where
\begin{align}\label{zeta0}
	\zeta_{n,0}&=  \zeta_{n,0}(f)
	:= \fv_0B_{0,\nnii}f, 
	\\
	\zeta_{n,i}&=  \zeta_{n,i}(f)
	:= \gam_{i-1}\gD M_iB_{i,\nnii}f, \qquad 1\le i\le n.
	\label{zetai}
\end{align}
The main part of the proof is to use the assumptions to show that
the random variables $\zeta_{n,i}$ are suitably small.
Note that 
\begin{align}\label{vn1}
\fv_n1=\tfm_n1-\nu 1=0, 
\end{align}
since both $\tfm_n$ and $\nu$ are
probability measures. Note also that \ref{as:balance} implies that 
\begin{align}\label{noa}
  \gD M_i1 =  R\nni_{Y_i}1 - \E_{i-1} R\nni_{Y_i}1 = 1-1 =0
\qquad\text{a.s.}
\end{align}
Hence \eqref{zeta0}--\eqref{zetai} and \eqref{B1} show that taking $f=1$,
we obtain $\zeta_{n,i}(1)=0$ \as{} for every $i\ge0$. Consequently, by
linearity, for any $i\ge0$ and any constant $c$,
\begin{align}\label{fc}
  \zeta_{n,i}(f) = \zeta_{n,i}(f-c)
\qquad\text{a.s.}
\end{align}

Recall (see \cite[(VII.4.5)]{Conway} or \cite[Section VII.3]{Dunford-Schwartz})
that if $T$ is a bounded operator on a complex Banach space, with spectrum
$\gs(T)$, and $h$ is a function that is analytic in a neighbourhood of
$\gs(T)$,
then $h(T)$ is the bounded operator defined by
\begin{align}\label{riesz}
	h(T):=\frac{1}{2\pi\ii}\oint_\gG h(z) (z-T)\qw\dd z,
\end{align}
integrating over a union $\gG$ of rectifiable closed curves that encircle
each component of $\gs(T)$ once in the positive direction,
such that furthermore $h$ is analytic on $\gG$ and in the interior of each
of the curves. 
For properties of the map $h\mapsto h(T)$
see \cite[Theorem VII.4.7]{Conway}.
In particular, note that if $h=h_1h_2$, with $h_1$ and $h_2$ analytic in a
neighbourhood of $\gs(T)$, then
\begin{align}
  \label{h1h2}
h(T)=h_1(T)h_2(T).
\end{align}
Furthermore,
the resolvent $z\mapsto (z-T)^{-1}$ is analytic outside $\gs(T)$
\cite[Theorem~VII.3.6]{Conway}, 
and thus $\norm{(z-T)\qw}$ is bounded on $\gG$;
hence \eqref{riesz} implies the existence of a constant $C_\gG$ (depending
also on $T$) such that
\begin{align}\label{||h||}
  \norm{h(T)} \le C_\gG\sup_{z\in\gG}|h(z)|.
\end{align}
Recall also
that
\eqref{riesz} extends the elementary
definition of $h(T)$ for polynomials $h$.
Hence,
\begin{align}\label{Bbmn}
	B_{m,n}  
	=b_{m,n}(\RR),
\end{align}
where $b_{m,n}(z)$ is the  polynomial
\begin{align}\label{bmn}
	b_{m,n}(z):=  \prod_{k=m}^{n-1} (1+\gamma_k(z-1)).
\end{align}
Moreover, for any complex $c$, the function $\ee^{cz}$ is entire so $\ee^{cT}$
can be defined by \eqref{riesz} as a bounded operator; this agrees
with the  definition using the usual power series expansion. 
In particular, if $t>0$, this defines
$t^T=\ee^{(\log t)T}$.

\begin{lemma}\label{L1}
	For each compact set $K\subset\bbC$, we have uniformly 
	for $z\in K$ and $0\le m\le n$,
	with $n\ge1$,
	\begin{align}\label{l1}
		 |b_{m,n}(z)| \le C \parfrac{\fm_0(E)+n}{\fm_0(E)+(m\vee1)}^{\!\!\Re z-1}.
	\end{align}
Furthermore, there exists a family of analytic functions
$h_{m,n}:\mathbb{C}\to\mathbb{C}$ 
defined by
	\begin{align}\label{l1o}
		b_{m,n}(z)
		&=(1+h_{m,n}(z))\,\biggparfrac{\fm_0(E)+n}{\fm_0(E)+(m\vee1)}^{z-1}\\
        \label{l1oline2}
&=(1+h_{m,n}(z))\,\exp\big[(z-1)\big(\log (\fm_0(E)+n)-\log (\fm_0(E)+(m\vee1))\big)\big]
	\end{align}
such that uniformly for $z\in K$ and $0\le m\le n$,
\begin{align}\label{l1h}
| h_{m,n}(z)| \le \frac{C}{m\vee1}.
\end{align}
\end{lemma}

\begin{proof}
	The function $h_{m,n}$ defined by 
\eqref{l1o}
is analytic, and hence it only
   remains to prove \eqref{l1h}, since then \eqref{l1} follows from \eqref{l1o}.

	Let $C_K:=\sup_{z\in K}|z-1|$. 
We may in the sequel assume $n\ge m\ge 2C_K$, and in particular that 
$m\vee1=m$.
	The result for smaller $m$ then follows from the result for $m= \ceil{2C_K}$
	because each factor in \eqref{bmn} is bounded  by $1+C_K$ on $K$.
(The case $n< \ceil{2C_K}$ is trivial.)

	For $k\ge m\ge 2C_K$ and $z\in K$, we have $\gam_k\le 1/k\le 1/(2C_K)$, and thus
	$|\gam_k(z-1)|\le 1/2$.
	Hence,
	\begin{align}
		\left|\log\bigpar{1+\gam_k(z-1)}-\gamma_k(z-1)\right|
        &\leq \gamma_k^2|z-1|^2\leq \frac{C_K^2}{k^2}\label{l1a}
	.\end{align}
    Consequently,
	\begin{align}\label{l1b}
		b_{m,n}(z)&
		=\exp\Bigpar{\sum_{k=m}^{n-1}  \log\bigpar{1+\gam_k(z-1)}}
		=\exp\Bigpar{\sum_{k=m}^{n{\cec -1}}{ \frac{z-1}{\fm_0(E)+1+k}}+ O(\nicefrac1m)}
		\notag\\&
		=\exp\big[(z-1){ \big(\log (\fm_0(E)+n) -\log(\fm_0(E)+m)\big)}+O(\nicefrac1m)\big],
	\end{align}
where the implicit constant in $O(\nicefrac1m)$ does not depend on $\fm_0$,
	and the result \eqref{l1h} follows.
\end{proof}

\begin{remark}\label{RGamma}
Alternatively, one can show \eqref{l1} and \eqref{l1h} 
using the exact formula 
	\begin{align}\label{rgamma}
		b_{m,n}(z)=\prod_{k=m}^{n-1} \frac{\fm_0(E)+k+z}{\fm_0(E)+k+1}
		=\frac{\gG(n+\fm_0(E)+z)}{\gG(n+\fm_0(E)+1)}
		\frac{\gG(m+\fm_0(E)+1)}{\gG(m+\fm_0(E)+z)}
	\end{align}
	and Stirling's formula.
\end{remark}

\begin{lemma}\label{L0}
Assume  that the conditions of Theorem~\ref{T1} hold. Then $\Pi_1 f =
    (\nu f)1$, for all $f\in D$. 
As a consequence,
\begin{align}\label{l0}
  (\II-\Pi_1)D
=\set{f\in D:\Pi_1f=0}
=\set{f\in D:\nu f=0}.
\end{align}
In particular, if $\gD$ is a clopen subset of $\gs(\RR_D)\setminus\set1$,
and $f\in\Pi_\gD D$, then $\nu f=0$.
\end{lemma}

\begin{proof}
Define $\Pi f:=(\nu f)1$. Then 
$\Pi$ is a bounded operator in $D\subseteq\BW$ because 
$\nu W<\infty$ by \ref{as:nu2} and $1\in D$. 
Furthermore, $\Pi$ is a projection in $D$ (since $\nu1=1$),
	and
	\eqref{R1} and \eqref{nurr} imply that $\RR\Pi =\Pi=\Pi \RR$. 
	Thus $\Pi$ commutes with $\RR$, and therefore with $\Pi_1$
    (see~\cite[Proposition VII.4.9]{Conway}). 
	Furthermore, $\Pi$ and $\Pi_1$ are both projections with rank~1, 
	and the eigenfunction~1 belongs to both their ranges. 
	Hence $\Pi$ and $\Pi_1$ are both projections onto the subspace of constant functions.
	We thus get that, for any $f\in D$,
\begin{align}\label{pippi}
	\Pi_1 f = \Pi \,\Pi_1 f = \Pi_1 \Pi f = \Pi f,
\end{align}
as stated.
The equalities \eqref{l0} follow.
Finally, if $1\notin\gD$ and $f\in\Pi_\gD D$, then
$\Pi_1 f = \Pi_1\Pi_\gD f=\Pi_{\set1\cap\gD}f = \Pi_\emptyset f =0$,
see \eg{} \cite[Corollary VII.3.21]{Dunford-Schwartz},
and thus $\nu f=0$.
\end{proof}

\subsection{Proof of \eqref{eqT1first} and \eqref{eqT1sec} of Theorem~\ref{T1}}
We prove first some lemmata.  
\begin{lemma}  \label{L2bis}
	Assume  that the conditions of Theorem~\ref{T1} hold.
    \begin{thmenumerate}
    \item \label{L2a}
 For every $\delta\in(0,1-\gthd)$,  there exists a constant $C_\delta$ 	such that for every $f\in D$ with $\nu f=0$, $n\ge1$ and $0\le m\le n$, 
	\begin{align}\label{l2bis}
		\bignorm{B_{m,n}f}_{B(W)}\leq C_\delta \parfrac{\fm_0(E)+(m\vee1)}{\fm_0(E)+n}^{\!\!\delta}\norm{f}_{B(W)}.
	\end{align}
  \item \label{L2b}
If\/ $f\in\BW$ and  
$(\RR-\lambda\II)^\kappa f=0$ for some $\lambda\in \mathbb C$  and
$\kappa\geq 1$, then 
for $n\geq 1$ and $0\le m\le n$,
	\begin{align}\label{l2ter}
	\bignorm{B_{m,n}f}_{B(W)} \le
	C_{\lambda, \kappa} \parfrac{\fm_0(E)+(m\vee1)}{\fm_0(E)+n}^{\!\!1-\Re(\lambda)}
\left[1+\left(\log\left(\frac{\fm_0(E)+n}{\fm_0(E)+(m\vee1)}\right)\right)^{\!\!\kappa-1}\right]\,\norm{f}_{B(W)},
\end{align}
for some constant $C_{\lambda, \kappa}$ 
not depending on $f$.
    \end{thmenumerate}
\end{lemma}

\begin{proof} 
\pfitemref{L2a}
Note that $f\in(\II-\Pi_1)D$ by \refL{L0}.
We let $\RR'$ denote the restriction of $\RR$ (or $\RR_D$)
to $(\II-\Pi_1)D$; 
$\RR'$ is a bounded operator 
with spectrum $\gs(\RR')=\gs(\RR_D)\setminus\set1$ 
(see for instance~\cite[Theorem VII.3.20]{Dunford-Schwartz}%
). 

Fix $\delta\in(0,1-\gthd)$.
Then 
$\sup_{z\in \gs(\RR')}\Re(z)=\gthd< 1-\delta$,
 and thus we can find a rectifiable curve $\Gamma$  in $\mathbb{C}$
that encircles  
$\gs(\RR')$ such that $\sup_{z\in \Gamma}\Re(z)\leq 1-\delta$.
Consequently, by \eqref{Bbmn} and \eqref{riesz},
\begin{align}\label{hong}
  B_{m,n}f
=b_{m,n}(\RR)f
=b_{m,n}(\RR')f
=\frac{1}{2\pi\ii}\oint_\gG b_{m,n}(z) (z-\RR')\qw f\dd z.
\end{align}
Furthermore, 
\refL{L1} implies that for $z\in\gG$, 
\begin{align}\label{kong}
  |b_{m,n}(z)| 
\le C \parfrac{\fm_0(E)+n}{\fm_0(E)+(m\vee1)}^{\!\!\Re z-1}
\le C \parfrac{\fm_0(E)+n}{\fm_0(E)+(m\vee1)}^{\!\!-\gd}
= C \parfrac{\fm_0(E)+(m\vee1)}{\fm_0(E)+n}^{\!\!\gd}.
\end{align}
The result \eqref{l2bis} follows from \eqref{hong} and \eqref{kong},
see \eqref{||h||}.
\pfitemref{L2b}
We use the factorization \eqref{l1o} and \eqref{h1h2}.
Thus,
\begin{align}\label{mamba}
  B_{m,n}f=b_{m,n}(\RR)f
=\bigpar{\II+h_{m,n}(\RR)}\parfrac{\fm_0(E)+n}{\fm_0(E)+(m\vee1)}^{\!\!\RR-\II}f.
\end{align}
Furthermore, by \eqref{l1h}, 
the functions $h_{m,n}$, 
 for $n\ge1$ and $0\le m\le n$,
are uniformly bounded on any fixed
compact subset of $\bbC$, and thus \eqref{||h||} implies that the operators
$h_{m,n}(\RR)$ are uniformly bounded on $\BW$  
by a constant that does not depend on $\fm_0$.

Moreover, for all functions~$f\in D$ such that 
$(\RR-\lambda \II)^\kappa f=0$, for $m\ge1$,
\begin{align}\label{boa}
\parfrac{\fm_0(E)+n}{\fm_0(E)+m}^{\!\!\RR-\II}f
&=\parfrac{\fm_0(E)+n}{\fm_0(E)+m}^{\!\!\gl-1}\exp\bigg(\log\parfrac{\fm_0(E)+n}{\fm_0(E)+m}\,(\RR-\gl \II)\bigg)f
\notag\\&
=\left(\frac{\fm_0(E)+m}{\fm_0(E)+n}\right)^{\!\!1-\lambda}
\sum_{k=0}^{\kappa-1} 
\log^k \parfrac{\fm_0(E)+n}{\fm_0(E)+m}\frac{(\RR-\lambda\II)^kf}{k!}.
	\end{align}
The result follows by \eqref{mamba} and \eqref{boa}
 since the operators $(\RR-\lambda\II)^k$ are bounded.
\end{proof}

Recall from Section~\ref{S:not} that we use $\frak C$ for constants that may depend on $\frak m_0$.
\begin{lemma}
	\label{lem:prelem}
Assume  that the conditions of Theorem~\ref{T1} hold. 
\begin{thmenumerate}
\item \label{LPLa}  
For every $\delta\in(0,1-\gthd)$,  there exists a constant
$\frak C_\delta<\infty$
such that for every $f\in D$, 
	\begin{align}
		|\zeta_{n,0}| &\leq \frak C_\delta (\nicefrac{1}{n})^{\delta}\,\norm{f}_{B(W)},\qquad n\geq 1,\label{eq:LP-W-0bis}\\
		\left(\E_{i-1}|\zeta_{n,i}|^q\right)^{\nicefrac1q} & \leq \frac{\frak C_\delta}{i} (\nicefrac{i}{n})^{\delta}\, \,\norm{f}_{B(W)} \,\left(\tfm_{i-1}(V)\right)^{\nicefrac1q},\qquad n\geq i\geq 1.\label{eq:LP-W-1bis}
	\end{align}
\item \label{LPLb}
If\/ $\lambda\in  \mathbb C$  and $\kappa\geq 1$,  there exists a  constant $\frak C_{\lambda,\kk}$
such that
if\/ $f\in\BW$ and $(\RR-\lambda\II)^\kappa f=0$, then
		\begin{align}
	|\zeta_{n,0}| &
\leq \frak C_{\lambda,\kk} (\nicefrac{1}{n})^{1-\Re(\lambda)}(\log n)^{\kappa-1}\,
\norm{f}_{B(W)},\qquad  n\geq 2,\label{eq:LP-W-0ter}\\
 \left(\E_{i-1}|\zeta_{n,i}|^q\right)^{\nicefrac1q} &
\leq \frac{\frak C_{\lambda,\kk}}{i} (\nicefrac{i}{n})^{1-\Re(\lambda)}
\bigsqpar{1+(\log(\nicefrac{n}{i}))^{\kappa-1}}\,\norm{f}_{B(W)} \,
\tfm_{i-1}(V)^{\nicefrac1q},\qquad n\geq i\geq 1.
\label{eq:LP-W-1ter}
	\end{align}
\end{thmenumerate}
\end{lemma}

\begin{proof}
By homogeneity, we may without loss of generality  assume 
$\norm{f}_{B(W)}=1$. 
Furthermore, by \eqref{fc} we may replace $f$ by 
$f-\nu f$;
hence we may also assume $\nu f=0$.

\pfitemref{LPLa}
Fix $\delta\in(0,1-\gthd)$.
	 According to  Lemma~\ref{L2bis}, we have  for all 
$n\ge1$ and $0\le i\le n$,
	\begin{align}\label{selma}
	|B_{i,\nnii} f|\leq \frak C'_\delta  \parfrac{i\vee1}{n}^{\!\!\delta} W.
	\end{align}
First, taking $i=0$, we obtain \eqref{eq:LP-W-0bis} from  \eqref{zeta0} and
\eqref{selma}, since $\fv_0\in\cM(W)$ by \refL{Lmwqq}.

For $n\ge i\ge 1$, we have
by \eqref{zetai},
Lemma~\ref{Lmwqq} (Equation~\eqref{lgnew}) and \eqref{selma}, 
\begin{align}\label{gosta}
     \E_{i-1}|\zeta_{n,i}|^q&=\gamma_{i-1}^q \E_{i-1}|\Delta M_i B_{i,n}f|^q
\leq \frac{C}{i^q}\,\|B_{i,n}f\|_{B(W)}^q\,\tfm_{i-1}(V)
\leq \frac{\frak C''_\delta}{i^q}\left(\frac{i}{n}\right)^{\delta q}\,\tfm_{i-1}(V)
.\end{align}
This concludes the proof of~\eqref{eq:LP-W-1bis}.

\pfitemref{LPLb}	
	The same arguments but using~\eqref{l2ter} instead of~\eqref{l2bis} lead
    to~\eqref{eq:LP-W-0ter} and~\eqref{eq:LP-W-1ter}. 
\end{proof}

For technical reasons, we have stated \refT{T1} for an invariant subspace
$D$ containing the constant functions; these are eigenfunctions with
eigenvalue 1 by \eqref{R1}, and thus $1\in\gs(\RR_D)$.
It will now be convenient to consider also invariant subspaces not
containing constants; we then use the generic notation $\Dx$ to help the reader
distinguish the assumptions.

\begin{lemma}  \label{LK1a}
Suppose that $\Dx$ is an $\RR$-invariant subspace of $\BW$ and that
$\gthh\in\bbR$ is such that 
$\sup\Re\gs\xpar{\RR_{\Dx}}<\gthh$.
Then,
 \begin{align}\label{lk1am0rand}
 \E |\fv_n f|^2\le C \,\tfm_0 V\left(\frac{\fm_0(E)+1}{\fm_0(E)+n}\right)^{\!\!1\bmin2(1-\gthh)}\normbwqq{f}^2,
 \qquad f\in \Dx,
 \; n\ge1.
 \end{align}
\end{lemma}

\begin{proof}
The terms in \eqref{alf} are orthogonal, and thus, using \eqref{lg0}
and~\eqref{lg}
in  Lemma~\ref{Lmwqq},
 \begin{align}\label{lkab}
  \E|\fv_n f|^2
&=
\E\left[|\fv_0 B_{0,n}f|^2\right] + \sumin \gam_{i-1}^2 \E \big[\bigabs{\gD M_i B_{i,n}f}^2\big]
\notag\\&
\le C   \tfm_0 V\normbwqq{B_{0,n}f}^2 + C\sumin (\fm_0(E)+i)^{-2}\,\tfm_0 V\,\normbwqq{B_{i,n}f}^2.
\end{align}

We apply \eqref{||h||} to $B_{m,n}=b_{m,n}(\RR)$ as an operator on $\Dx$.
By the assumption,
we may choose a curve $\gG$ encircling $\gs(\RR_{\Dx})$
such that $\Re z\le\gthh$ for $z\in\gG$, and then \eqref{||h||} and
\eqref{l1} yield, uniformly for $0\le m\le n$,
 \begin{align}\label{lkbb}
  \norm{B_{m,n}}_{\Dx} \le C \Bigparfrac{\fm_0(E)+n}{\fm_0(E)+(m\vee1)}^{\!\!\gthh-1}.
\end{align}
By homogeneity, we may assume $\normbwqq{f}=1$, and then
\eqref{lkab} and \eqref{lkbb} yield
 \begin{align}
 \frac{ \E|\fv_n f|^2}{\tfm_0 V}
&
\le C \left(\frac{\fm_0(E)+n}{\fm_0(E)+1}\right)^{\!\!2(\gthh-1)} + C\sumin (\fm_0(E)+i)^{-2} \left(\frac{\fm_0(E)+n}{\fm_0(E)+i}\right)^{\!\!2(\gthh-1)}
\notag\\
&\le C \left(\frac{\fm_0(E)+n}{\fm_0(E)+1}\right)^{\!\!2\gthh-2}+C\left(\fm_0(E)+n\right)^{2\gthh-2}\sumin \left(\fm_0(E)+i\right)^{-2\gthh} 
\notag\\
&\le C \left(\frac{\fm_0(E)+n}{\fm_0(E)+1}\right)^{\!\!2\gthh-2}+ C\begin{cases}
 (\fm_0(E)+n)^{-1} &\text{ if }\gthh <\half\\
 (\fm_0(E)+n)^{-1} \log (\fm_0(E)+n) &\text{ if }\gthh = \half\\
 (\fm_0(E)+n)^{2\gthh-2} &  \text{ if }\gthh > \half.
\end{cases}\label{lkac}
\end{align}
This yields \eqref{lk1am0rand} when $\gthh\neq\half$.
If $\gthh=\half$, then 
one can replace $\gthh$ by some new $\gthh<\half$; then \eqref{lkac}
yields~\eqref{lk1am0rand} in this case too. 
\end{proof}

The estimates~\eqref{eqT1first} and \eqref{eqT1sec} of Theorem~\ref{T1} 
directly follow from the following result.
Recall that $\fv_n:=\tfm_n-\nu$.	
\begin{lemma}\label{LP-Wbis}
		Assume  that the conditions of Theorem~\ref{T1} hold.
        \begin{thmenumerate}
          
        \item \label{LP-Wbisa}
For every $\delta\in(0,{  1-\gthd})$, there exists a constant
$C_\delta<\infty$	such that,
	for any $f\in D$ and any $n\geq 1$,
 	\begin{align}
		\label{eq:LP-W-2bis}
		\E |\fv_nf|^2 \le C_\delta\, \tfm_0 V\left(\frac{\fm_0(E)+1}{\fm_0(E)+n}\right)^{\!\!2\delta\wedge 1}\normbwqq{f}^2.
	\end{align}

        \item \label{LP-Wbisb}
If, furthermore, $\RR$ is \slqc{} on $B(W)$,
then, for every $\delta\in(0,{  1-\theta})$,
 there exists a constant
$C_\delta<\infty$ such that, for all $f\in B(W^2)$,
 	\begin{align}
		\label{eq:LP-W-3bis}
		\E\left|\fv_n f\right|
&\leq C_\delta \,\tfm_0 V\,\left(\frac{\fm_0(E)+1}{\fm_0(E)+n}\right)^{\!\!({2\delta}\wedge 1)\frac{q/2-1}{q-1}}\,\|f\|_{B(W^2)}.
	\end{align}
 \end{thmenumerate}
\end{lemma}

\begin{proof}
\pfitemref{LP-Wbisa}
This is essentially equivalent to \refL{LK1a}.
Recalling \refL{L0}, we define
\begin{align}
\Dx:=(1-\Pi_1)D=\set{f\in D:\nu f=0}.
\end{align}
Then, as in the proof of \refL{L2bis}, 
$\Dx$ is an invariant subspace of $\BW$ and
$\RR_\Dx$ has spectrum $\gs(\RR_\Dx)=\gs(\RR_D)\setminus\set1$, and thus
$\sup\Re\gs(\RR_\Dx)=\gthd$.
We define $\gthh:=1-{  \gd}$, and note that the assumption implies $\gthh>\gthd$.
Hence, \eqref{lk1am0rand} applies and yields \eqref{eq:LP-W-2bis} for $f\in \Dx$.
Finally, for a general $f\in D$, we apply instead \eqref{lk1am0rand} to 
$f-(\nu f)1=(1-\Pi_1)f\in\Dx$ (recalling \refL{L0}), noting that $\fv_n1=0$
by \eqref{vn1}.

\pfitemref{LP-Wbisb}
	We now assume that the operator $\RR$ is \slqc, 
so we may take $D=\BW$ in \ref{LP-Wbisa}.
For an arbitrary $f\in B(W^2)$, we will use truncations:
For all $K\geq 1$,
	\begin{equation}\label{eq:K}
%
	\E\left|\fv_n f\right|
	\leq \bigl|\E[\fv_n (f\bs 1_{W^2\leq K})]\bigr|+\E[\tfm_n |f\bs 1_{W^2> K}|]+\nu|f\bs 1_{W^2> K}|.
	\end{equation}
First, since $|f(x)|\bs 1_{W(x)^2\leq K}\leq \|f\|_{B(W^2)}
W(x)\,\sqrt{K}$, we deduce that $\|f\bs 1_{W^2\leq
  K}\|_{B(W)}\leq \|f\|_{B(W^2)}\,\sqrt{K}$. Therefore, by
~\eqref{eq:LP-W-2bis} applied to $f\bs 1_{W^2\leq K}\in
B(W)$, we get, for any fixed $\delta\in(0,{  1-\theta})$ 
(the constants $C$ below 
do not depend on $f$ or $K$), 
	\begin{equation}\label{eq:Kbound1}
	\E\bigabs{\fv_n (f\bs 1_{W^2\leq K})}
\le
{\E[\abs{\fv_n (f\bs 1_{W^2\leq K})}^2]}^\half
\leq C \, \tfm_0 V \left(\frac{\fm_0(E)+1}{\fm_0(E)+n}\right)^{\!\!\frac{2\delta\wedge 1}{2}} \|f\|_{B(W^2)}\,\sqrt{K}.
	\end{equation}
On the other hand, 
if $W^2>K$ then $VW^{-2}=W^{q-2}>K^{\nicefrac q2-1}$
and thus
	\begin{align}
		\E[\tfm_n |f\bs 1_{W^2> K}|]&
\leq \E[\tfm_n\xpar{ W^2 \bs 1_{W^2> K}}]\,\|f\|_{B(W^2)}
		\leq K^{1-\nicefrac q2}\E[\tfm_n V]\,\|f\|_{B(W^2)}
\notag\\&
		\leq C K^{1-\nicefrac q2}\,{\|f\|_{B(W^2)}}\, \tfm_0 V,\label{eq:Kbound2}
	\end{align}
	where we used \eqref{lmv4a}.
	The same computation also holds for $\nu|f\bs 1_{W^2> K}|$, 
	since, by assumption~\ref{as:nu2}, $\nu V<\infty$.

	Finally, choosing  $K=\left(({\fm_0(E)}+1)/({\fm_0(E)+n})\right)^{-({  2\delta}\wedge 1)/(q-1)}$ and using
    \eqref{eq:Kbound1} and~\eqref{eq:Kbound2}, we deduce 
\eqref{eq:LP-W-3bis}.
\end{proof}

\subsection{Proof of~\eqref{eqT1firstbis}}
We improve the estimate in \refL{LP-Wbis}\ref{LP-Wbisa} to an estimate for a maximum over a restricted range.

\begin{lemma}
  \label{LK1b}
Suppose that $\Dx$ is an $\RR$-invariant subspace of $\BW$ and that
$\gthh>\half$ is such that $\sup\Re\gs\xpar{\RR_{\Dx}}<\gthh$.
Moreover, 
let $0<\tau<1$. Then there exists a constant $C=C(\tau)$ such that
 \begin{align}\label{lk1b}
  \E \sup_{N-N^\tau\le n\le N}|\fv_n f|^2\le C  \tfm_0 V\,\left(\frac{\fm_0(E)+1}{\fm_0(E)+N}\right)^{\!\!2(1-\gthh)+\tau}\,N^\tau\,\normbwqq{f}^2,
\qquad f\in \Dx,
\; N\ge1.
\end{align}
\end{lemma}

\begin{proof}
If $0\le i\le n\le N$, then the definition \eqref{Bmn} implies
\begin{align}\label{BNN}
B_{i,n}B_{n,N}=B_{i,N}.   
\end{align}

Let $f\in\Dx$.
For $0\le n\le N$, we apply \eqref{alf} to $B_{n,N}f$ and obtain, 
using \eqref{BNN},
\begin{align}\label{hq}
  	\fv_n B_{n,N}f=\fv_0B_{0,N}f+\sumin\gamma_{i-1}\gD M_i B_{i,N}f.
\end{align}
Since $\E_{i-1}\gD M_i=0$, \eqref{hq} shows that 
$\bigpar{ \fv_nB_{n,N} f}_{0\le n\le N}$ is a martingale 
for each fixed $N$.
(Note that the terms on the \rhs{} do not depend on $n$.)
Consequently, Doob's maximal inequality yields, together with $B_{N,N}=\II$ and
\refL{LK1a},
\begin{align}
  \E\bigabs{\sup_{n\le N}\fv_n B_{n,N} f}^2
&\le 4   \E\bigabs{\fv_N B_{N,N} f}^2
=4   \E\bigabs{\fv_N  f}^2
\notag\\
& 
 \leq C \,\tfm_0 V\,\left(\frac{\fm_0(E)+1}{\fm_0(E)+N}\right)^{\!\!2(1-\gthh)}\normbwqq{f}^2\label{lkd}
.\end{align}

Let $K$ be a compact neighbourhood of $\gs(\RR_\Dx)$ such that $\sup\Re K
<\gthh$.
Let $n=N-m$, where $0\le m\le N^\tau$, and suppose that $N$ is so large that
$N^\tau\le N/2$. Also, let $L:=\floor{1/(1-\tau)}$.
Then, $m\le N/2$ and thus
\begin{align}
  \log\parfrac{\fm_0(E)+N}{\fm_0(E)+n}
&=\bigg|\log\parfrac{\fm_0(E)+n}{\fm_0(E)+N}\bigg|
=\bigg|\log\parfrac{\fm_0(E)+N-m}{\fm_0(E)+N}\bigg|
\notag\\
&=\Bigabs{\log\Bigpar{1-\frac{m}{\fm_0(E)+N}}}
\le \frac{2m}{\fm_0(E)+N}
\le 2 (\fm_0(E)+N)^{\tau-1}.\label{cp}
\end{align}
Since $n=N-m\ge N/2$, \eqref{l1h} yields, for all $z\in K$,
\begin{align}\label{da}
  |h_{n,N}(z)| \le \frac{C}{n}  \le \frac{C}{N}.
\end{align}
Assume in the sequel that $N$ and $n\le N$ are as above, and also that $N$
is so large that \eqref{da} implies $|h_{n,N}(z)|\le\half$
when $z\in K$.
Then, \eqref{l1oline2}--\eqref{l1h} imply that, uniformly for 
$z\in K$ and all such $n$ and $N$, 
\begin{align}\label{cq}
\frac1{b_{n,N}(z)}&=
\Bigpar{1+ O\Bigparfrac{1}{N}}\exp\bigg(-(z-1)\log\parfrac{\fm_0(E)+N}{\fm_0(E)+n}\bigg)
\notag\\
&=
\exp\bigg(-(z-1)\log \parfrac{\fm_0(E)+N}{\fm_0(E)+n}\bigg)
+ O\Bigparfrac{1}{N}
\notag\\&
=\sumlL \frac{1}{\ell!}\log^\ell\Bigparfrac{\fm_0(E)+N}{\fm_0(E)+n}(1-z)^\ell  
+ O\Bigpar{\log^{L+1}\Bigparfrac{\fm_0(E)+N}{\fm_0(E)+n}}+O\Bigparfrac1N.
\notag\\&
=\sumlL \frac{1}{\ell!}\log^\ell\Bigparfrac{\fm_0(E)+N}{\fm_0(E)+n}(1-z)^\ell  
+O\Bigparfrac1N,
\end{align}
where the last equality uses \eqref{cp} and $(L+1)(1-\tau)>1$.
In particular, $b_{n,N}\qw(z)$ is finite for $z\in K$, so $b_{n,N}\qw(z)$
is analytic in a neighbourhood of $\gs(\RR_\Dx)$; hence $B_{n,N}$ is
invertible on $\Dx$, with $B_{n,N}\qw=b_{n,N}\qw(\RR)$.
Define the operator on $\Dx$
 \begin{align}\label{cr}
V_{n,N}:=B_{n,N}\qw-\sumlL \frac{1}{\ell!}\log^\ell\Bigparfrac{\fm_0(E)+N}{\fm_0(E)+n}(\II-\RR)^\ell
.\end{align}
It follows from \eqref{cq} and \eqref{||h||} that
\begin{align}
\label{crv}
\norm{V_{n,N}}_{\Dx}=O\left(\frac1N\right).  
\end{align}
Moreover, \eqref{cr} yields, with 
$f_\ell:=(\II-\RR)^\ell f/\ell!$ and
$g_{n,N}:=V_{n,N}f$,
 \begin{align}\label{cs}
  B_{n,N}\qw f = 
\sumlL \log^\ell\Bigparfrac{\fm_0(E)+N}{\fm_0(E)+n}f_\ell + g_{n,N},
\end{align}
and thus 
 \begin{align}\label{ct}
  \fv_n f = \fv_n B_{n,N} B_{n,N}\qw f
=\sumlL \log^\ell\Bigparfrac{\fm_0(E)+N}{\fm_0(E)+n} \fv_n B_{n,N}f_\ell + \fv_n B_{n,N}g_{n,N}.
\end{align}
According to~\eqref{cp}, we have 
 $0\le\log\bigparfrac{\fm_0(E)+N}{\fm_0(E)+n}\leq 2 N^{\tau-1}\leq 2$ and hence,
with constants $C$ depending on $\tau$ in the remainder of
the proof,
\begin{align}\label{cu}
\sup_{N-N^\tau\le n \le N}|\fv_n f|^2 
&\le
C\sumlL \sup_{N-N^\tau\le n \le N}\bigabs{ \fv_n B_{n,N}f_\ell }^2
+C\sup_{N-N^\tau\le n \le N}\bigabs{ \fv_n B_{n,N}g_{n,N}}^2
\notag\\&
\le
C\sumlL \sup_{ n \le N}\bigabs{ \fv_n B_{n,N}f_\ell }^2
+C\sum_{N-N^\tau\le n \le N}\bigabs{ \fv_n B_{n,N}g_{n,N}}^2
.\end{align}
Furthermore, for $n$ and $N$ as above,
\eqref{crv} holds;
hence
\begin{align}\label{cv}
  \normbwqq{g_{n,N}}
=
  \normbwqq{V_{n,N}f}
\le C N\qw \normbwqq{f}
 \le C \frac{\fm_0(E)+1}{\fm_0(E)+N} \normbwqq{f}.
\end{align}
Taking the expectation in \eqref{cu} 
and using \eqref{lkd} and \eqref{cv} yields
 \begin{align}\label{cw}
&\E\sup_{N-N^\tau\le n \le N}|\fv_n f|^2 
\notag\\
&\le
C \sumlL \tfm_0 V\,\left(\frac{\fm_0(E)+1}{\fm_0(E)+N}\right)^{\!\!2(1-\gthh)} \normbwqq{f_\ell}^2
+C \!\!\!\sum_{N-N^\tau\le n \le N} \!\!\!\tfm_0 V\,\left(\frac{\fm_0(E)+1}{\fm_0(E)+N}\right)^{\!\!2(1-\gthh)}\normbwqq{g_{n,N}}^2
\notag\\
&\le
C \tfm_0 V\,\left(\frac{\fm_0(E)+1}{\fm_0(E)+N}\right)^{\!\!2(1-\gthh)}\normbwqq{f}^2
+ C \tfm_0 V\,\left(\frac{\fm_0(E)+1}{\fm_0(E)+N}\right)^{\!\!2(1-\gthh)+2}N^\tau\normbwqq{f}^2
.\end{align}
This shows \eqref{lk1b} when $N$ is large enough since $\tau<1<2$.
The remaining cases are trivial, since \eqref{lk1b} for any fixed $N$
follows from \refL{LK1a}.
\end{proof}

We are now ready to prove~\eqref{eqT1firstbis} and thus conclude the proof
of Theorem~\ref{T1}: 
\begin{lemma}
  \label{LK2}
  \begin{thmenumerate}
  \item\label{LK2a}
    Suppose that $\Dx$ is an $\RR$-invariant subspace of $\BW$ and that
$\gthh>\half$ is such that $\sup\Re\gs\bigpar{\RR_{\Dx}}<\gthh$.
Then, for every $f\in \Dx$, \as{} and in $L^2$ as \ntoo,
\begin{align}
\label{eq:extendedform0random}
n^{1-\gthh}\fv_n f \to 0
.\end{align}
\item \label{LK2b}
Assume that the conditions of \refT{T1} holds, 
and let $\gd\in(0,{  (1-\gthd)\land\nicefrac12})$.
Then 
$n^{  \gd}\fv_n f \to 0$ \as{} and in $L^2$ as \ntoo,
for every $f\in D$.
  \end{thmenumerate}
\end{lemma}
\begin{proof}
\pfitemref{LK2a}
Let $\gth:=\sup\Re\gs(\RR_\Dx)$
and choose
$\gthhh\in(\gth\vee\half,\gthh)$.
Then \refL{LK1a} applied with $\gthhh$ yields,
for any $f\in \Dx$,
 \begin{align}
  \E \bigg|\left(\frac{\fm_0(E)+n}{\fm_0(E)+1}\right)^{\!\!1-\gthh}\fv_n f\bigg|^2
&\le C \tfm_0 V\, \left(\frac{\fm_0(E)+n}{\fm_0(E)+1}\right)^{\!\!2(1-\gthh)+2(\gthhh-1)}\normbwqq{f}^2
\notag\\
&=  C  \tfm_0 V\,\left(\frac{\fm_0(E)+n}{\fm_0(E)+1}\right)^{\!\!2(\gthhh-\gthh)}\normbwqq{f}^2
=o(1).\label{eq:m0depend}
\end{align}
This implies the convergence \eqref{eq:extendedform0random} in $L^2$.

To show the convergence \as,
choose $\tau\in(0,1)$ with $\tau>1+\gthhh-\gthh$.
Define an increasing sequence $(n_k)$ by $n_0:=1$ and
$n_{k+1}:=n_k+\floor{n_k^{\tau}}$. 
Then \refL{LK1b} applied with $\gthhh$ yields,
for every $k\ge1$, 
 (here $\frak C$ are constants that may depend on $\fm_0$ and $f$)
\begin{align}\label{cx}
\E \sup_{n_{k-1}< n \le n_k}|n^{1-\gthh}\fv_n f|^2 
&\le
C n_k^{2-2\gthh}\E \sup_{n_{k}-n_k^\tau\le n \le n_k}|\fv_n f|^2
\notag\\
&\le
\frak C n_k^{2-2\gthh+2(\gthhh-1)}
=
\frak C n_k^{2\gthhh-2\gthh}
\notag\\&
\le \frak C \sum_{n=n_{k-1}+1}^{n_{k}} n^{2\gthhh-2\gthh-\tau}
.\end{align}
The exponent in the final sum is 
\begin{align}
2\gthhh-2\gthh-\tau < 2\gthhh-2\gthh -  (1+\gthhh-\gthh) = \gthhh-\gthh-1 <-1.
\end{align}
Consequently,
\begin{align}
\E\sumk \sup_{n_{k-1}< n \le n_k}|n^{1-\gthh}\fv_n f|^2 
\le \frak C \sumk  \sum_{n=n_{k-1}+1}^{n_{k}} n^{2\gthhh-2\gthh-\tau}
=
 \frak C \sum_{n=2}^\infty n^{2\gthhh-2\gthh-\tau}<\infty.
\end{align}
Hence, \as, 
\begin{align}
\sumk \sup_{n_{k-1}< n \le n_k}|n^{1-\gthh}\fv_n f|^2<\infty,   
\end{align}
which implies that
$\sup_{n_{k-1}\le n \le n_k}|n^{1-\gthh}\fv_n f|^2 \to0$ as \ktoo, 
and thus
$n^{1-\gthh}\fv_n f \to 0$ as \ntoo.

\pfitemref{LK2b}
Let, as in the proof of \refL{LP-Wbis},
$\Dx:=(1-\Pi_1)D$ and apply \ref{LK2a} to $\Dx$ 
and $f-(\nu f)1\in\Dx$
with 
$\gthh:=1-{  \gd}>\gthd=\sup\Re\gs\bigpar{\RR_{\Dx}}$.
\end{proof}

\begin{remark}
    \label{rem:m0rand0}
    We observe that, if $\fm_0$ is random, then \eqref{eq:m0depend} holds
conditioned on $\fm_0$.
Hence, if 
$\E\bigsqpar{(\fm_0(E)+1)^{2(1-\gthh)} \tfm_0V}<\infty$, then,
using dominated convergence,
\begin{align}
  \E \bigabs{\bigpar{\fm_0(E)+n}^{1-\gthh}\fv_n f}^2&
\le  C
\E\Bigsqpar{(\fm_0(E)+1)^{2(1-\gthh)} \tfm_0 V\,
 \left(\frac{\fm_0(E)+n}{\fm_0(E)+1}\right)^{\!\!2(\gthhh-\gthh)}}\normbwqq{f}^2
\notag\\&
\to0 
\label{m0dep99}
.\end{align}
Hence, the
    convergence~\eqref{eq:extendedform0random} still holds in $L^2$.
\end{remark}

\section{Proof of Theorem~\ref{T2}}\label{Spf2}

\subsection{Proofs of Theorem~\ref{T2}(1) and (2)}\label{Spf212}
Recall that $\tfm_n-\nu=\fv_n$ and that $\fv_n1=0$ by \eqref{vn1}, 
which implies that $\fv_n f$ is not affected if we
subtract a constant 
from $f$. 
It is also obvious that subtracting a constant from $f$ does not affect 
$\gox(f)$ and $\gss(f)$ in \eqref{goxf}--\eqref{gssf}
and \eqref{goxf2}--\eqref{gssf2}.
Hence, replacing $f$ by $f-\nu f$, we may in the proof assume that
$\nu f=0$.
For convenience, we also assume
$\norm{f}_{B(W)}\le1$, 
as we may by homogeneity. 

We will prove (1) and (2) in parallel, since most the arguments are the same
for both cases.
Our proof relies on a
central theorem for martingales given by
Hall \& Heyde \cite[Corollary 3.1]{HH} (see \cite{HH} for other versions
and references). This theorem in \cite{HH} is stated there for real-valued
variables, but it extends immediately to vector-valued variables (in a
finite-dimensional space) 
by the Cram\'er--Wold device \cite[Theorem 5.10.5]{Gut};
in particular, the theorem holds for complex-valued variables by considering
the real and imaginary parts, and can then be stated as follows.
(In general, $\gox$ and $\gss$ may be random, but we are only interested in
the special case when they are constant.)
\begin{theorem}[{\cite[Corollary 3.1]{HH}}]\label{th:hall-heyde} 
	Let $(\zzeta_{n,i}, n\geq 0, 1\leq i\leq n)$ be a complex-valued
    martingale difference array. 
	If there exist 
$\gox\in \mathbb C$ and 
$\sigma^2\ge 0$ such that, in probability when $n\to\infty$,
	\begin{itemize}
		\item[{(a)}] 
$ \sum_{i=1}^n \mathbb E_{i-1} [|\zzeta_{n,i}|^2 \bs 1_{ |\zzeta_{n,i}|\geq \varepsilon}] \to 0$ for all $\varepsilon>0$, and
\item[{(b)}] $\sum_{i=1}^n \mathbb E_{i-1} [\zzeta_{n,i}^2] \to \gox$, and
\item[{(c)}] $ \sum_{i=1}^n \mathbb E_{i-1} [|\zzeta_{n,i}|^2] \to \sigma^2$,
	\end{itemize}
	then, in distribution when $n\to\infty$, $\sum_{i=1}^n \zzeta_{n,i}
    \Rightarrow \Lambda_1 + i\Lambda_2$ where 
the random vector $(\Lambda_1,\Lambda_2)$  has a centered
 Gaussian distribution with covariance matrix 
	\begin{align}
	\frac12\begin{pmatrix}
	\sigma^2+\Re(\gox)&\Im(\gox)\\
	\Im(\gox)&\sigma^2-\Re(\gox)
	\end{pmatrix}
	\end{align}
\end{theorem}

In 
(1), we assume that $\RR_D$ is a small
operator. In this case, 
recall from \refL{Lth} that $\gthd<1/2$.
We may thus choose 
$\delta\in(\half,1)$ such that $ \gd < 1-\gthd$;
we fix such a $\delta$ for the rest of the proof.

In 
(2), we assume that $\RR_D$ and $\RR$ are \slqc{}
operators and that the spectrum of $\RR_D$ is given by  
\begin{align}
\sigma(\RR_D)=\{1,\lambda_1,\ldots,\lambda_p\}\cup \Delta, 
\end{align} 
where 
$\Re(\lambda_1)=\cdots=\Re(\lambda_p)=1/2$,  and
$\sup\Re(\Delta) < 1/2$. 
Thus $\gthd=\half$.
Let $\gD':=\gD\cup\set1$. 
Then
\begin{align}\label{afa}
  f = \Bigpar{\Pi_{\gD'}+\sumjp\Pi_{\gl_j}}f = \Pi_{\gD'} f +\sumjp \fj.
\end{align}
Furthermore, $\RR$ is a small 
operator in $D':=\Pi_{\gD'}D$.
 Hence, according to \eqref{eqT1first} of
Theorem~\ref{T1} applied to $D'$ with  $\gd=\half$, 
\begin{align}\label{aea}
\E\Bigabs{\frac{\sqrt{n}}{(\log n)^{\kappa-\nicefrac12}}\fv_n (\Pi_{\gD'}f)}^2
\le \frak C(\log n)^{1-2\kk}
  \xrightarrow[n\to+\infty]{} 0,
\end{align}
and hence it is sufficient to prove \eqref{t12} for $f-\Pi_{\gD'} f$ instead
of $f$.  
In other words, in (2) we may assume that
\begin{align}\label{ada}
  f = f-\Pi_{\gD'} f = \sumjp \fj.
\end{align}
Note that $\norm{\fj}_{\BW}\le C$, since
each $\Pi_{\gl_j}$ is a bounded operator.

\smallskip
Returning to treating (1) and (2) together,
we use \eqref{ba}, which we now write as
\begin{align}\label{ba'}
	\ba_n \fv_nf=\sumion \ba_n \zeta_{n,i} = \sumion \zzeta_{n,i}
\end{align}
where 
\begin{align}\label{bn-def}
	\ba_n:=
	\begin{cases}
		n\qq&\text{ under the conditions of (1),}\\
		\frac{n\qq}{(\log n)^{\kappa-\nicefrac12}}&\text{ under the conditions of (2)},
	\end{cases}
\end{align}
and $\zzeta_{n,i}:=\ba_n   \zeta_{n,i}$.
For $1\leq i\leq n$, we also set, 
for (2) considering in the sequel only $n\ge2$,
\begin{align}\label{defc}
\cd_{i,n}:=\begin{cases}
i^{2\delta-2}\,n^{1-2\delta}&\text{ under the conditions of (1),}\\
i^{-1}(\log n)^{-1}&\text{ under the conditions of (2)}.
\end{cases}
\end{align}
By Lemma~\ref{lem:prelem}, we have, using part \ref{LPLa} for (1) and part
\ref{LPLb} together with the decomposition \eqref{ada} for (2), 
recalling \eqref{kkj}, 
\begin{align}
|\zzeta_{n,0}|\leq \frak C\,\cd_{1,n}\qq\xrightarrow[n\to+\infty]{}0,
\end{align}
meaning that the $\zzeta_{n,0}$ may be ignored in~\eqref{ba'}.

We check that the $\zzeta_{n,i}$ satisfy conditions (a), (b) and (c) of
Theorem~\ref{th:hall-heyde}:

\medskip\textbf{Condition (a).} We want to  show 
the conditional Lindeberg condition
\begin{align}\label{lindeberg}
	\sum_{i=1}^n \E_{i-1} \bigsqpar{|\zzeta_{n,i}|^2 \etta_{|\zzeta_{n,i}|\geq \varepsilon}} &\xrightarrow[n\to+\infty]{\mathrm p} 0,
	\qquad \text{for every $\varepsilon>0$}
.\end{align}
From Lemma~\ref{lem:prelem} and \eqref{bn-def},
\begin{align}\label{aga}
\E_{i-1}|\zzeta_{n,i}|^{q}\leq \frak C\,\cd_{i,n}^{\nicefrac q2}\tfm_{i-1}(V),
\end{align}
where $\cd_{i,n}$ is defined in~\eqref{defc}. 
By~\eqref{lmv4a} in \refL{Lmwqq}, this implies 
\begin{align}\label{khan}
	\E\left[|\zzeta_{n,i}|^2\etta_{|\zzeta_{n,i}|\geq \varepsilon}\right]
\le \eps^{2-q}\E \bigabs{\zzeta_{n,i}}^{q}
\le \frak C_{\eps} \cd_{i,n}^{\nicefrac q2},
\end{align}
 for some constant $\frak C_{\eps}$ which may depend on $\fm_0$ and $\eps$.
We deduce that
\begin{align}\label{pong}
\E\sum_{i=1}^n \E_{i-1}\bigsqpar{|\zzeta_{n,i}|^2 \etta_{|\zzeta_{n,i}|\geq \varepsilon}}
&= \sumin 	\E\left[|\zzeta_{n,i}|^2\etta_{|\zzeta_{n,i}|\geq \varepsilon}\right]
\le \frak \Ceps \sumin \cd_{i,n}^{\nicefrac q2}
\notag\\ &
\le \frak \Ceps \bigpar{\max_{i\le n} \cd_{i,n}}^{\nicefrac q2-1}\sumin \cd_{i,n}
\to 0,
\end{align}
as \ntoo, 
since \eqref{defc} implies
$\max_{i\le n} \cd_{i,n}=\cd_{1,n} \to0$ and
$\sumin \cd_{i,n}\le C$.
 Hence, 
\eqref{lindeberg}
holds, which is  Condition~(a) of Theorem~\ref{th:hall-heyde}.

\medskip\textbf{Condition (b).} First note that for $i\ge1$, 
using the fact that $\zzeta_{n,i} = \ba_n\zeta_{n,i}$ and the definition of $\zeta_{n,i}$ in \eqref{zetai}, 
and setting $\fin:=B_{i,\nnii}f$, we obtain
\begin{align}\label{qd}
	\E_{i-1} [\zzeta_{n,i}^2]&
=\ba_n^2\gam_{i-1}^2\E_{i-1}\big[(\gD M_i \fin)^2\big]
=\ba_n^2\gam_{i-1}^2 \E_{i-1}\bigsqpar{\bigpar{R\nni_{Y_i}\fin-\E_{i-1}R\nni_{Y_i}\fin}^2}
	\notag\\&
	=\ba_n^2\gam_{i-1}^2 \bigl(\E_{i-1}\big[\bigpar{R\nni_{Y_i}\fin}^2\big]
	-\xpar{\E_{i-1}R\nni_{Y_i}\fin}^2\bigr)
	\notag\\&
	=\ba_n^2\gam_{i-1}^2 \bigl(\E_{i-1}\BB_{Y_i}(\fin)
	-\xpar{\E_{i-1}R\nni_{Y_i}\fin}^2\bigr)
	\notag\\&
	=\ba_n^2\gam_{i-1}^2 \big(
	\tfm_{i-1}
	\BB(\fin)
	-\bigpar{\tfm_{i-1}\RR\fin\big)^2
	}
	\notag\\&
	=\ba_n^2\gam_{i-1}^2 \big(
	\nu\BB(\fin) + \fv_{i-1}\BB(\fin) -\bigpar{\tfm_{i-1}\RR\fin}^2
	\big)
	.\end{align}
We treat the three terms in the final parenthesis separately.
We start with the third term; by \eqref{nurr} 
and~\eqref{nuB},
\begin{align}\label{qg}
	\nu \RR \fin = \nu\fin
	=\nu B_{i,\nnii}f=\nu f=0.
\end{align}
Hence, according to~\eqref{eq:LP-W-2bis}, 
and using $\fv_0 W<\infty$ when $i=1$,
there exists $\varepsilon>0$ such that
\begin{align}\label{lyra} 
\ba_n^2\gam_{i-1}^2	\E \big[\abs{\tfm_{i-1}\RR\fin}^2\big]
= \ba_n^2\gam_{i-1}^2 \E \big[|\fv_{i-1}\RR\fin|^2\big]
\le \frak C \ba_n^2\gam_{i-1}^2 i^{-\varepsilon} \norm{\fin}_{B(W)}^2
.\end{align}
Furthermore, by \eqref{gamma} and Lemma~\ref{L2bis}, 
again using the decomposition \eqref{ada} for (2),
we have for all $n\geq i\geq 1$,
\begin{align}
	\label{eq:fin-bound}
 \ba_n^2\gam_{i-1}^2\norm{\fin}_{B(W)}^2
= \ba_n^2\gam_{i-1}^2\|B_{i,n}f\|^2
\leq \frak C \cd_{i,n}.
\end{align}
We thus get that
\begin{align}\label{qk3}
\E\Bigabs{	\sumin \ba_n^2\gam_{i-1}^2(\tfm_{i-1}\RR\fin)^2}\le
\E	\sumin \ba_n^2\gam_{i-1}^2\bigabs{\tfm_{i-1}\RR\fin^2}&
	\le \frak C\,\sumin i^{-\varepsilon}\cd_{i,n}\xrightarrow[n\to+\infty]{}0.
\end{align}

\smallskip
We now treat the second term in~\eqref{qd}. 
Using~\eqref{eq:LP-W-3bis} in Lemma~\ref{LP-Wbis}, together with
the fact that $\BB$ is a bounded quadratic operator 
$\BW\to B(W^2)$ (see~\eqref{rbb}),
and \eqref{eq:fin-bound},
	we obtain that there exists $\varepsilon>0$ such that
	\begin{align}\label{qh-W}
		\ba_n^2\gam_{i-1}^2\E |\fv_{i-1}\BB(\fin) | 
		&\le \frak C i^{-\varepsilon} \ba_n^2\gam_{i-1}^2\norm{\BB(\fin)}_{B(W^2)}
		\notag
		\le \frak C i^{-\varepsilon} \ba_n^2\gam_{i-1}^2\norm{\fin}_{B(W)}^2\\
		&\le \frak C i^{-\varepsilon} \,\cd_{i,n}
	\end{align}
	and hence
	\begin{align}\label{qk2-W}
		\sumin \ba_n^2\gam_{i-1}^2   \E \bigabs{\fv_{i-1}\BB(\fin)}&
		\le \frak C \sumin i^{-\varepsilon} \cd_{i,n}
		\xrightarrow[n\to+\infty]{}0.
	\end{align}

\smallskip

We now consider the first term of~\eqref{qd}, 
which needs a different treatment under the conditions of (1) and (2), so we
treat the two cases separately.

\textit{Under the conditions of {\rm (1)},} we rewrite the sum of the terms corresponding to the first term in~\eqref{qd} (note that this sum is non-random) as an integral:
\begin{align}\label{ql}
	\sumin \ba_n^2\gam_{i-1}^2 \nu \BB(\fin)=	\sumin n\gam_{i-1}^2 \nu \BB(\fin)
	=
	\intoi n^2\gam_{\floorx{nx}}^2\nu\BB(f_{\ceil{xn},n}) \dd x.
\end{align}
Using $\nu W^2<\infty$ (implied by \ref{as:nu2}),
\eqref{rbb}, and
\eqref{eq:fin-bound},
we obtain that
\begin{align}\label{qm}
	n^2\gam_{\floorx{nx}}^2\big|\nu\BB(f_{\ceil{nx},n})\big|&
\le n^2\gam_{\floorx{nx}}^2\norm{\BB(f_{\ceil{nx},n})}_{B(W^2)}
	\le C n^2\gam_{\floorx{nx}}^2 \norm{f_{\ceil{nx},n}}_{B(W)}^2
\notag\\&
\le \frak C n \cd_{\ceil{nx},n}
= \frak C {\ceil{nx}}^{2\delta-2}n^{2-2\delta}\leq \frak C x^{2\delta-2}.
\end{align}
Furthermore, for every fixed $x\in(0,1)$, 
we have
by \refL{L1},
uniformly for $z$ in a compact set and all $n\ge1$,
\begin{align}
  b_{\ceil{nx},n}(z)
=\Bigpar{1+O\Bigparfrac{1}{\ceil{nx}}}\parfrac{n}{\ceil{nx}}^{z-1}
=x^{1-z}+O(1/n).
\end{align}
Hence, \eqref{||h||} shows that
\begin{align}\label{qoj}
  \norm{B_{\ceil{nx},n}-x^{1-\RR}} = 
  \norm{b_{\ceil{nx},n}(\RR)-x^{1-\RR}} = O(1/n),
\end{align}
and, in particular,
\begin{align}\label{qo}
	f_{\ceil{nx},n}=B_{\ceil{nx},\nnii}f
	\xrightarrow[n\to+\infty]{} x^{1-\RR}f 
= x x^{-\RR}f  
= x\ee^{-(\log x)\RR}f  
\end{align}
in $B(W)$.
Furthermore, 
$g\mapsto \nu g$ is a continuous linear functional on $B(W^2)$, since 
$\nu W^2<\infty$, and thus, recalling \refR{RBB},
$f\mapsto\nu\BB(f)$ is a continuous quadratic form on
$B(W)$. 
Hence, \eqref{qo} implies
\begin{align}\label{qn0}
	\nu\BB(f_{\ceil{xn},n})
	\xrightarrow[n\to+\infty]{} \nu\BB\bigpar{x\ee^{-(\log x)\RR}f } 
	=
x^2\nu\BB\bigpar{\ee^{-(\log x )\RR}f }. 
\end{align}
Moreover, 
we have $n^2\gam_{\floorx{nx}}^2\to x^{-2}$ when $n\to+\infty$.
Consequently, for every fixed $x\in(0,1)$, 
\begin{align}\label{qn}
	n^2\gam_{\floorx{nx}}^2\nu\BB(f_{\ceil{xn},n})
	\xrightarrow[n\to+\infty]{} 
	\nu\BB\bigpar{\ee^{-(\log x )\RR}f }. 
\end{align}
Consequently, by \eqref{ql} and dominated convergence justified by \eqref{qm},
followed by a change of variables,
\begin{align}\label{qk1}
	\sumin \ba_n^2\gam_{i-1}^2\nu \BB(\fin)&
	\xrightarrow[n\to+\infty]{}
	\intoi \nu\BB\bigpar{\ee^{-(\log x )\RR}f } \dd x
	=
	\intoo \nu\BB\bigpar{\ee^{s R}f } \ee^{-s}\dd s
	\notag\\&
	=\gox(f),
\end{align}
where it also follows that the integral is absolutely convergent as claimed
in Theorem~\ref{T1}. 
The final equality in \eqref{goxf} follows by Fubini's theorem.

\textit{Under the conditions of {\rm (2)},}
we observe that for $n,m\ge1$,
	\begin{align}\label{menlos}
\left(\frac{n}{m}\right)^{\RR-\II}\Pi_{\lambda_j}f
=\left(\frac{n}{m}\right)^{\lambda_j-1}\left(\frac{n}{m}\right)^{\RR-\lambda_j}
\Pi_{\lambda_j}f
=\left(\frac{n}{m}\right)^{\lambda_j-1}
\sum_{k\geq 0}\frac{1}{k!}(\log\nicefrac{n}{m})^{k}(\RR-\lambda_{j})^k\Pi_{\lambda_j} f,
\end{align}
	where the terms with $k\geq \kappa_j$ are null by \eqref{kkj}.
We deduce from
\eqref{mamba}, \eqref{ada} and \eqref{menlos}
that, for  $n\geq m\ge1$, 
	\begin{align}
		\label{Bmndecomp}
f_{m,n}=
	B_{m,n}f 
	= 
	(1+h_{m,n}(\RR))
\sum_{j=1}^p
\sum_{k=0}^{\kappa_j-1}
\left(\frac{n}{m}\right)^{\lambda_j-1}
\frac{1}{k!}\,
\left(\log(\nicefrac{n}{m})\right)^k (\RR-\lambda_j \II)^k \Pi_{\lambda_j} f
	\end{align}
	where $\opnorm{h_{m,n}(\RR)} = \mathcal O(\nicefrac1m)$ on $D$ by
    \eqref{l1h} and \eqref{||h||}.
We deduce that
\begin{align}\label{bidd}
\sumin \ba_n^2\gam_{i-1}^2 \nu \BB(\fin)
=
\sum_{j,j'=1}^p
\sum_{k=0}^{\kappa_j-1}
\sum_{\ell=0}^{\kappa_{j'}-1}
\sum_{i=1}^n a^{\sss (k,\ell,j,j')}_{n,i}x^{\sss (k,\ell,j,j')}_{n,i},
\end{align}
where
\begin{align}\label{qoa}
a^{\sss (k,\ell,j,j')}_{n,i} = 
i^{-\lambda_j-\lambda_{j'}}n^{\lambda_j+\lambda_{j'}-1}(\log \nicefrac{n}{i})^{k+\ell}/(\log n)^{2\kappa-1},
\end{align}
and, using that $\BBB$ is a bounded bilinear operator (\refR{RBB}) and thus
$\nu\BBB$ is a bounded bilinear form on $\BW$,
and also \eqref{gamma},
\begin{align}\label{qox}
x_{n,i}^{\sss (k,\ell,j,j')} 
&= \frac{\gamma_{i-1}^2i^2}{k!\ell!}
\nu\BBB\left((1+h_{i,n}(\RR))(\RR-\lambda_j \II)^k
  \Pi_{\lambda_j}f,(1+h_{i,n}(\RR))(\RR-\lambda_{j'}
  \II)^\ell\Pi_{\lambda_{j'}} f\right)
\notag\\
&=
\frac{\gamma_{i-1}^2i^2}{k! \ell!}
\nu\BBB\left((\RR-\lambda_j
           \II)^k\Pi_{\lambda_j} f,(\RR-\lambda_{j'}
           \II)^\ell\Pi_{\lambda_{j'}} f\right)
+O(\nicefrac{1}{i})
\notag\\
&=
\frac{1}{k! \ell!}
\nu\BBB\left((\RR-\lambda_j
           \II)^k\Pi_{\lambda_j} f,(\RR-\lambda_{j'}
           \II)^\ell\Pi_{\lambda_{j'}} f\right)
+O(\nicefrac{1}{i})
.\end{align}

Fix $k,\ell,j,j'$ as in \eqref{bidd}. By definition of $a^{\sss
  (k,\ell,j,j')}_{n,i}$ (see~\eqref{qoa}), and because $\Re(\lambda_j)=\Re(\lambda_{j'})=\nicefrac12$,
  we get
\begin{align}\label{jul}
	\sum_{i=1}^n a^{\sss (k,\ell,j,j')}_{n,i}
	&= \frac{n^{\ii\Im(\lambda_j+\lambda_{j'})}}{(\log n)^{2\kappa-1}}
   \sum_{i=1}^n i^{-1-\ii\Im(\lambda_j+\lambda_{j'})}\log^{k+\ell}(\nicefrac
   ni)
\notag\\
	&= \begin{cases}
		\frac{1+o(1)}{2\kappa-1}&\text{ if
        }\Im(\lambda_j+\lambda_{j'})=0\text{ and }k=\ell=\kappa-1,
\\
		o(1)&\text{ if }\Im(\lambda_j+\lambda_{j'})\neq 0\text{ or }k<\kappa-1\text{ or }\ell<\kappa-1,
	\end{cases} 
\end{align}
when $n\to+\infty$,
where we refer to Lemma~\ref{lem:tecLemma} for detailed calculations.
Furthermore,
\begin{align}\label{staffan}
	\sum_{i=1}^n \bigabs{a^{\sss (k,\ell,j,j')}_{n,i}}\frac{1}{i}
	&= \frac{1}{(\log n)^{2\kappa-1}}
   \sum_{i=1}^n i^{-2}\log^{k+\ell}(\nicefrac ni)
\le
 \frac{(\log n)^{k+\ell}}{(\log n)^{2\kappa-1}}
   \sum_{i=1}^n i^{-2}
\le \frac{C}{\log n}\to0.
\end{align}
It folllows from \eqref{bidd}, \eqref{qox}, \eqref{jul} and \eqref{staffan}
that, as \ntoo,
\begin{multline}
	\sumin \ba_n^2\gam_{i-1}^2 \nu \BB(\fin)
	\to  \goy(f):=
\\
\sum_{j,j'=1}^p  \frac{\etta_{\kappa_j=\kappa_{j'}=\kappa,\
          \overline{\lambda_j}=\lambda_{j'}}}{(2\kk-1)((\kappa-1)!)^2}
\nu\BBB\left((\RR-\lambda_j \II)^{\kappa-1}\Pi_{\lambda_j} f,(\RR-\lambda_{j'} \II)^{\kappa-1}\Pi_{\lambda_{j'}} f\right).
	\label{eq:condc-claim2}
\end{multline}

\medskip\textbf{Condition (c).} 
Condition~(c) of Theorem~\ref{th:hall-heyde} 
is verified in the same way as  Condition (b) above, with
mainly notational differences.
We therefore omit the details and only give a sketch.
For $i\ge1$, we have, corresponding to \eqref{qd},
\begin{align}\label{qdbis}
	\E_{i-1} [|\zzeta_{n,i}^2|]&
	=\ba_n^2\gam_{i-1}^2 \big(
	\nu\mathbf{C}(\fin) + \fv_{i-1}\mathbf{C}(\fin) -\bigabs{\tfm_{i-1}\RR\fin}^2
	\big),
\end{align}
with $\BC$ defined in \eqref{BB}.
As for (b),
the two last terms can be neglected and, concerning the first term, we have
the following convergence results, which depend on
whether we work under the conditions of (1) or (2):

\textit{Under the conditions of {\rm (1)}.} One shows that
\begin{align}\label{qk1bis}
	\sumin \ba_n^2\gam_{i-1}^2\nu \mathbf{C}(\fin)
	\xrightarrow[n\to+\infty]{}&
	\intoi \nu\mathbf{C}\bigpar{\ee^{-(\log x )\RR}f } \dd x
	=
	\intoo \nu\mathbf{C}\bigpar{\ee^{s R}f } \ee^{-s}\dd s
	\notag\\&
	=\gss(f),
\end{align}
where the integral is absolutely convergent, as claimed in~Theorem~\ref{T1}.

\textit{Under the conditions of {\rm (2)}.} Using the same approach as
above, but conjugating the second argument of $\BBB$, we obtain, as
$n\to+\infty$, 
\begin{align}\label{lucia}
\sumin \ba_n^2\gam_{i-1}^2 \nu \mathbf{C}(\fin)&\to  
\sum_{j,j'=1}^p  \frac{\etta_{\kappa_j=\kappa_{j'}=\kappa,\ \lambda_j=\lambda_{j'}}}
{(2\kk-1)((\kappa-1)!)^2}
\nu\BBB\left((\RR-\lambda_j \II)^{\kappa-1}\Pi_{\lambda_j} f,
\overline{(\RR-\lambda_{j'} \II)^{\kappa-1}\Pi_{\lambda_{j'}} f}\right)
.\end{align}
Note that the condition  $\lambda_j=\overline{\lambda_{j'}}$ 
in \eqref{eq:condc-claim2} has been changed into 
$\lambda_j={\lambda_{j'}}$, so that the sum in \eqref{lucia} really is a
single sum; hence \eqref{lucia} can be written, recalling \eqref{BB},
\begin{align}
	\label{eq:condc-claim2bis}
	\sumin \ba_n^2\gam_{i-1}^2 \nu \mathbf{C}(\fin)&
\xrightarrow[n\to+\infty]{} 
\sum_{j=1}^p  \frac{\etta_{\kappa_j=\kappa}}{(2\kk-1)((\kappa-1)!)^2}
{\nu\mathbf{C}}\left((\RR-\lambda_j \II)^{\kappa-1}\Pi_{\lambda_j} f\right)=\gs^2(f).
\end{align}

\medskip
We have thus checked that under the conditions of either 
Theorem~\ref{T2}(1) or Theorem~\ref{T2}(2), 
the $\zzeta_{n,i}$ ($1\leq i\leq n$) satisfy Conditions (a), (b) and (c) of Theorem~\ref{th:hall-heyde}.
The values of $\gox$ and $\sigma^2$ in (b) and (c) are given by~\eqref{qk1} and~\eqref{qk1bis} under the conditions of (1),
and by~\eqref{eq:condc-claim2} and~\eqref{eq:condc-claim2bis} under the
conditions of (2).
Therefore, since $\ba_n\fv_n f = \sum_{i=0}^n \zzeta_{n,i}$, and since we have
shown that $\zeta_{n,0}\to 0$ a.s., 
\refT{th:hall-heyde} yields the results \eqref{t11} and \eqref{t12}.

\subsection{Proof of Theorem~\ref{T2}(3)}\label{SpfT2}
Next, consider the case of a
generalized eigenfunction corresponding to an eigenvalue $\gl$ with
$\Re\gl>\half$.
We state a lemma under slightly more general assumptions. 

\begin{lemma}
  \label{LIH1}
Suppose that $\Dx$ is an $\RR$-invariant subspace of $\BW$ such that
$\gs(\RR_\Dx)=\set{\gl}$ consists of a single point $\gl$ with
$\half<\Re\gl\le1$. 
Then each operator $B_{m,n}$, $0\le m\le n$, is invertible on $\Dx$. 
If $f\in \Dx$, then there exists a complex random variable $\gL_f$ 
such that
\begin{align}\label{lih1a}
\fv_n B_{0,n}\qw f\to \gL_f  
\end{align}
\as{} and in $L^2$ as \ntoo; moreover,
 for any $0<\eps  <\Re\gl-\half$, there exists a constant $C_\eps>0$
such that
  \begin{align}
  \label{eq:speedconv1}
  \E\left|\fv_n B_{0,n}\qw f-\gL_f  \right|^2\leq \frac{C_\eps\,\tfm_0 V}{n^{2(\Re\lambda-\eps)-1}} \norm{f}_{\BW}^2. 
  \end{align}  
Furthermore, 
\begin{align}\label{lih1b}
  \E\gL_f=\fv_0f=\tfm_0f-\nu f
\end{align}
 and
 \begin{align}\label{lih1bsquare}
 \E |\Lambda_f|^2\leq C\tfm_0 V\,\|f\|_{\BW}^2.
 \end{align}   
\end{lemma}

\begin{proof}
First note that, since 
$\Re\gl>\half$ and $0<\gam_k<1$, 
we have 
\begin{align}\label{radh}
 \bigabs{1+\gam_k(\gl-1)}\ge
 \Re\bigpar{1+\gam_k(\gl-1)}
=1-\gam_k+\gam_k\Re\gl>\half, 
\qquad k\ge0,  
\end{align}
and thus
$b_{m,n}(\gl)\neq0$ by \eqref{bmn}.
Hence $b_{m,n}\neq0$ on $\gs(\RR_\Dx)$, so 
$b_{m,n}\qw$ is analytic in a neighbourhood of $\gs(\RR_\Dx)=\set{\gl}$
and it follows that,
as an operator on $\Dx$,  
$B_{m,n}=b_{m,n}(\RR_\Dx)$
is invertible with inverse
$B_{m,n}\qw=b_{m,n}\qw(\RR_\Dx)$.

If $0\le i\le n$, then \eqref{BNN} (or \eqref{Bmn}) shows that
$B_{0,i}B_{i,n}=B_{0,n}$, which yields
\begin{align}\label{h2}
  B_{i,n}B_{0,n}\qw=B_{0,i}\qw.
\end{align}

Let $f\in\Dx$.
By \eqref{alf}, applied to $B_{0,n}\qw f$, and \eqref{h2}, we have 
\begin{align}\label{h3}
  \fv_nB_{0,n}\qw f
=\fv_0f +\sumin \gam_{i-1}\gD M_i B_{0,i}\qw f.
\end{align}
Since $\E_{i-1}\gD M_i=0$, \eqref{h3} shows that 
$\bigpar{ \fv_nB_{0,n}\qw f}_{n\ge0}$ is a martingale.
(Cf.\ the closely related \eqref{hq}.)
We will show that the martingale \eqref{h3} is $L^2$ bounded; the result
\eqref{lih1a}
then follows by the martingale convergence theorem.

Let $0<\eps  <\Re\gl-\half$, and let $K$ be the closed disc
$\set{z:|z-\gl|\le\eps}$.
Then $\Re z>\half$ for all $z\in K$, and 
it follows, 
as in \eqref{radh}, that $|1+\gam_k(z-1)|>\half$ on $K$.
Hence,  each $b_{m,n}$ is non-zero on $K$, and thus is invertible
on $K$;
furthermore, \eqref{bmn} gives the trivial bound
\begin{align}\label{runar}
  \bigabs{b_{m,n}(z)\qw} \le 2^{n-m}, 
\qquad 0\le m\le n,\;z\in K.
\end{align}
To get a better bound, we fix $m \geq 1$ such that \eqref{l1b} holds
for all $n\ge m$
 and $z\in K$. Thus, using \eqref{runar} for $b_{0,m}(z)\qw$, 
\begin{align}\label{h7}
  \bigabs{b_{0,n}(z)\qw}
&=
 \bigabs{ b_{0,m}(z)\qw
  b_{m,n}(z)\qw}
\le C   |b_{m,n}(z)|\qw
=C \bigabs{\ee^{-(z-1)\log\frac{\fm_0(E)+n}{\fm_0(E)+m}+O(1)}}
\notag\\&\le C \left(\frac{\fm_0(E)+n}{\fm_0(E)+m}\right)^{\!\!1-\Re z}
 \le C \left({\fm_0(E)+n}\right)^{1-\Re \gl+\eps}.
\end{align}
By \eqref{runar}, the same bound holds trivially (with a suitable $C$) also
for $1\le n<m$,
so the estimate \eqref{h7} holds for all $n\ge1$ and $z\in K$.

Since
$\gs(\RR_{\Dx})=\set{\gl}$, it follows from \eqref{||h||}, taking $\gG$ to
be the  circle $\set{z:|z-\gl|=\eps}\subset K$, together with \eqref{h7} that
 \begin{align}\label{h8}
\norm{B_{0,n}\qw}_{\Dx} \le C \sup_{z\in K}|b_{0,n}\qw(z)|
\le C \left({\fm_0(E)+n}\right)^{1-\Re \gl+\eps}, \qquad n\ge1.
\end{align}
For $f\in \Dx$ and  $i\ge1$, we thus have, by
\eqref{lg} and \eqref{h8}, 
 \begin{align}\label{h9}
  \E\bigabs{\gam_{i-1}\gD M_i B_{0,i}\qw f}^2
&\le C\gam_{i-1}^2\tfm_0 V\,\norm{B_{0,i}\qw f}_{\Dx}^2
\notag\\
&\le C \tfm_0 V\, (\fm_0(E)+i)^{-2} (\fm_0(E)+i)^{2(1-\Re \gl+\eps)}\norm{f}_{\Dx}^2
\notag\\
&= C  \tfm_0 V\, (\fm_0(E)+i)^{-2(\Re \gl-\eps)}\norm{f}_{\Dx}^2
.\end{align}
Since $2(\Re\gl-\eps)>1$, 
it follows from \eqref{h3} (where the terms are orthogonal), \eqref{lg0}, and
\eqref{h9},
that 
\begin{align}
\label{eq:usefulm0rand}
\E\big|\fv_nB_{0,n}\qw f\big|^2\leq |\fv_0 f|^2+\sumi  \E\bigabs{\gam_{i-1}\gD M_i B_{0,i}\qw f}^2\leq C\tfm_0 V\,\norm{f}_{\Dx}^2,
\end{align}
and thus the
martingale
\eqref{h3} converges in $L^2$ and a.s., as claimed. 
The properties ~\eqref{eq:speedconv1}, \eqref{lih1b} and~\eqref{lih1bsquare}
immediately follow
from \eqref{h3} and \eqref{h9}--\eqref{eq:usefulm0rand}.
\end{proof}

We combine \refL{LIH1} with a standard result for  functions of nilpotent
operators.
\begin{lemma}
  \label{Lnil}
Suppose that $\Dx$ is an $\RR$-invariant subspace of $\BW$ such that
$(\RR_\Dx-\gl)^\kk=0$ for some complex $\gl$ and integer $\kk\ge1$.
Let $h$ be a function that is analytic in a neighbourhood of $\gl$. Then,
for $f\in \Dx$,
\begin{align}
  \label{lnil}
h(\RR_\Dx)f = 
\sum_{k=0}^{\kk-1}\frac{h^{(k)}(\gl)}{k!}{(\RR-\gl)^kf}
.\end{align}
\end{lemma}
\begin{proof}
  A Taylor expansion yields, for some function $h_\kk$ analytic in the same
  domain as $h$,
\begin{align}\label{ln2}
h(z) = \sum_{k=0}^{\kk-1}\frac{h^{(k)}(\gl)}{k!}{(z-\gl)^k}
+(z-\gl)^\kk h_\kk(z).
\end{align}
We have $(\RR_\Dx-\gl)^\kk=0$ by assumption, and thus \eqref{ln2}
yields, using \eqref{h1h2}, 
\begin{align}
h(\RR_\Dx) = 
\sum_{k=0}^{\kk-1}\frac{h^{(k)}(\gl)}{k!}{(\RR_\Dx-\gl)^k},
\end{align}
as operators on $\Dx$, 
which is \eqref{lnil}.
\end{proof}

We can now 
show the convergence \eqref{hildegran} for $f$ in a generalized eigenspace.

\begin{lemma}
  \label{LIH2}
Suppose that $\Dx$ is an $\RR$-invariant subspace of $\BW$ such that
$(\RR_\Dx-\gl)^\kk=0$ for some complex $\gl$ with $\half<\Re\gl\le1$
and some integer $\kk\ge1$.
If  $f\in \Dx$, then, for some complex random variable $\gL$,
\begin{align}\label{lh2-}
  \frac{n^{1-\gl}}{\log^{\kk-1} n} \,\fv_n f 
\to 
\gL
\end{align}
\as{} and in $L^2$ as \ntoo.
Furthermore,
\begin{align}
  \label{lh2e}
\E\gL=\frac{\gG(\fm_0(E)+1)}{(\kk-1)!\,\gG(\fm_0(E)+\gl)}\fv_0(\RR-\gl)^{\kk-1}f.
\end{align}
\end{lemma}

\begin{proof}
%
Note that the assumption $(\RR_\Dx-\gl)^\kk=0$ implies that
$\gs(\RR_\Dx)=\set\gl$, 
for example by the spectral mapping theorem
\cite[Theorem VII.4.10]{Conway}.
Hence, \refL{LIH1} applies.

We use also \refL{Lnil} with $h=b_{0,n}$.
This yields, defining
$f_k:=(\RR-\gl)^k f/k!$,
\begin{align}\label{ih1}
  B_{0,n}f 
=  b_{0,n}(\RR_\Dx)f 
= \sum_{k=0}^{\kk-1} b_{0,n}\kkk(\gl) f_k
\end{align}
and thus
\begin{align}\label{ih2}
  \fv_n f=\fv_n B_{0,n}\qw B_{0,n}f
= \sum_{k=0}^{\kk-1} b_{0,n}\kkk(\gl) \fv_n B_{0,n}\qw f_k.
\end{align}
Each random variable $\fv_n B_{0,n}\qw f_k$ converges 
\as{} and in $L^2$
as \ntoo{}
by \refL{LIH1}, and it remains to study
the coefficients $b_{0,n}\kkk(\gl)$.
By \eqref{l1o} we have
\begin{align}\label{i1}
  b_{0,n}(z)=\bigpar{1+h_{0,n}(z)} \left(\frac{\fm_0(E)+n}{\fm_0(E)+1}\right)^{\!\!z-1}
.\end{align} 
In a fixed neighbourhood of
$\gl$, the functions $h_{0,n}(z)$, $n\ge1$, are uniformly bounded by
\eqref{l1h}, and thus Cauchy's estimates show that for each fixed  $k\ge0$,
\begin{align}\label{i2}
\bigabs{ h_{0,n}^{(k)}(\gl)}\le C.
\end{align}
Furthermore,
\begin{align}
   \frac{\ddx^k}{\dd z^k} \left(\frac{\fm_0(E)+n}{\fm_0(E)+1}\right)^{\!\!z-1}=\log^k\left(\frac{\fm_0(E)+n}{\fm_0(E)+1}\right) \left(\frac{\fm_0(E)+n}{\fm_0(E)+1}\right)^{\!\!z-1}.
\end{align}
Hence, using \eqref{i1} and Leibniz' rule, for a fixed $k$ and 
 $n\ge1$,
 \begin{align}
&\left|  b_{0,n}\kkk(\gl)- \bigpar{1+h_{0,n}(\gl)} \left(\frac{\fm_0(E)+n}{\fm_0(E)+1}\right)^{\!\!\gl-1}\log^k\left(\frac{\fm_0(E)+n}{\fm_0(E)+1}\right)\right|
\notag\\
&\qquad\qquad\leq  C\,\left(\frac{\fm_0(E)+n}{\fm_0(E)+1}\right)^{\!\!\Re\gl-1}\left( 1\vee \log^{k-1}\left(\frac{\fm_0(E)+n}{\fm_0(E)+1}\right)\right)\label{i3}
.\end{align} 
By \eqref{i3}, each coefficient $b\kkk_{0,n}(\gl)$ in \eqref{ih2}
with $k<\kk-1$  satisfies
 \begin{align}
\Bigg|\left(\frac{\fm_0(E)+n}{\fm_0(E)+1}\right)^{\!\!1-\lambda} \frac{ b_{0,n}\kkk(\gl)}{1\vee \log^{\kappa-1}\big(\frac{\fm_0(E)+n}{\fm_0(E)+1}\big)}\Bigg|
\leq  \frac{C}{ 1\vee\log\big(\frac{\fm_0(E)+n}{\fm_0(E)+1}\big)}.\label{prev}
 \end{align}   
Using \eqref{i3} with $k=\kappa-1$, we similarly deduce that
\begin{align}
\label{eq:newexplicit1}
\Bigg|  \left(\frac{\fm_0(E)+n}{\fm_0(E)+1}\right)^{\!\!1-\gl}\frac{b_{0,n}^{(\kappa-1)}(\lambda)}{1\vee \log^{\kappa-1}\big(\frac{\fm_0(E)+n}{\fm_0(E)+1}\big)} - \bigpar{1+h_{0,n}(\gl)} \Bigg|
\leq \frac{C}{ 1\vee\log\big(\frac{\fm_0(E)+n}{\fm_0(E)+1}\big)}.
\end{align}
We obtain
from \eqref{ih2},~\eqref{prev}, \eqref{eq:newexplicit1}, and  the
fact that $h_{0,n}(\lambda)$ is uniformly bounded, 
 with $\gL_{f_{\kk-1}}$ from \refL{LIH1},
\begin{align}
&\left|\left(\frac{\fm_0(E)+n}{\fm_0(E)+1}\right)^{\!\!1-\gl}   
\frac{ \fv_n f}{1\vee\log^{\kk-1} \big(\frac{\fm_0(E)+n}{\fm_0(E)+1}\big)}
-\bigpar{1+h_{0,n}(\gl)} \gL_{f_{\kk-1}}\right|
\notag\\
&\qquad\qquad\leq  
C\frac{\sum_{k=0}^{ \kappa-1}
|\fv_n B_{0,n}\qw f_k| }{ 1\vee\log\big(\frac{\fm_0(E)+n}{\fm_0(E)+1}\big)}+C\left|\fv_n B_{0,n}\qw f_{\kappa-1}-\Lambda_{f_{\kappa-1}}\right|.\label{i4bis}
\end{align}
In addition, \eqref{l1o}, \eqref{rgamma} 
 and a well-known consequence of Stirling's formula
(see \eg{} \cite[5.11.13]{NIST})
imply
that for any fixed $z$, 
\begin{align}\label{rgamma2}
{1+h_{0,n}(z)}
&= \left(\frac{\fm_0(E)+n}{\fm_0(E)+1}\right)^{\!\!1-z}b_{0,n}(z)
=\left(\frac{\fm_0(E)+n}{\fm_0(E)+1}\right)^{\!\!1-z}\frac{\gG(n+\fm_0(E)+z)}{\gG(n+\fm_0(E)+1)}
\frac{\gG(\fm_0(E)+1)}{\gG(\fm_0(E)+z)}
\notag\\ 
&= \frac{1}{(\fm_0(E)+1)^{1-z}}
\frac{\gG(\fm_0(E)+1)}{\gG(\fm_0(E)+z)}
\Bigpar{1+O\Bigpar{\frac{1}{n}}}
\notag\\&
= \frac{1}{(\fm_0(E)+1)^{1-z}}
\frac{\gG(\fm_0(E)+1)}{\gG(\fm_0(E)+z)}
+O\Bigpar{\frac{1}{n}}
.\end{align}
\refL{LIH1} implies that the \rhs{} of \eqref{i4bis} tends to 0 \as{} as \ntoo.
Hence, \eqref{i4bis} and \eqref{rgamma2}  imply 
that 
 \begin{align}\label{lh2+}
\left(\fm_0(E)+n\right)^{1-\gl}    \frac{ \fv_n f}{\log^{\kk-1} \big(\frac{\fm_0(E)+n}{\fm_0(E)+1}\big)}
\to 
\gL
\end{align}
holds \as{} with 
\begin{align}
  \label{gLgL}
\gL=\frac{\gG(\fm_0(E)+1)}{\gG(\fm_0(E)+\gl)} \gL_{f_{\kk-1}}.
\end{align} 
We may simplify \eqref{lh2+} and conclude that \eqref{lh2-} holds a.s.

Moreover, \eqref{eq:speedconv1} and \eqref{lih1bsquare} imply
that for every fixed $k\ge0$,
\begin{align}\label{jb1}
  \E\bigabs{\fv_nB_{0,n}\qw f_k}^2 
\le C\tfm_0V \norm{f_k}^2_{\BW}+C\E\bigabs{\gL_{f_k}}^2
\le C\tfm_0V \norm{f_k}^2_{\BW}
\le C\tfm_0V \norm{f}^2_{\BW}
.\end{align}
 Taking the expectation of the square in~\eqref{i4bis},
we deduce  using  \eqref{jb1} and \eqref{eq:speedconv1}, 
\begin{align}
\label{eq:explicit2}
&\E\Bigg|\left(\frac{\fm_0(E)+n}{\fm_0(E)+1}\right)^{\!\!1-\gl}    \frac{ \fv_n f}{1\vee\log^{\kk-1} \big(\frac{\fm_0(E)+n}{\fm_0(E)+1}\big)}
-\bigpar{1+h_{0,n}(\gl)} \gL_{f_{\kk-1}}\Bigg|^2\leq \frac{C \,\tfm_0 V\,\|f\|_{\BW}^2}{1\vee\log^2\big(\frac{\fm_0(E)+n}{\fm_0(E)+1}\big)}
\end{align}
Furthermore, by \eqref{rgamma2}, \eqref{gLgL} and \eqref{lih1bsquare},
\begin{align}\label{jb3}
  \E\lrabs{\bigpar{1+h_{0,n}(\gl)}\gL_{f_{\kk-1}}-\frac{1}{(m_0(E)+1)^{1-\gl}}\gL}^2
\le \frac{C}{n^2}\E\bigabs{\gL_{f_{\kk-1}}}^2
\le \frac{C}{n^2}\tfm_0V\normbwqq{f}^2.
\end{align}
Combining \eqref{eq:explicit2} and \eqref{jb3}, we obtain
\begin{align}\label{jb5}
&\E\,\biggabs{\bigpar{\fm_0(E)+n}^{1-\gl}   
\frac{ \fv_n f}{1\vee\log^{\kk-1} \big(\frac{\fm_0(E)+n}{\fm_0(E)+1}\big)}
- \gL}^2
\notag\\
&\qquad\qquad
\le
C\bigpar{\fm_0(E)+1}^{2(1-\Re\gl)}
\frac{\tfm_0 V}{1\vee\log^2\big(\frac{\fm_0(E)+n}{\fm_0(E)+1}\big)}
\normbwqq{f}^2.
\end{align}
Hence, \eqref{lh2+} holds also in $L^2$, and thus so does \eqref{lh2-}.
Finally, \eqref{lh2e} follows by \eqref{gLgL},
\eqref{lih1b}, 
and the definition of $f_{\kk-1}$,
which completes the proof.
\end{proof}

We are now ready to prove Theorem \ref{T2}(3);
we assume that the operator $\RR_D$ is \slqc{}  and that the spectrum
of $\RR_D$ is given by  
\begin{align}\label{h0}
\sigma(\RR_D)=\{1,\lambda_1,\ldots,\lambda_p\}\cup \Delta, 
\qquad p\geq 1,
\end{align} 
where 
$\Re(\lambda_1)=\dotsm=\Re(\lambda_p)=\gthd\in(\half,1)$, 
and
$\sup\Re(\Delta) < \gthd$.
Note that this implies that $\gl_1,\dots,\gl_j$ are isolated points of
the spectrum $\gs(\RR_D)$, and that $\gD$ is a clopen subset.
Thus the
spectral projections $\Pi_{\gl_j}$ and $\Pi_\gD$ are defined and, for any
$f\in D$, recalling \refL{L0},
\begin{align}\label{h6}
  f=\Pi_1f+\sumjp \Pi_{\gl_j}f+\Pi_{\gD}f
=\nu f+\sumjp \Pi_{\gl_j}f+\Pi_{\gD}f.
\end{align}
Hence, it suffices to prove \eqref{hildegran}
for the functions $\nu f$,
$ \Pi_{\gl_j}f$ and $\Pi_{\gD}f$ separately; in other words, it suffices to
consider the cases 
$f=c$ constant, 
$f\in  \Pi_{\gl_j}D$ and $f\in\Pi_{\gD}D$.
Recall that $\tfm_n-\nu=\fv_n$.

First, we may ignore the constant term $\nu f$ in \eqref{h6}, 
since $\fv_n1=0$.

Secondly, \refL{LIH2} applies to each space
$\Pi_{\gl_j}D$, since we assume
\eqref{kkj} and thus $(\RR-\gl_j)^\kk=0$ on $D_j:=\Pi_{\gl_j}D$.
It follows that, for some complex random variable $\gL_j\in L^2$,
\begin{align}
  \label{lr3j}
  \frac{n^{1-\Re\gl_j}}{\log^{\kk-1} n} \fv_n \Pi_{\gl_j}f
- n^{\ii\Im\gl_j}\gL_j
\to 0
\end{align}
\as{} and in $L^2$.
Furthermore, \eqref{nurr} and Lemma~\ref{L0} imply that
\begin{align}
\nu(\RR-\lambda_j)^{\kappa-1} \Pi_{\lambda_j} f=(1-\lambda_j)^{\kappa-1} 
\nu\Pi_{\lambda_j} f=0,   
\end{align}
so that \eqref{lh2e} yields \eqref{solveig}.

Thirdly, \refL{LK2} applies to $\Pi_{\gD}D$ and $\gthh:=\gth_D=\Re\gl_1$,
and shows
\begin{align}
  \label{lr3gd}
  {n^{1-\Re\gl_1}} \fv_n \Pi_{\gD}f
\to 0
\end{align}
\as{} and in $L^2$.

\refT{T2}(3) follows by combining \eqref{h6} with \eqref{lr3j} and \eqref{lr3gd}.
This completes the proof of \refT{T2}.
\qed

\begin{remark}
 \label{rem:quantif}
Note that \eqref{jb5} implies an upper
bound $O\bigpar{1/\log n}$ 
for the speed of convergence in~$L^2$ of~\eqref{lh2+},
which yields the same rate in \eqref{lh2-} in Lemma~\ref{LIH2}.
Since, in addition, 
\eqref{eq:m0depend} in
the proof of Lemma~\ref{LK2} yields an upper
bound $O\xpar{n^{-\eps}}$ (for some $\eps=\gthh-\gthhh>0$)
for the speed of convergence in~$L^2$
in~\eqref{eq:extendedform0random},  
one finds
$O\bigpar{1/\log n}$ as an explicit upper bound for the speed of
  convergence in~$L^2$ of \refT{T2}(3). 
\end{remark}

 \begin{remark}
     \label{rem:m0rand2}
If $\fm_0$ is random,
then under the conditions of \refL{LIH2},
\eqref{jb5} holds conditioned on $\fm_0$.
Taking the expectation, we see by dominated convergence that
if further 
\begin{align}\label{jb6}
\E\bigsqpar{ \bigpar{\fm_0(E)+1}^{2(1-\Re\gl)}\tfm_0 V}<\infty,  
\end{align}
then
the \lhs{} of \eqref{jb5} converges to 0 as \ntoo.
With the notation
\begin{align}
a_n:=\frac{(\fm_0(E)+n)^{1-\gl}}
{1\vee\log^{\kk-1} \big(\frac{\fm_0(E)+n}{\fm_0(E)+1}\big)},
\end{align}
this says that $a_n\fv_n f \to\gL$ in $L^2$. 
Since we also have convergence \as{} 
(by \eqref{lh2+} and conditioning on $\fm_0$),
this implies that the sequence $|a_n\fv_n f|^2$ is uniformly integrable, see
\eg \cite[Theorem 5.5.2]{Gut}. Let $b_n:=n^{1-\gl}/\log^{\kk-1}n$. Then, for
$n\ge3$, $|b_n|\le|a_n|$, and it follows that also $|b_n\fv_n f|^2$ is \ui.
Furthermore, also $b_n\fv_n f\to \gL$ \as{}, and thus \cite[Theorem
5.5.2]{Gut} again shows that $b_n\fv_n f\to\gL$ in $L^2$. 
Consequently, under the assumption \eqref{jb6}, \eqref{lh2-} holds both
\as{} and in $L^2$.

By combining this and \refR{rem:m0rand0},
it follows as above that 
\eqref{hildegran} in \refT{T2}(3) holds also in $L^2$
for random $\fm_0$ that satisfies \eqref{m0rand+},
      as claimed in Remark~\ref{rem:m0rand}.
 \end{remark}

\section{Proof of \refTs{T210}--\ref{T230}.}\label{Spf3}
In this section we prove
\refTs{T210}--\ref{T230} on possible degeneracies in the limit distributions
in \refT{T2}.

\begin{lemma}\label{Las}
Suppose that $f\in\BW$ and that $\nu|f|=0$, or, equivalently,
\begin{align}\label{las1}
  f(x)=0 \text{ for \nuae\ $x$}.
\end{align}
\begin{romenumerate}
\item \label{Lasa}
Then $\nu|\RR f|=0$, \ie, \eqref{las1} holds for $\RR f$ too.

\item \label{Lasb}
Moreover, for \nuae{} $x$,
\begin{align}\label{las2}
  \RA_xf=0 \qquad a.s.
\end{align}
\end{romenumerate}
\end{lemma}
\begin{proof}
By linearity we may assume that $f\ge0$.
Let $N:=\set{x:f(x)\neq0}$; then $\nu N=0$ by the assumption \eqref{las1}.
If $x\notin N$, then 
$\RA_xf\ge0$, because $\RA_x$ is positive on $E\setminus\set x$ and $f(x)=0$.
Hence, by taking the expectation, also
\begin{align}\label{las4}
\RR f(x) = \E \RA_x f\ge0,
\qquad x\notin N.  
\end{align}
Thus $\RR f\ge0$ \nuae{}

On the other hand, by \eqref{nurr} and the assumption \eqref{las1},
\begin{align}\label{las5}
  \nu(\RR f)=(\nu\RR) f = \nu f = 0.
\end{align}
It follows from \eqref{las4} and \eqref{las5} that $\RR f=0$ \nuae, which
proves \ref{Lasa}.

Moreover, let $N_1:=\set{x:\RR f(x)\neq0}$.
If $x\notin N\cup N_1$, then,
as just shown, $\RA_xf\ge0$, 
and also $\E\RA_xf=\RR f(x)=0$;
hence,
\eqref{las2} holds.
This proves \ref{Lasb}, since 
$\nu(N_1)=0$ by \ref{Lasa}, and thus
$\nu(N\cup N_1)=0$.
\end{proof}

\begin{proof}[Proof of \refT{T210}]
  By replacing $f$ by $f-\nu f$, we may for simplicity assume $\nu f=0$.

Note first that \eqref{zeta} implies $\gss(f)=0\implies\chi(f)=0$.
Hence, \ref{T210a}$\iff$\ref{T210b} follows from the formula for
$\Sigmaq(f)$ in \eqref{t11}.

Next,
$s\mapsto \ee^{s\RR} f$ is a continuous map $\ooo\to\BW$, and
thus, by \refR{RBB} and $\nu W^2<\infty$, 
$s\mapsto \nu \BC\bigpar{\ee^{s\RR} f}$ is a continuous function of $s\ge0$.
Furthermore, by \eqref{BB}, we have $\BC(\ee^{s\RR}f)\ge 0$, and thus
$\nu \BC(\ee^{s\RR}f)\ge 0$. Consequently,
by \eqref{gssf},
\begin{align}\label{tj0}
  \gss(f)=0
\iff
\nu \BC(\ee^{s\RR}f)=0
\qquad \text{for every $s\ge0$}.
\end{align}
In particular, taking $s=0$, we see that 
\ref{T210b}$\implies$\ref{T210c}.

Furthermore, $\nu \BC(f)=0\iff \BC_x f=0$ for \nuae\ $x$, 
which by \eqref{BB} is equivalent to $\RA_x f=0$ \as, for \nuae\ $x$.
Hence, \ref{T210c}$\iff$\ref{T210d}.

Finally assume \ref{T210d}, and let $N\subset E$ be a set with $\nu(N)=0$
such that $\RA_x f=0$ \as\ when $x\notin N$.
By taking the expectation, we obtain $\RR f(x)=\E \RA_x f=0$ for $x\notin N$.
Hence, $\RR f=0$ \nuae, i.e., $\RR f$ satisfies \eqref{las1}.
We may thus apply \refL{Las} to $\RR f$ and conclude by induction that
$\RR^k f=0$ \nuae, for every $k\ge1$. Consequently, for any $s\ge0$,
\begin{align}\label{tj1}
  \ee^{s\RR}f - f = \sumk\frac{s^k}{k!}\RR^k f =0
\qquad \text{\nuae}
\end{align}
We apply \refL{Las} again, this time to $\ee^{s\RR}f-f$, and conclude by
\refL{Las}\ref{Lasb} that for \nuae~$x$,
\begin{align}\label{tj2}
  \RA_x\bigpar{\ee^{s\RR}f-f}=0
\qquad a.s.
\end{align}
Together with the assumption $\RA_x f=0$ \as{} for \nuae\ $x$, 
this shows that
for \nuae\ $x$,
\begin{align}\label{tj3}
  \RA_x\bigpar{\ee^{s\RR}f}=0
\qquad a.s.
\end{align}
Hence, \eqref{BB} yields $\BC_x(e^{s\RR}f)=0$ for \nuae\ $x$, and thus
$\nu \BC(e^{s\RR}f)=0$, for every $s\ge0$.
Consequently, \eqref{tj0} shows that $\gss(f)=0$.
We have shown that \ref{T210d}$\implies$\ref{T210b}, which completes the proof.
\end{proof}

\begin{proof}[Proof of \refT{T220}]
The equivalence \ref{T220a}$\iff$\ref{T220b} follows as in the proof of
\refT{T210}.

  Let $g_j:=(\RR-\gl_j\II)^{\kk-1}\Pi_{\gl_j} f$.
Note that by \eqref{kkj}, 
$g_j =0$ if $\kk_j<\kk$, which shows the equivalence 
\ref{T220c}$\iff$\ref{T220d}.

By \eqref{gssf2}, it remains only to show that
\begin{align}\label{t22a}
  \nu\BC(g_j)=0 \iff g_j=0 \ \text{\nuae}
\end{align}
To see this, we first note that by definition of $\kk$,
\begin{align}\label{t22b}
  (\RR-\gl_j)g_j = (\RR-\gl_j)^\kk \Pi_jf=0.
\end{align}
and thus
\begin{align}\label{t22bb}
  \RR g_j = \gl_jg_j.
\end{align}
In other words, 
$g_j$ is (if non-zero) an eigenfunction with eigenvalue $\gl_j\neq0$.

Assume now $\nu\BC(g_j)=0$. Then, using \eqref{BB} again,
for \nuae~$x$, we have $\BC_x(g_j)=0$ and thus $\RA_xg_j=0$ a.s.
Taking the expectation shows that for such $x$, we have
$\RR g_j(x)=\E\RA_xg_j=0$.
Consequently, $\RR g_j=0$ \nuae, and \eqref{t22bb} implies $g_j=0$ \nuae{}
This shows one implication in \eqref{t22a}.

Conversely, assume $g_j=0$ \nuae{} 
Then \refL{Las} shows that for \nuae~$x$, we have
$\RA_xg_j=0$ \as, and thus $\BC_x(g_j)=0$ by \eqref{BB}. Hence,
$\nu\BC(g_j)=0$.
This completes the proof of \eqref{t22a}, and thus of
\ref{T220b}$\iff$\ref{T220d}, and of the theorem.
\end{proof}

\begin{proof}[Proof of \refT{T230}]
We note first that
\eqref{t22bb}  holds in the present case too,
and thus $g_j$ is an eigenfunction of $\RR$
with eigenvalue $\gl_j\neq1$.
Hence, $\nu$ and $g_j$ are left and right eigenvectors of $\RR$
with different eigenvalues (recall \eqref{nurr}), which implies, as is well
known,
\begin{align}\label{tag0}
  \nu g_j =0,
\end{align}
because we have
\begin{align}
  \nu g_j = (\nu \RR)g_j = \nu (\RR g_j) = \gl_j (\nu g_j).
\end{align}

\ref{T230a}$\iff$\ref{T230b}: Obvious.

\ref{T230b}$\implies$\ref{T230g},\ref{T230p}:
Suppose now that \ref{T230b} holds, \ie, $\gL_j=\E\gL_j$ a.s.
The proofs of \refT{T2}(3) and \refL{LIH2} 
(in particular \eqref{gLgL})
show that 
\begin{align}\label{az}
  \gL_j=c\gL_{g_j}
\end{align}
where $c>0$ is an explicit constant and
$\gL_{g_j}$ is given by \refL{LIH1}.
Since $\gL_{g_j}$ is constructed in the proof of \refL{LIH1} as the limit of
the martingale \eqref{h3} (with $f$ replaced by $g_j$), 
it follows that $\gL_j=\E \gL_j$ \as{} if and only if all martingale
differences in \eqref{h3} vanish \as, i.e.,
\begin{align}\label{tage}
  \gL_j=\E\gL_j \text{ \as} \iff
\gD M_i B_{0,i}\qw g_j = 0 \text{ \as, for every $i\ge1$}.
\end{align}
Moreover, as remarked above, $\RR g_j=\gl_j g_j$, 
Hence, by \eqref{Bbmn},
\begin{align}
 B_{0,i}\qw g_j
=
b_{0,i}(\RR)\qw g_j 
= b_{0,i}(\gl_j)\qw g_j,
\end{align}
where $b_{0,i}(\gl_j)\neq0$ by \eqref{bmn}, $0<\gam_n<1$, and $\Re\gl_j>0$.
Thus \eqref{tage} 
yields
\begin{align}\label{tk}
\gD M_i  g_j = 0 \text{ \as, for every $i\ge1$}
.\end{align}
Using also \eqref{a2}
(and replacing $i$ by $n+1$), \eqref{tk}
says that, for every $n\ge0$,
\begin{align}
  \label{tck}
\RYni g_j = \tfm_{n}\RR g_j \text{\quad \as} 
\end{align}
Conditioning on $\tfm_n$, and recalling that $Y_{n+1}$ has the conditional
distribution $\tfm_n$, we see that \eqref{tck} implies that 
$\tfm_n$ is \as{} such that, conditioned on $\tfm_n$,
\begin{align}
  \label{tag1}
\RAni_x g_j
= \tfm_{n}\RR g_j \text{\quad \as, for $\tfm_n$-\aex{} $x$} 
  .\end{align}
Consider first the case $n=0$.
Recall that $\tfm_0$ is non-random, and 
let $a:=\tfm_0\RR g_j$ (a non-random real number). Then 
the case $n=0$ of \eqref{tag1} says 
\begin{align}\label{tag1a}
  \RA_x g_j = a\text{\quad a.s., for $\tfm_0$-\aex{} $x$}.
\end{align}
Now return to a general $n\ge0$.
Since $\tfm_n$ is a positive number times $\fm_n$, we may in \eqref{tag1}
equivalently write ``for $\fm_n$-\aex{} $x$''. 
Furthermore, $\fm_n\ge\fm_0$,
and thus \eqref{tag1} implies  that the equality holds for $\fm_0$-\aex{}
$x$.
Moreover, $\RAni_x$ is independent of $\tfm_n$, and thus its conditional
distribution equals the distribution of $\RA_x$.
Hence, \eqref{tag1a} shows that, also conditioned on $\tfm_n$, 
\begin{align}
  \label{tag1b}
\RAni_x g_j
= a \text{\quad \as, for $\tfm_0$-\aex{} $x$}
  .\end{align}
Consequently, 
comparing \eqref{tag1} and \eqref{tag1b},
we obtain,
for every $n\ge0$,
\begin{align}\label{tag2}
  \tfm_n\RR g_j = a \text{\quad a.s.}
\end{align}
Thus,
\eqref{tag1} shows that, for every $n\ge0$, $\tfm_n$ is \as{} such that
\begin{align}\label{tag3}
  \RA_x g_j = a\text{\quad a.s., for $\tfm_n$-\aex{} $x$}.
\end{align}
By again
conditioning on $\tfm_n$, 
it follows from \eqref{tag3} that
\begin{align}\label{tb1}
  \RYni g_j=a
\quad\text{a.s.}
\end{align}
Consequently, by \eqref{mn} and induction,
\begin{align}\label{tb2}
  \fm_n g_j = \fm_0 g_j + na
\quad\text{a.s.}
\end{align}

On the other hand, taking the expectation in \eqref{tag3} yields that
$\tfm_n$ is \as{} such that
\begin{align}\label{az1}
 \RR g_j(x)=\E \RA_x g_j = a\text{\quad  for $\tfm_n$-\aex{} $x$}.
\end{align}
Recalling \eqref{t22bb}, this implies that 
\begin{align}\label{az2}
\tfm_ng_j =\gl_j\qw\tfm_n(\RR g_j) = \gl_j\qw a
\quad\text{a.s.}
\end{align}
Consequently, recalling \eqref{mass}, 
\begin{align}\label{az3}
\fm_n g_j = \fm_n(E)\tfm_ng_j =  \bigpar{\fm_0(E)+n}\gl_j\qw a
=   \fm_0(E)a/\gl_j+  n a/\gl_j
\quad\text{a.s.}
\end{align}
Comparing \eqref{tb2} and \eqref{az3}, we see that $a=a/\gl_j$, and thus
(since $\gl_j\neq1$), $a=0$.
Consequently, \eqref{az2} says $\tfm_n g_j=0$ a.s., which completes the
proof of \ref{T230b}$\implies$\ref{T230g}.

Moreover,
\eqref{az1} with $a=0$ and \eqref{t22bb}
show that \as, $g_j(x)=0$ for $\tfm_n$-\aex{}
$x$, which is the same as $\tfm_n|g_j|=0$.
Hence, also
\ref{T230b}$\implies$\ref{T230p}.

\ref{T230p}$\implies$\ref{T230g}: Trivial.

\ref{T230g}$\implies$\ref{T230o}:
Now suppose $\fm_n g_j=0$ \as, for every $n\ge0$.
Then $\tfm_n g_j=0$ \as, and, using \eqref{tag0}, 
\begin{align}
  \label{az4}
\fv_n g_j = \tfm_n g_j -\nu g_j =0 
\quad\text{a.s.}
\end{align}
Furthermore, $g_j$ is an eigenfunction of $\RR$ by \eqref{t22bb},
and thus also an eigenfunction of $B_{0,n}=b_{0,n}(\RR)$.
Hence, \eqref{az4} implies 
\begin{align}
  \label{az5}
\fv_n B_{0,n}\qw g_j = 0
\quad\text{a.s.}
\end{align}
for every $n\ge0$.
By \refL{LIH1}, we have 
$\fv_n B_{0,n}\qw g_j \asto \gL_{g_j}$. Thus \eqref{az5} implies
$\gL_{g_j}=0$ a.s., which by \eqref{az} yields $\gL_j=0$ \as{} and thus shows
\ref{T230o}.

\ref{T230o}$\implies$\ref{T230b}: Trivial.

\ref{T230g}$\iff$\ref{T230h}: Obvious by \eqref{mn}.

\ref{T230p}$\iff$\ref{T230q}: Trivial.

\ref{T230p}$\iff$\ref{T230r}:
\ref{T230p} is equivalent to
\begin{align}\label{tag8}
  \E\tfm_n|g_j|=0
\quad\text{for every $n\ge0$}
.\end{align}
By \eqref{annika} and induction, recalling \eqref{Bmn} and \eqref{Bbmn},
\begin{align}
  \E \tfm_n = \tfm_0B_{0,n}  = \tfm_0b_{0,n}(\RR).
\end{align}
Since $b_{0,n}(\RR)$ is a polynomial in $\RR$ of degree (exactly) $n$, 
\eqref{tag8} is equivalent to \ref{T230r}, which shows the equivalence
\ref{T230p}$\iff$\ref{T230r}.

\ref{T230p}$\implies$\ref{T230nu} when $\RR$ is \slqc:
\refT{T1} then applies to all functions in $\BW$, and in particular to $|g_j|$.
Hence, 
$\tfm_n|g_j|\asto \nu|g_j|$.
The condition \ref{T230p} thus implies $\nu|g_j|=0$,
which is \ref{T230nu}.

\ref{T230nu}$\implies$\ref{T230r} when 
$\fm_0$ is absolutely continuous w.r.t.\ $\nu$:
By \refL{Las} and induction, 
 \ref{T230nu} implies $\RR^n |g_j|=0$  \nuae{} for every $n\ge0$.
Our assumption $\fm_0\ll\nu$ then yields \ref{T230r}.
\end{proof}

\begin{example}\label{Enull}
  Let $E=\oi$, and let $\mu$ be the Lebesgue measure on $E$.
Let $0<\gth<1$ and let
$\RA_x$ be the (non-random) replacement kernel given by
\begin{align}\label{enull}
  \RA_x=\ER_x=
  \begin{cases}
    \mu,&x\neq0
\\
\gth \gd_0 + (1-\gth)\mu, &x=0.
  \end{cases}
\end{align}
We take $W=V=1$, and it is trivial to verify \BHNa, with $\nu=\mu$.
The operator $\RR$ (considered on $B(E)$ as usual)  has rank 2
and it is easily seen that $\gs(\RR)=\set{0,\gth,1}$, with the spectral
projections $\Pi_1$ and $\Pi_\gth$ both having rank 1 and corresponding
eigenvectors 
$1$ and $\etta_{\set0}$. Hence $\RR$ is always \slqc{} on $\BW$, and small if and
only if $\gth<\half$; moreover, our parameter $\gth$ is as in \eqref{gth}.

If we start with $\fm_0=\mu$, or with $\gd_x$ for any $x\neq0$, then 
\as{} $\fm_n=\fm_0+n\mu$, so the \mvpp{} is deterministic.
However, if we take $\fm_0=\gd_0$, then the evolution is different;
the \mvpp{} then is essentially a triangular urn of the type considered in 
\eg{} \cite{SJ169}, where its asymptotic distribution can be found.
(To see this, call colour 0 'white' and lump all other colours in $E$
together as 'non-white'.)

In particular, if $\half<\gth<1$, then \refT{T2}(3) applies with $D=B(E)$,
$p=1$ and $\gl_1=\gth$.
Moreover, if we take $f:=\etta_{\set0}$, 
then \refT{T230} applies with $g_1=f=\etta_{\set0}$.
It follows easily that the limit $\gL_1=0$ in \eqref{hildegran} 
if and only if $\fm_0\set0=0$.

This, admittedly artifical, example shows that one cannot always
ignore functions that are \nuae{} 0; thus  
some care may be required when considering $\RR$ as acting on $L^\infty(E,\nu)$.
\end{example}

\begin{example}\label{Enull2}
  We may vary \refE{Enull} by fixing 3 distinct points
  $x_0,x_1,x_2\in\oi$ and defining (non-random)
\begin{align}\label{enull2}
  \RA_x=\ER_x=
  \begin{cases}
\frac12\gd_{x_1}+\frac12\gd_{x_2}, & x=x_0,
\\
\gth \gd_x + (1-\gth)\mu, &x=x_1, x_2,
\\
    \mu,&\text{otherwise}.
  \end{cases}
\end{align}
The spectrum is still \set{0,\gth,1}, and 
the range of the spectral projection $\Pi_\gth$ has dimension 2.
Let $\frac12<\gth<1$, so that \refT{T2}(3) applies.
One can easily check that $f:=\etta_{x_1}-\etta_{x_2}\in\Pi_\gth$
and that \refT{T230} applies with $g_1=f$.
Using \refT{T230}\ref{T230r}, it follows that $\gL_1=0$ if and
only if $\fm_0\set{x_0,x_1,x_2}=0$. 
In particular, note that $\gL_1$ is non-random if $\fm_0=\gd_{x_0}$;
this shows
that in \refT{T230}, it is not enough to assume $\fm_0|g_j|=0$.
\end{example}

\section{Examples}\label{Sex}

We consider some examples, in separate subsections.

\subsection{Out-degree distribution in the random recursive tree}\label{SSRRT}
This example is already considered by~\cite{SJ155} in the P\'olya urn
context and in~\cite{MV19} in the {\sc mvpp} context.
The random recursive tree is built recursively as follows: at time~1 the tree has one node, its root, and, at every discrete time-step, 
we add one node to the tree, and this new node chooses its parent uniformly at random among the nodes that are already in the tree.
The out-degree of a node is its number of children. 
For all $n\geq 1$ and $k\geq 0$, 
we set $U_k(n)$ the number of vertices of out-degree $k$ in the $n$-node
{\sc rrt},
and
\begin{align}{\frak m}_n := \sum_{k\geq 0} U_k(n)\delta_k
.\end{align}
(We start this process at time $n=1$; this is just a matter of notation.)
We show that
\begin{proposition}\label{prop:RRT}
If $f : \mathbb N_0\to \mathbb C$ satisfies  
$f(k)= O\bigpar{r^k}$ 
for some $r<\sqrt2$, then
there exists a covariance matrix $\Sigmaq(f)$ such that, 
as $n\to+\infty$,
\begin{align}n^{\nicefrac12}\sum_{k\geq 0} \Big(\frac{U_k(n)}n - 2^{-k-1}\Big)f(k)
\dto \mathcal N\big(0, \Sigmaq(f)\big).
\end{align}
\end{proposition}

We show how to calculate $\Sigmaq(f)$ at the end of the section, at least
in some cases.

\begin{remark}
We compare Proposition~\ref{prop:RRT} to the results of~\cite{SJ155}
and~\cite{MV19}. 
The results in~\cite{SJ155} give an equivalent of Proposition~\ref{prop:RRT}
but only for functions~$f$ with finite support.
The results of~\cite{MV19} apply to unbounded functions $f$ as long as they are negligible in front of $x\mapsto 2^{x-\varepsilon}$ for some $\varepsilon>0$. This class of functions is larger than the one in Proposition~\ref{prop:RRT}, 
but~\cite{MV19} proves \as{} convergence of $\frac1n\sum_{k\geq 0}
U_k(n) f(k)$ to $\sum_{k\geq 0}2^{-k-1}f(k)$ while
Proposition~\ref{prop:RRT} gives the fluctuations around this almost-sure
limit. 
\end{remark}

\begin{proof}
To prove this proposition, first note that $({\frak m}_n)_{n\geq 1}$ is an {\sc mvpp} 
with $E=\bbNo:=\set{0,1,2,\dots}$, and deterministic $R\nn=\RA=\ER$
such that, for all $k\geq 0$,
\begin{align}\label{RRT}
\RA_k
=\gd_{k+1}-\gd_k+\gd_0,
\qquad k\ge0.
\end{align}
Note that this includes subtracting the drawn ball $k$ (unless $k=0$).
In other words, the operator $\RR$ is defined by \eqref{RR} as
\begin{align}\label{RRT1}
  \RR f(k) = \ER_k f = f(k+1)-f(k)+f(0),
\qquad k\ge0.
\end{align}
Dually, 
\begin{align}\label{RRT2}
  \gd_k\RR 
=\ER_k
=\gd_{k+1}-\gd_k+\gd_0,
\end{align}
and thus,
for any complex measure $\mu$ on $\bbNo$, (with $\mu\set{-1}:=0$)
\begin{align}\label{RRT3}
  (\mu \RR)(k) = \mu\set{k-1} - \mu\set{k}+\etta_{k=0}(\mu1).
\end{align}

The urn is balanced by \eqref{RRT}, i.e., \ref{as:balance} holds.

We first choose $W=V=1$.
Then \ref{as:BW} holds since $\norm{\RA_k}\le 3$ for every $k\ge0$, see \refR{RW=1}.
Furthermore, it is easily checked from \eqref{RRT3} that the probability
measure 
\begin{align}\label{RRTnu}
  \nu\set k=2^{-k-1},
\qquad k\in\bbNo,
\end{align}
(i.e., a geometric distribution $\Ge(1/2)$)
is an eigenvector satisfying $\nu\RR=\nu$, and thus \ref{as:nu2} holds too.

We next show that $\RR$ is a small operator on $B(W)=B(E)$.
To do so, we show first that the dual operator $\Rx$
is a small operator on $\cM(E)$; 
recall that $\Rx$ is the operator in \eqref{RRT2}--\eqref{RRT3} which we
there, as usually, denote by $\RR$ (acting on the right).

The space $\cME$ of complex measures is naturally identified with $\ell^1$;
we also identify it with the space 
\begin{align}\label{cA}
 \cA:=\Bigset{\sumko a_k z^k:\sumko |a_k|<\infty}
\end{align}
of analytic functions. (The functions in $\mathcal A$ are thus the analytic functions in
the unit disc with a Taylor series that is absolutely convergent on the
closed unit disc.)
The identification is the obvious one, mapping a measure $\mu\in\cME$ to
$\sumko\mu\set{k} z^k$.
Note that $\mathcal A$ is a Banach algebra under pointwise multiplication.
(The norm in $\cA$ is inherited from $\cME=\ell^1$.)

The operator $\Rx$ acting on $\cME$ by \eqref{RRT2} corresponds to the
operator
$\Rq:\mathcal A\to \mathcal A$ given by
\begin{align}\label{Rq}
  \Rq z^k = z^{k+1}-z^k+1.
\end{align}
This means that, for all $f\in \mathcal A$,
\cf{} \eqref{RRT3},
\begin{align}\label{ra}
  \Rq f(z)=zf(z)-f(z)+f(1)
=(z-1)f(z)+f(1).
\end{align}

We first show that
\begin{align}\label{Rqgs}
\gs(\Rq)\subseteq\set{\gl:|\gl+1|\le1}\cup\set1;
\end{align}
this implies that \ref{ergo2} holds.
Fix $\lambda\in\mathbb C$ such that $|\gl+1|>1$ and $\gl\neq1$; our aim is to show that $\lambda\notin\sigma(\Rq)$, 
i.e.\ $\lambda\in\rho(\Rq)$.
To do so, we fix $g\in \mathcal A$ and
consider the equation $(\gl-\Rq)f=g$.
By \eqref{ra}, the equation can be written
\begin{align}\label{rb}
  (1+\gl-z)f(z)-f(1)=g(z).
\end{align}
In particular, taking $z=1$ yields
\begin{align}\label{rc}
(\gl-1)f(1)=g(1).  
\end{align}
Then, \eqref{rc} gives $f(1)=g(1)/(\gl-1)$, and
\eqref{rb} is solved (uniquely) by
\begin{align}\label{rd}
f(z)=    \frac{g(z)+f(1)}{1+\gl-z}
=    \frac{g(z)+g(1)/(\gl-1)}{1+\gl-z}.
\end{align}
Furthermore, this solution $f$ belongs to $\mathcal A$, since $1/(1+\gl-z)\in \mathcal A$ 
when $|\gl+1|>1$
and $\mathcal A$ is a  Banach algebra.
Hence, $(\gl-\Rq)f=g$ has a unique solution $f\in \mathcal A$ for every $g\in \mathcal A$; in
other words, $\gl\in\rho(\Rq)$, which concludes the proof of~\eqref{Rqgs} and thus of~\ref{ergo2}.

Furthermore, the resolvent $(\gl-\Rq)\qw g$ is given by \eqref{rd}, and
thus, by \cite[Equation~VII.6.9]{Conway},
the spectral projection $\Pi_1$ is given by
\begin{align}\label{rf}
 \Pi_1g(z) = \frac{1}{2\pi\ii}\oint_\gG(\gl-\Rq)\qw g(z)\dd \gl
 = \frac{1}{2\pi\ii}\oint_\gG   \frac{g(z)+g(1)/(\gl-1)}{1+\gl-z}\dd \gl,
\qquad |z|\le1,
\end{align}
where $\Gamma$ is a small circle  around~1. (Any circle of radius less
that~1 will do.)
If $|z|\le1$, then  $1/(1+\gl-z)$ is an analytic function of $\gl$ on and
inside $\gG$, and it follows 
by the residue theorem that
the integral \eqref{rf} equals the residue at $\gl=1$, which is
$g(1)/(2-z)$.
Thus,
\begin{align}\label{rg}
 \Pi_1g(z) = \frac{g(1)}{2-z},
\qquad |z|\le1,
\end{align}
which together with \eqref{Rqgs}
shows that \ref{ergo3} holds with the eigenfunction 
$1/(2-z)=\sumko 2^{-k-1}z^k$.  
This eigenfunction corresponds to 
$\nu$ in \eqref{RRTnu}, which shows again that $\nu\RR=\nu$.

Therefore, $\Rq$ is  \aslqc{} operator on $\cA$.
Furthermore, we conclude from
\eqref{Rqgs}   that it is a small operator on
$\mathcal A$, 
and thus that $\Rx$ is a small operator on $\cME$. 
By \refC{Cgs}, with $\cX=\BE$ and $\cY=\cME$,
this implies that $\RR$ is a small operator on $\BE$.

We have verified the conditions of  Theorem~\ref{T2}(1), 
which thus applies and shows asymptotic normality of $\fm_n f$
as in \eqref{t11} for every $f\in\BW=B(E)$.

We can extend the range of this result by considering other functions $W$.
Fix $r\ge1$ and take now 
\begin{align}
W(k)=  W_r(k):=r^k.
\end{align}
Thus $V(k)=W(k)^q=r^{qk}$ for some $q>2$.
Recall that \ref{as:nu2} requires $\nu V<\infty$.
Since $\nu$ still is given by \eqref{RRTnu},
this is equivalent to $r^q<2$.
Similarly,
\eqref{RRT1} shows that
\begin{align}
  \ER_kV=V(k+1)-V(k)+V(0) = (r^q-1)V(k) +1,
\end{align}
and thus \ref{as:BWi} holds if and only if $r^q<2$.
It is easily seen that \ref{as:BWii} holds for every $r\ge1$.
Furthermore, the urn starts with the composition $\gd_0$, and thus
\ref{as:BWiii} is trivial.
Hence, \ref{as:BW} and \ref{as:nu2} both hold if and only if $r^q<2$.
Since 
\ref{as:balance} holds regardless of $W$,
we conclude that
\begin{align}\label{RRTBHN}
  \text{\BHNa{} hold for some  $q>2$} \iff r<\sqrt2.
\end{align}

We now have to find the spectral gap of $\RR$ as an operator on $B(W_r)$.
We argue as in the case $r=1$ above, and begin by noting that 
$\cM(W_r)=\bigset{\mu:\sum_0^\infty |\mu\set{k}|r^k<\infty}$ is a
norm-determining subspace of $B(W_r)^*$.
Moreover, $\cM(W_r)$ may be identified with the space
\begin{align}\label{cAr}
 \cA_r:=\Bigset{\sumko a_k z^k:\sumko |a_k|r^k<\infty}
\end{align}
of analytic functions. The functions in $\cA_r$ are 
continuous in the closed disc $\set{z:|z|\le r}$ 
and analytic in its interior.
$\cA_r$ is, as $\cA=\cA_1$ studied above, 
a Banach algebra under pointwise multiplication.

As in the case $r=1$, the operator $\Rx$ on $\cM(W_r)$ corresponds to an
operator $\Rq$ on $\cA_r$ given by \eqref{Rq} and \eqref{ra}.
The argument above then shows that $\gl\in\rho(\Rq)$ provided
$\gl\neq1$ and $1/(1+\gl-z)\in\cA_r$, \ie, if $\gl\neq1$ and $|1+\gl|>r$.
Consequently, \eqref{Rqgs} is replaced by
\begin{align}\label{Rqgsr1}
\gs(\Rx)\subseteq\set{\gl:|\gl+1|\le r}\cup\set1.
\end{align}
Hence, on $B(W_r)$,
using \refL{Lgs1}\ref{Lgs1b} and
with $\gs(\Rx)\sphatx$ defined in Definition~\ref{def:Kchapeau},
\begin{align}\label{Rqgsr2}
\gs(\RR)
\subseteq\gs(\Rx)\sphatx
\subseteq\set{\gl:|\gl+1|\le r}\cup\set1.
\end{align}
In particular,
\begin{align}\label{mon}
  \gth_{B(W_r)}\le r-1.
\end{align}
We have seen above that we have to take $r<\sqrt2$ in order to have \ref{as:BW} and
\ref{as:nu2}, and \eqref{mon} shows that in this case $\gth<\half$ follows.
Consequently, if $r<\sqrt2$, then
the asymptotic normality \eqref{t11} extends to all $f\in
B(W_r)$, \ie, all $f$ such that $f(k)=O(r^k)$.
This completes the proof.
\end{proof}

\begin{remark}\label{Rgs}
It is easy to see that we have equality in \eqref{Rqgsr2}--\eqref{mon}.
In fact, we know that $1$ is an eigenvalue by
\eqref{R1}. Moreover, if $|1+\gl|\le r$ and $\gl\neq0$, then
$f(k):=(1+\gl)^k+1/(\gl-1)$ 
satisfies $f\in B(W_r)$ and $\RR f = \gl f$ by \eqref{RRT1}, 
see also \eqref{eva} below,
so 
$\gl$ is an eigenvalue of $\RR$ and thus $\gl\in\gs(\RR)$. 
Hence, we have equality in 
\eqref{Rqgsr2}--\eqref{mon} too. (For $\gl=0$ and $r>1$,
$f(k)=k-1$ is an eigenfunction, but this case follows also because
$\gs(\RR)$ is closed.)
Consequently, $\RR$ is a small operator in $B(W_r)$ $\iff r<3/2$.
\end{remark}

In the rest of this subsection, we show how to calculate the
  asymptotic covariance matrix $\Sigmaq(f)$ in \refP{prop:RRT}
for the following  functions~$f$:
Fix $r\in [1,\sqrt{2})$ and
let, for a complex $a$ with $|a|\le r$,
\begin{align}
  f_a(k):=a^k.
\end{align}
Then $f_a\in B(W_r)$, $\nu f_a = \sumko 2^{-k-1}a^k = 1/(2-a)$, and, by
\eqref{RRT1}, 
\begin{align}\label{eva}
  \RR(f_a-\nu f_a)=a^{k+1}-a^k+1-\frac{1}{2-a}
=a^{k+1}-a^k+\frac{1-a}{2-a}
=(a-1)(f_a-\nu f_a).
\end{align}
In other words, provided $a\neq1$ (so the function does not vanish),
$\tf_a:=f_a-\nu f_a\in B(W_r)$ is an eigenfunction of $\RR$ with
eigenvalue $a-1$. 
This makes it easy to compute asymptotic variances and covariances 
in \refT{T2} for the functions $f_a$.

Let $a$ and $b$ be complex numbers with $|a|,|b|\le r$.
First, note that by \eqref{BBB}, 
since $\RA_x=\ER_x$ is deterministic,
\begin{align}\label{rickard}
  \BBB(\tf_a,\tf_b)&
 =\ER_{\bigdot} \tf_a\cdot \ER_\bigdot \tf_b
=(\RR \tf_a)\cdot (\RR \tf_b)
= (a-1)(b-1)\tf_a\tf_b
\notag\\&
=(a-1)(b-1)\bigpar{f_{ab}-(\nu f_a) f_b -(\nu f_b) f_a + (\nu f_a)(\nu f_b)}.
\end{align}
Hence,
\begin{align}\label{karin}
\nu \BBB(\tf_a,\tf_b)&
=(a-1)(b-1)\bigpar{\nu f_{ab}- (\nu f_a)(\nu f_b)}
\notag\\&
=(a-1)(b-1) \Bigpar{\frac{1}{2-ab}-\frac{1}{(2-a)(2-b)}} 
=\frac{2(a-1)^2(b-1)^2}{(2-ab)(2-a)(2-b)}
\end{align}
and thus, recalling again \eqref{eva},
\begin{align}\label{kain}
  \intoo \nu\BBB\bigpar{\ee^{s\RR}\tf_a,\ee^{s\RR}\tf_b}\ee^{-s}\dd s
&=
  \intoo \nu\BBB\bigpar{\ee^{s(a-1)}\tf_a,\ee^{s(b-1)}\tf_b}\ee^{-s}\dd s
\notag\\&
=
  \intoo \nu\BBB\bigpar{\tf_a,\tf_b}\ee^{-(3-a-b)s}\dd s
\notag\\&
=\frac{2(a-1)^2(b-1)^2}{(3-a-b)(2-ab)(2-a)(2-b)}
.\end{align}
Taking $b=a$ in \eqref{kain} gives $\chi(f_a)$ in \eqref{goxf},
and taking $b=\bar a$ gives $\gss(f_a)$ in \eqref{gssf}.
(See \refR{RBB}).
In particular, for $a$ real with $|a|<\sqrt2$, 
\refT{T2} shows (see \refR{RcomplexG}) that 
\begin{align}\label{mag}
  n^{-1/2}\biggpar{\sumko U_k(n)a^k - \frac{n}{2-a}}
= n\qq\bigpar{\tfm_nf_a-\nu f_a}
\dto \cN\bigpar{0,\gss(f_a)},
\end{align}
with
\begin{align}\label{dal}
  \gss(f_a)= \frac{2(a-1)^4}{(3-2a)(2-a^2)(2-a)^2}.
\end{align}
More generally, we have joint convergence for several (real or complex) $a$,
with asymptotic covariances easily found from \eqref{kain}.

\begin{remark}\label{RU}
  It follows that the asymptotic variances and covariances of $n\qqw U_k(n)$
can be obtained as Taylor coefficients of the 
bivariate rational function in \eqref{kain};
this was earlier shown in \cite{SJ155} by related calculations using urns
with finitely many colours.
\end{remark}

\begin{remark}\label{Rvar}
  Moreover, using Fourier analysis, any function $f$ in \refP{prop:RRT}
may be expressed as an integral of functions $f_a$:
for any $\rho\in(r,\sqrt2)$,
\begin{align}\label{uarda}
  f = \frac{1}{2\pi}\int_0^{2\pi}\hf\bigpar{\rho\qw e^{-\ii t}}
f_{\rho  e^{\ii t}}\dd t
\end{align}
where $\hf(z):=\sumko f(k)z^k$.
By substituting \eqref{uarda} in \eqref{goxf} and \eqref{gssf}, and using
\eqref{kain}, one can obtain integral formulas for $\chi(f)$ and $\gss(f)$,
and thus for $\Sigmaq(f)$. The result is rather complicated, however, and we
leave the details to the reader.
\end{remark}

\begin{remark}\label{Rbest}
 The asymptotic variance in \eqref{dal} diverges as $a\upto \sqrt2$,
and thus the result cannot be extended (in this form at least) to $a\ge\sqrt2$.
Hence, the condition $r<\sqrt2$ in \refP{prop:RRT} and the argument above is
not just a 
technical condition required by our proofs; it is essential for 
\eqref{mag}--\eqref{dal}, which strongly suggests that it is necessary
in \refP{prop:RRT} too. This also shows that the technical conditions \ref{as:BW}
and \ref{as:nu2} are more or less best possible; in particular, it is not enough to
take $q<2$ in \ref{as:BW}. 
\end{remark}

We do not know what happens for functions $f$ that grow faster than allowed
in \refP{prop:RRT}. In particular, the following case seems interesting.

\begin{problem}\label{PRRT}
  What is the asymptotic distribution of $\sumko U_k(n) a^k$ for
  $a\ge\sqrt2$?

Is there any difference between the cases $a<3/2$ and $a>3/2$?
(Recall that $\RR$ is a small operator in $B(W_a)$ for $a<3/2$, 
but not for larger $a$.)
\end{problem}

\subsection{The heat kernel on the square}\label{SSheat}
Imagine some flowers planted in a closed square room:
we start with one flower in the room (say at the centre of the room).
Each flower blooms at exponential rate, 
independently from the others, and when a flower blooms, 
it sends one seed in the air, which travels in the air according to a
Brownian motion
reflected at the walls for a unit-time, then fall onto the ground and instantly becomes a new flower.
We assume that the rate of blooming is so small that we can imagine that 
the seeds perform their unit-one Brownian motions instantly.
We set $\tau_n$ to be the instant of the $n$-th bloom ($\tau_0:=0$),
and $\xi_n$ to be the position of the $n$-th flower in $[0,\ell]^2$ ($\xi_0 = (\nicefrac\ell2, \nicefrac\ell2)$).
We are interested in the long-term behaviour of the distribution of flowers in the room:
\begin{align}\Xi_n = \sum_{i\geq 0} \delta_{\xi_i}.
\end{align}
It is expected that $\Xi_n/n$ converges to the uniform distribution on the square, and this is indeed confirmed by Theorem~\ref{T1}\ref{T1b}; 
Theorem~\ref{T2} allows to study the fluctuations around this limit.
This yields the following.
\begin{proposition}\label{prop:flowers}
For all bounded measurable functions $f : [0, \ell]^2\mapsto \mathbb R$, 
\begin{equation}\label{eq:cvas_flowers}
\frac1n\Xi_n f 
=\frac1n\sumion  f(\xi_i)
\to 
\frac{1}{\ell^2}
\int_{[0,\ell]^2} f(x)\,\mathrm dx,\quad\text{ almost surely when }n\to+\infty.
\end{equation}
For all  $m, p \in \mathbb N_0^2$,  set
\begin{equation}\label{eq:lambdas}
\lambda_{m,p} := \exp\left(-\frac{\pi^2 (m^2 + p^2)}{2\ell^2}\right),
\end{equation}
and 
\begin{equation}\label{eq:phis}
\varphi_{m, p}(x,y) 
:= \cos\Big(\frac{\pi m x}{\ell}\Big)\cos \Big(\frac{\pi p y}{\ell}\Big).
\end{equation}
Also, set $I(\ell) := \{(m,p)\in\mathbb N_0^2\colon  \lambda_{m,p}<\nicefrac12\}$
and let $D$ be the closed linear span  in $B([0,\ell]^2)$ of\/ $1$ and
$\{\varphi_{m,p}\colon (m,p)\in I(\ell)\}$. 
Similarly, set $J(\ell) := \{(m,p)\in\mathbb N_0^2\colon  \lambda_{m,p}\leq
\nicefrac12\}$ 
and let $D'$ be the closed linear span  of\/ $1$ and
$\{\varphi_{m,p}\colon (m,p)\in J(\ell)\}$. 
\begin{romenumerate}
\item\label{Pfla} 
For every function $f\in D$,
there exists a covariance matrix $\Sigmaq(f)$ such that
\begin{align}\label{pfla}
n^{\nicefrac12} \lrpar{\frac1n\sumion  f(\xi_i)
-\frac{1}{\ell^2}\int_{[0,\ell]^2} f(x,y)\dd x\dd y}
\to \mathcal N\bigpar{0, \Sigmaq(f)},
\end{align}
in distribution  
as $n\to+\infty$. 
\item\label{Pflb} 
If\/ $\frac{2\log 2}{\pi^2}\ell^2\in\{m^2+p^2\colon (m,p)\in \mathbb
N_0^2\}$, so $J(\ell)\neq I(\ell)$, then 
for every function $f\in D'$,
there exists a covariance matrix $\Sigmaq(f)$ such that
\begin{align}\label{pflb}
\frac{n^{\nicefrac12}}{(\log n)^\half} 
\lrpar{\frac1n\sumion f(\xi_i)-\frac{1}{\ell^2}\int_{[0,\ell]^2} f(x,y)\dd x\dd y}
\to \mathcal N\bigpar{0, \Sigmaq(f)},
\end{align}
in distribution 
as $n\to+\infty$.

\item\label{Pflc} 
If $\ell>\pi/{\sqrt{2\log 2}}$, then
for every function $f\in B([0,\ell]^2)$, there exists a random variable
$W(f)$ such that 
\begin{align}\label{pflc}
n^{1-\exp(-\pi^2/2\ell^2)}
\lrpar{\frac1n\sumion f(\xi_i)-\frac{1}{\ell^2}\int_{[0,\ell]^2} f(x,y)\dd x\dd y}
\to W(f),
\end{align}
almost surely and in $L^2$ when $n\to+\infty$.
\end{romenumerate}
\end{proposition}

\begin{remark}
If $\ell<\pi/{\sqrt{2\log 2}}$, then $D = B([0,\ell]^2)$, and then
\ref{Pfla} applies to all bounded $f$.
Similarly,
if $\ell=\pi/{\sqrt{2\log 2}}$, then $D'= B([0,\ell]^2)$ and \ref{Pflb}
applies
to all bounded $f$.
\end{remark}

\begin{proof}
First note that $\Xi_n$ is an {\sc mvpp} with colour space $E = [0,\ell]^2$,
initial composition $\delta_{(\nicefrac\ell2,\nicefrac\ell2)}$, and random
replacement kernel 
\begin{align}\label{fl1}
\RA_x = \delta_{B^{\sss (x)}_1},
\end{align} 
where $B = (B_t)_{t\geq 0}$ is the standard Brownian motion on the square 
of side-length $\ell$ started at $B^{\sss (x)}_0 = x$ and reflected at the
boundary. 
Note that $\RA_x$ is a positive measure.
We have
\begin{align} \label{fl2}
\ER_x = \cL\bigpar{B_1^{\sss (x)}},
\end{align}
the distribution of the reflected Brownian motion.
Hence, for any probability measure $\mu$ on $E$,
\begin{align}\label{fl3}
  \mu\RR=\cL\bigpar{B_1^{\sss \mu}},
\end{align}
the distribution of the reflecting Brownian motion at time 1 when started
according to $\mu$.

This {\sc mvpp} satisfies Assumption \ref{as:balance}.
We choose $W=V=1$, and then \ref{as:BW} holds by \refR{RW=1}.
Furthermore,  \ref{as:nu2} holds because the uniform distribution $\nu$ on
$[0,\ell]^2$ is invariant for the reflected Brownian motion and thus
satisfies $\nu{\RR} = \nu$ by \eqref{fl3}. 

The kernel $\ER_x$ in \eqref{fl2} of $\RR$ 
is known as the {\it heat kernel with Neumann boundary conditions}.
Its eigenvalues and eigenfunctions are well known, and can be found \eg{} 
as follows. (We give a sketch, omitting the standard details.)
First, since the kernel is absolutely continuous, and depends continuously
on $x$, it is easily seen that it does not matter whether we consider $\RR$ as
an operator on $B(E)$ or $L^\infty(E)$. 
(See \refL{LXZ}, with $\cN$ the space of bounded functions
that are 0 a.e.) 
Furthermore, the density of $\ER_x$ 
is bounded, uniformly in $x$, and it follows that $\RR$ maps $L^2(E)$ into
$L^\infty(E)$. Hence, \refL{LXY} shows that eigenvalues and other spectral
properties are the same in $L^\infty(E)$ and in $L^2(E)$ (except possibly at
0, which is not important for us).
Finally, we regard $L^2(E)=L^2([0,\ell]^2)$ as the subspace of
$L^2([-\ell,\ell]^2)$ 
consisting of functions that are even in each variable, and then 
extend these functions periodically to $\bbR^2$.
We then can replace the reflecting Brownian motion by ordinary Brownian
motion on $\bbR^2$, and it follows that the functions $\gf_{m,p}$ in
\eqref{eq:phis} form a complete orthogonal set of eigenfunctions in $L^2(E)$,
with corresponding eigenvalues $\gl_{m,p}$ given by \eqref{eq:lambdas}.
(In this example, $\RR$ is a self-adjoint operator on $L^2$, which makes the
spectral theory in $L^2$ particularly simple.)

Since $\gl_{m,p}\to0$ as $m+p\to\infty$, it follows that
\begin{align}\label{fl4}
\gs(\RR)=\set{\gl_{m,p}:m,p\in\bbNo}\cup\set0,   
\end{align}
in $L^2(E)$, and by \refL{LXY} as indicated above,
also in $L^\infty(E)$ and in $B(E)$.

The eigenvalue 1 is obtained only for $m=p=0$, and 
thus it follows from \eqref{fl4} that $\RR$ is \slqc.
Moreover, the second largest eigenvalue is
$\gl_{1,0}=\gl_{0,1}=\exp\bigpar{-\pi^2/(2\ell^2)}$, and thus $\RR$ is small
if and only if $\pi^2/(2\ell^2)>\log 2$, \ie, if 
$\ell<\pi/\sqrt{2\log 2}$.

The almost sure convergence in~\eqref{eq:cvas_flowers} is thus a direct
consequence of Theorem~\ref{T1}\ref{T1a}, which also gives an (upper)
estimate of the rate.

Next, we show that
\begin{align}\label{fl6}
\gs(\RR_D)
=\set{\gl_{m,p}:(m,p)\in I(\ell)\cup\set{(0,0)}}
=\set{\gl_{m,p}:\gl_{m,p}<\half}\cup\set1.
\end{align}
To see this, we first note that
if $\hD$ is the closure of $D$ in $L^2(E)$, \ie, 
the closed linear span  in $L^2(E)$ of\/ $1$ and
$\{\varphi_{m,p}\colon (m,p)\in I(\ell)\}$,
then $\gs(\RR_{\hD})$ is given by \eqref{fl6}, since the functions
$\gf_{m,p}$ are orthogonal eigenfunctions. 
Then, \eqref{fl6} follows by \refL{LXY}, because $\RR:\hD\to D$.

It follows from \eqref{fl6} that $\RR_D$ is a small operator, and thus
\ref{Pfla} 
is a direct consequence of~Theorem~\ref{T2}(1).

Similarly, by the same argument,
\begin{align}\label{fl6'}
\gs(\RR_D)
=\set{\gl_{m,p}:(m,p)\in J(\ell)\cup\set{(0,0)}}
=\set{\gl_{m,p}:\gl_{m,p}\le\half}\cup\set1.
\end{align}
and
\ref{Pflb} follows from Theorem~\ref{T2}(2), 
with $p=1$, $\gl_1=\frac12$, and $\kk=\kk_1=1$.

Finally, \ref{Pflc}
follows from Theorem~\ref{T2}(3),
with
$p=1$, $\gl_1=\ee^{-\pi^2/(2\ell^2)}$, and $\kk=\kk_1=1$.
\end{proof}

\begin{remark}\label{RflFourier}
  The covariance matrices of the limits in \eqref{pfla} and \eqref{pflb}
can easily be computed from the formulas in \refT{T2}
and a Fourier expansion of  $f$ into the functions $\gf_{m,p}$; we leave the
details to the reader.
\end{remark}

We can use \refTs{T210}--\ref{T230} to see whether the limit distributions
in \refP{prop:flowers} are degenerate.
Note that if $\gl\neq0$, then $\Pi_\gl$ is a projection onto a
finite-dimensional 
space spanned by some $\gf_{m,p}$; these are all continuous, and thus
$\Pi_\gl f$ is continuous for any $f\in B(E)$.

First, for \ref{Pfla}, it is easily seen from \refT{T210} that the limit in
\eqref{pfla} is degenerate only if $f=c$ \aex{} for some constant $c$.

Secondly, for \ref{Pflb}, \refT{T220} (with $\kk=1$ and $p=1$) shows that
the limit is degenerate if and only if $\Pi_{1/2}f=0$ a.e.; since 
$\Pi_{1/2}f$ is continuous, this holds if and only if $\Pi_{1/2}f=0$.
It is easily seen
that this holds if and only if $f\in D$ (and thus \ref{Pfla} applies, and
gives a more precise result).

Similarly, for \ref{Pflc}, \refT{T230} shows that the limit is degenerate if
and only if $\Pi_{\gl_1}f=0$, where $\gl_1=e^{-\pi^2/(2\ell^2)}$.
Assume this.
The next largest eigenvalue of $\RR$ is
$\gl_2=e^{-\pi^2/\ell^2}$. 
Hence, if $\gl_2\le\half$, we can apply \ref{Pfla} or \ref{Pflb} to $f$.
If $\gl_2>\half$, we  may instead apply \refT{T2}(3) to the subspace 
$D_1:=(\II-\Pi_{\gl_1})B(E)$;
note that $T$ is \slqc{} in $D_1$ and $\gth_{D_1}=\gl_2=e^{-\pi^2/\ell^2}$.

\begin{remark}
  In this example, the generalized eigenspaces $\Pi_\gl$ ($\gl\neq0$)
are all spanned by eigenvectors. Hence, $\kk=1$ in \refT{T2}, regardless of
the multiplicities of the eigenvalues.
The multiplicities show up 
when considering joint convergence of several $f$, as discussed in
\refR{Rjoint2}. In fact, in \refP{prop:flowers}\ref{Pflc},
the dominating eigenvalue $\gl_{1,0}=\gl_{0,1}$
has multiplicity 2, and thus there is a two-dimensional space of limits.

In \refP{prop:flowers}\ref{Pflb}, the dimension of the space of limits
equals
the multiplicity of the eigenvalue $\half$, which equals the number of
solutions to $m^2+p^2=N:=(2\log2)\ell^2/\pi^2$.
A formula for the number of such solutions is well known (and was stated
already by Gauss),
see \cite[Theorem 278 and Notes p.~243]{HardyW}, as well as a criterion for
the existence of any solutions at all (so $D'\neq D$)
\cite[Theorem 366]{HardyW}.
\end{remark}

\begin{remark}
 We could replace $[0,\ell]^2$ by any finite measure space 
 $(E,\mu)$ and the Brownian motion $B_1^{(x)}$ by jumps
according to any transition kernel $P(x,\dd y)$ on $E$ 
that has a density with respect to $\mu$ that is bounded (or, more
generally, in $L^2(\mu)$), uniformly in $x\in E$.
The operator $\RR$ then maps $L^2(\mu)\to B(E)$.
Moreover, 
$\RR$ is 
a Hilbert--Schmidt integral operator on $(E,\mu)$,
and thus $\RR$ is a compact operator on $L^2(\mu)$.
By the spectral theorm for compact operators,
\cite[Theorem~VII.7.1]{Conway}, 
the spectrum $\gs(\RR)$ can be written as
$\set{\gl_i}_{i=1}^N\cup\set0$ for some $N\le\infty$ and eigenvalues
$\gl_i\neq0$; either $N<\infty$ or $\gl_i\to0$ as $i\to\infty$.
Furthermore, $\Pi_{\gl_i}(L^2(E.\mu))$ has finite dimension for every $\gl_i$.
$\RR$ is a bounded  operator also on $L^\infty(E)$ and $B(E)$,
and
by \refLs{LXY} and \ref{LXZ}, the spectrum of $T$ is the same for these
spaces as for $L^2(E,\mu)$.

The function 1 is an eigenfunction with eigenvalue 1, so $1\in\gs(T)$, and 
$|\gl_i|\le1$ for all~$i$ since $\norm{\RR}_{B(E)}=1$.
In particular, $\RR$ is \slqc{} provided $\Pi_1(\RR)$ does not contain any
non-constant function.
Assuming the latter property, 
we thus obtain the same type of behaviour as in \refP{prop:flowers}.

The main
advantage of choosing the Brownian motion on $E=[0, \ell]^2$ is that 
its spectral decomposition is explicitly known and very simple.
(That the operator is self-adjoint on $L^2$ helps but is not essential.)
Other examples 
for which the spectral decomposition is fully known are the 
reflected Brownian motion
on the rectangle, on the isosceles triangle (see e.g.\ 
\cite[Chapter~5]{Henrot}) or on the annulus (see~\cite{KS84survey} and~\cite{GN13survey} for surveys on eigenfunctions and eigenvalues of the heat kernel). 
\end{remark}

\subsection{A branching random walk}\label{SSBRW}

The following branching random walk is studied in \cite{SJ32}.
Let $G$ be a compact group, and let $(Y_n)_1^\infty$ be an \iid{} sequence
of random variables in $G$ with some distribution $\mu\in\cP(G)$.
Let $X_0\in G$ be given. (In \cite{SJ32}, $X_0$ may be random.
We assume here that $X_0$ is non-random; otherwise
we may condition on $X_0$, \cf{} \refR{Rm0}.)
For $n\ge1$, let $I_n$ be uniformly distributed on \set{0,\dots,n-1} and
assume that all $I_n$ and $Y_m$ are independent.
Then define $X_n\in G$ inductively by
\begin{align}\label{32X}
  X_n:=X_{I_n}Y_n,
\qquad n\ge1.
\end{align}
In other words, for each $n$, we first choose a parent uniformly among
$X_0,\dots,X_{n-1}$, and then let $X_n$ be a daughter with a random
displacement $Y_n$ from its parent.

This process can be regarded as a \mvpp{} with colour
space $E=G$ by defining
\begin{align}\label{mx}
  \fm_n:=\sumion \gd_{X_i}.
\end{align}
The construction of $X_n$ in \eqref{32X} then means that $(\fm_n)_n$ is a
\mvpp{} with replacements given by
\begin{align}
  \RA_x=\gd_{xY_1}, \qquad x\in G.
\end{align}
We choose $W=V=1$, and let $\nu$ be the normalized Haar measure.
The conditions \BHN{} are easily verified.
We have
\begin{align}
  \ER_x=\cL(xY_1),
\end{align}
which is $\mu$ left translated by $x$.
Hence, $\RR$ acts on functions by convolution $\RR f = f*\check \mu$, where
$\check\mu$ is the distribution of $Y\qw$.

The results in \cite{SJ32} are about asymptotic normality, under certain
conditions, of the sums
\begin{align}
  S_n(f):=\sumion f(X_i) =\fm_n f
\end{align}
for suitable functions $f$. (The proof uses the method of moments.)

Consider for simplicity the case when $G$ is commutative. (The case of
non-commutative $G$ is similar but more technical and requires study of the 
irreducible representations of $G$; see \cite{SJ32}.)
Let $\hG$ be the dual group, consisting of all characters on $G$
(\ie, continuous homomorphisms $G\to \set{z\in\bbC:|z|=1}$), and define the
Fourier transform of $\mu$ by
\begin{align}\label{hmu}
  \hmu(\gam) :=\int_G\gam(g)\dd\mu(g)
=\E\gam(Y_1),
\qquad \gam\in\hG.
\end{align}
Then, every character $\gam$ is an eigenfunction of $\RR$, with 
\begin{align}
  \RR\gam=\hmu(\gam)\gam.
\end{align}
Hence, on the Hilbert space $L^2(G)$, $\RR$ has an ON basis of eigenfunctions,
and 
\begin{align}
  \gs(\RR)=\set{\hmu(\gam):\gam\in\hG}.
\end{align}
If we assume (as in \cite{SJ32}) that $\mu$ is not supported on any proper
closed subgroup of $G$, then $\hmu(\gam)\neq1$ and thus $\Re\hmu(\gam)<1$
for every $\gam\neq1$. If we further assume, for example, that $\mu$ is
absolutely continuous w.r.t.\ the Haar measure $\nu$, then $\hmu\in c_0(\hG)$
by (a general version of) the Riemann--Lebesgue lemma, and it follows that
$\RR$ is \slqc{} on $L^2(G)$. 
Moreover, if the density $\ddx\mu/\ddx\nu$ of $\mu$ is in $L^2(G)$,
then $\RR:L^2\to B(G)$, and it follows from
\refLs{LXY} and \ref{LXZ} that $\RR$ is \slqc{} also on $L^\infty(G)$ and on
$B(G)$.

Theorem \ref{T2} then applies and yields asymtotic normality of $S_n(f)$ if 
\begin{align}
\gth:=\sup\set{\Re\hmu(\gam):\gam\neq1}\le\half
;  
\end{align}
this is essentially 
\cite[Theorems 3.1 and 3.2]{SJ32}, although the technical conditions there
on $f$ and $\mu$ are somewhat different from ours. 
(They neither imply or are implied by our
conditions here; an example where \refT{T2} applies but not \cite{SJ32} is
when $\ddx\mu/\ddx\nu\in L^2(G)\setminus L^\infty(G)$, and $f\in
B(G)\setminus C(G)$.)
Moreover, if $\half<\gth<1$, then \refT{T2}(3) applies, and extends the
brief comments  given in \cite{SJ32} for that case.

\begin{remark}\label{Rhomo}
  \cite{SJ32} considers also a generalization to compact homogeneous spaces;
this is treated by constructing a branching random walk as above on a compact
group $G$, and then considering the projection to $G/H$ for a closed
subgroup $H$ of $G$. (This assumes that the distribution $\mu$ is invariant
under left or right multiplication by elements of $H$.)
The space $B(G/H)$ can be identified with a subspace of $B(G)$,
and thus \refTs{T1} and \ref{T2} can be applied in this setting too.
\end{remark}

\begin{remark}
  This example is closely related to the one in \refSS{SSheat}.
In fact, the latter example can,
by identifying $[-\ell,\ell]^2$ with the group $\bbT^2$,  
be treated as a branching random walk as above on
the group $G=\bbT^2$, but considering only the subspace of bounded functions 
that are even in each coordinate.
\end{remark}

\subsection{Reinforced process on a countable state space}
In this section, we 
consider a reinforced process which is a particular case of balanced P\'olya
urn on a countable state space. Let $(X_n)_{n\in\mathbb{Z}_+}$ be an
irreducible Markov chain evolving in a countable state space~$E$, and
denote by $\P_x$ and $\E_x$ the law and expectation of the
process starting from $X_0=x\in E$. 
Similarly, if $\nu\in\cP(E)$, we use
$\P_\nu$ and $\E_\nu$ for  
the Markov chain started with a random $X_0\sim\nu$.
We assume that $X$ admits a Lyapunov type
function: there exist a function
$V:E\to[1,+\infty)$ such that $\{x\in E: V(x)\leq A\}$ is
finite for every $A<\infty$, 
and for some constants $\lambda\in(0,1)$ and $C<\infty$,
\begin{align}
	\label{eq:VLyap}
\E_x\bigsqpar{V(X_1)}\leq \lambda V(x)+C\qquad \text{for all $x\in E$}.
\end{align}

We fix $T\in\{2,3,\ldots\}$ and consider the reinforced process
$Z=(Z_n)_{n\geq 0}$ constructed as follows: $Z_0=z_0\in E$ is fixed and $Z$
evolves according to the dynamic of $X$ up to time $T-1$. At time $T$, it
jumps  to a random position distributed according to its empirical
occupation measure $\frac{1}{T}\sum_{i=0}^{T-1} \delta_{Z_i}$;
in other words, the process jumps back to its position at a uniformly random
earlier time $i\in[0,T)$.
Then $Z$ evolves according to the dynamic of $X$ up to time $2T-1$ and, at
time $2T$, it jumps  to a random position distributed according to its
current empirical occupation measure, and so on. 
(The process thus jumps back to a random earlier position
at times $kT$, $k\in\bbN$.)

Let $\bmu_n:=\frac{1}{n+1}\sum_{i=0}^{n} \delta_{Z_i}$ denote the empirical
occupation measure of~$Z$ at time $n$; i.e.
\begin{align}\label{mun}
\bmu_n f = \frac{1}{n+1}\sum_{i=0}^{n}f(Z_i).
\end{align}
We show that $\bmu_n$
converges almost surely  
(and in a weak $L^2$ sense) 
to the unique invariant distribution of~$X$, 
and that, at least if $T$ is large enough, 
$\bmu_n$ satisfies a central limit theorem.

\begin{proposition}
	\label{prop:examplequasicomp}
The Markov chain $X$ has a unique invariant distribution $\nu$.
Moreover:
\begin{enumerate} 
  \item[{\rm (a)}]
For any
$q>2$,  there exists $\gd=\delta(q)>0$ such that, 
for every $f\in B(V^{\nicefrac1q})$,
	\begin{align}
		\label{eq:L2conv}
    \E \left|\bmu_nf- \nu f\right|^2 = O\bigpar{n^{-2\delta}}.
	\end{align}
    and 
    \begin{align}
    \label{eq:ASconv}
    n^{\delta}\left|\bmu_nf- \nu f\right|\xrightarrow[n\to+\infty]{a.s.} 0.
    \end{align}
    
  \item [\rm (b)]
If in addition $\big(\frac{1}{T}\,\frac{1-\lambda^{T}}{1-\lambda}\big)^{\nicefrac1q}<\nicefrac12$, then, for any $f\in B(V^{\nicefrac1q})$, 
one of the conclusions {\rm (1)}, {\rm (2)} or {\rm (3)} of Theorem~\ref{T2} holds with:
\begin{itemize}
    \item 
$(\nicefrac nT)^{\nicefrac12}\left(\bmu_nf-\nu f\right)$ 
    instead of $n^{\nicefrac12}(\tfm_n f-\nu f)$ in {\rm (1)}, 
    \item 
$\frac{(\nicefrac nT)^{\nicefrac12}}{(\log  n)^{\kappa-\nicefrac12}}\left(\bmu_nf-\nu f\right)$ instead of $\frac{n^{\nicefrac12}}{(\log n)^{\kappa-\nicefrac12}}(\tfm_n f-\nu f)$ in {\rm (2)}, 
    \item  
$\bmu_n f$ instead of $\tfm_n f$ and $\gL_j':=T^{1-\gl_j}\gL_j$ instead of
$\gL_j$ in \eqref{hildegran}.
\end{itemize}

\item[{\rm (c)}]
There exists $T_0=T_0(q)\geq 2$ such that, for any $T\geq T_0$,
conclusion {\rm (1)} of \refT{T2}  holds {for all $f\in B(V^{\nicefrac1q})$}.
\end{enumerate}
\end{proposition}

The proof uses the following lemma, which we prove at the end of this
subsection.
We use the 
notations $r(\RR)$ and $r_e(\RR)$ for the spectral radius and essential spectral
radius of the operator $\RR$; see \refD{Dre}.

\begin{lemma}
	\label{lem:QC}
	Let $\RR$ be 
the operator given by  \eqref{RR} for some
  probability kernel $\ER$ from $E$ to $E$,
and let $V:E\to  [1,+\infty)$ 
be a function such that $\{x\in E: V(x)\leq A\}$ is finite for every  
$A<\infty$.
	If there exist $\vartheta<1$ and $C<\infty$ such that 
    \begin{align}\label{lqc}
\RR V\leq \vartheta V +C,       
    \end{align}
then, for every $q>1$, $\RR$ acts as a bounded operator on
$B(V^{\nicefrac1q})$ with spectral radius 
$r(\RR)=1$ and essential spectral radius $r_e(\RR)\le\vartheta^{\nicefrac1q}$. 
\end{lemma}

In particular, $\RR$ then is quasi-compact, see \refR{Rqc}.

\begin{proof}[Proof of Proposition~\ref{prop:examplequasicomp}]
     We observe that the sequence 
    \begin{align}
    \fm_{n}:=\frac{1}{T}\sum_{i=0}^{(n+1)T-1}\delta_{Z_i}
    \end{align}
      is an \mvpp{} on the set $E$ with (random) initial measure 
      \begin{align}\label{moo}
      \fm_0= \frac{1}{T}\sum_{i=0}^{T-1}\delta_{Z_i}
      \end{align}
      and replacement kernel
	\begin{align}\label{RAMC}
	\RA_x\eqd \frac{1}{T}\sum_{i=0}^{T-1} \delta_{X_i},\quad\text{where $(X_i)_{i\geq 0}$ has law $\P_x$.}
	\end{align}

    \medskip
    {\bf We start by proving that  $\fm$ satisfies assumptions~\ref{as:balance},~\ref{as:BW} and~\ref{as:nu2}.} 
    Assumption~\ref{as:balance} holds true since 
    \begin{align}
    R^{(1)}_x(E)\eqd \frac{1}{T}\sum_{i=0}^{T-1}\delta_{X_i}(E)=1.
    \end{align}
    We now show that Assumption~\ref{as:BW} holds with 
$W:=V^{1/q}$ and
    \begin{align}\label{gthv}
    \vartheta:=\frac{1}{T}\sum_{i=0}^{T-1}\lambda^i
=\frac{1}{T}\,\frac{1-\lambda^{T}}{1-\lambda}\in(0,1).
    \end{align}
    Note that $\vartheta\in(0,1)$ since
$\lambda\in(0,1)$  and $T\geq 2$. 

    For~\ref{as:BWi}, we obtain from~\eqref{eq:VLyap} used iteratively that, for all $x\in E$ and all $n\geq 0$,
    \begin{align}\label{erland}
    \E_x V(X_{n+1})\leq \lambda \E_x V(X_n)+C\leq \lambda^{n+1} V(x)+C_1,
    \end{align}
    where $C_1:=\sum_{i=0}^\infty \lambda ^i C<+\infty$. Hence
    \begin{align}
    \label{eq:VLyapRR}
    \ER_xV=\frac{1}{T}\sum_{i=0}^{T-1} \E_x\bigsqpar{V(X_i)}
\leq \frac{1}{T}\sum_{i=0}^{T-1}(\lambda^i V(x)+C_1)=\vartheta V(x)+C_1.
    \end{align}
This proves~\ref{as:BWi}. 
    
\ref{as:BWii} then follows by \refR{RBWii}, 
since $\RA_x\ge0$ \as{} 

    For~\ref{as:BWiii}, we simply observe that  $\fm_0 V=\sum_{i=0}^{T-1}
    V(Z_i)<+\infty$.  
    This concludes the proof that Assumption~\ref{as:BW} holds true.
    
    We now show that Assumption~\ref{as:nu2} holds. 
Recall \eqref{eq:VLyap} and note that it follows that the set 
\set{x\in E: \E_x \bigsqpar{V(X_1)} > V(x)-1} is finite.
Hence, 
    by \cite[Theorem~7.5.3]{DoucEtAl2018},%
\footnote{Theorem 7.5.3 in \cite{DoucEtAl2018} is not stated correctly, but
  the direction we use is correct. The other direction becomes correct if, for example, one replaces the irreducibility assumption by a strong irreducibility assumption, which corresponds to our (classical) notion of irreducibility.
}
it follows from \eqref{eq:VLyap} 
   and the irreducibility of $X$  that $X$ is positive recurrent and thus, 
    see \cite[Theorem~7.2.1 and Definition~7.2.2]{DoucEtAl2018}, 
    that it admits a unique invariant probability measure $\nu$. 
Thus, for every bounded measurable function~$f$ and all $n\geq 0$,
    \begin{align}
    \E_\nu f(X_n)=\nu f.
    \end{align}
    Hence, still for every bounded $f$,
    \begin{align}\label{ull}
    \nu\ER f=\E_\nu\left[\frac{1}{T}\sum_{i=0}^{T-1} f(X_i)\right]=\frac{1}{T}\sum_{i=0}^{T-1} \E_\nu\left[f(X_i)\right]=\nu f
    \end{align}
and thus $\nu\ER=\nu$.
It remains to verify that $\nu V<\infty$, which follows by the following
standard arguments.
By irreducibility of $X$ 
and the fact that $E$ is countable, we have $\nu(\{x\})>0$ for all $x\in E$
and hence (see for instance \cite[Theorems~5.2.11 and~5.2.9]{DoucEtAl2018}),  
    for all $A>0$ and $x\in E$,
    \begin{align}
    \frac{1}{n}\sum_{i=1}^n \bigpar{V(X_i)\wedge A}
\xrightarrow[n\to+\infty]{} \nu(V\wedge A)\quad\mathbb P_x\text{-almost surely.}
    \end{align}
By dominated convergence and using \eqref{erland}, this implies that
    \begin{align}
\nu(V\wedge A)= \lim_{n\to+\infty} \E_x\frac{1}{n}\sum_{i=1}^n\bigpar{ V(X_i)\wedge A}\leq \liminf_{n\to+\infty} \frac{1}{n}\sum_{i=1}^n (\lambda^i V(x)+C_1)=C_1
    \end{align}
    and hence, letting $A\to+\infty$, that $\nu V\le C_1<+\infty$.
    This completes the proof  that Assumption~\ref{as:nu2}
    holds true. 

      Furthermore,  $\ER$ is  the probability kernel of an
      irreducible Markov chain on $E$, 
and thus we deduce from \cite[Theorem~7.5.3]{DoucEtAl2018}
and~\eqref{eq:VLyapRR} that      $\nu$ is the unique invariant probability
measure of $\ER$.

    \medskip
    {\bf We now show that Theorem~\ref{T1} applies to~$\fm$, which implies Proposition~\ref{prop:examplequasicomp}(a).} 
We 
first show that $\RR$ defined by \eqref{RR}
is  \aslqc{} operator on $B(W)=B(V\xq)$, i.e.\ that it satisfies
    conditions 
\ref{ergo3} and~\ref{ergo2} of Definition~\ref{Dsmall}, which entails that Theorem~\ref{T1} applies.

Note that \eqref{lqc} holds by \eqref{eq:VLyapRR}. 
Hence,
according to Lemma~\ref{lem:QC}, $r_e(\RR)\le\gthv\xq<1$ 
and thus by \refD{Dre}, 
for any $\rho\in\xpar{ \vartheta^{\nicefrac1{q}},1}$, there exists a
decomposition of 
$B(W)$ into two closed $\RR$-invariant subspaces: 
    \begin{align}
    \label{eq:decompSpec}
    B(W)=F_\rho\oplus H_\rho,
    \end{align}
    such that $F_\rho$ has finite dimension, 
and the spectral radius of $\RR\restr_{H_\rho}$ is less than $\rho$. 
Since the spectrum of $F_\rho$ is finite, this says that the spectrum
$\gs(\RR)$ contains only a finite number of points $\gl$ with $|\gl|\ge\rho$;
moreover, these points satisfy $|\gl|\le r(\RR)=1$ and thus $\Re\gl<1$
unless $\gl=1$.
This shows 
both that \ref{ergo2} holds and
that $1$ is an isolated 
point in $\gs(\RR)$. (As always, $1\in\gs(\RR)$ because $\RR1=1$.)

The generalized eigenspace of $\RR$ corresponding to the eigenvalue 1 is a
subspace of $F_\rho$, and thus has finite dimension. In order to verify
\ref{ergo3}, it remains to show that this dimension is 1,
i.e., that the eigenvalue 1 has algebraic multiplicity 1. 

We first show that the eigenvalue $1$ is simple: 
The corresponding eigenfunctions of $\RR$ satisfy $\RR f= f$,
which means that they are harmonic functions for the Markov kernel $\ER$.
As shown above, $\nu$ is the unique invariant probability measure for $\ER$,
and furthermore $\nu V<\infty$; hence,
\cite[Proposition~5.2.12]{DoucEtAl2018} shows that
every harmonic function  in $B(V)$ 
is constant $\nu$-a.e., and hence constant everywhere because
$\nu(\{x\})>0$ for all $x\in E$. 
This implies that $1$ has simple geometric multiplicity:
it remains to prove that it also has simple algebraic multiplicity. 
To do so, let $f\in B(W)$ be such that $(\RR-\II)^2f=0$. Then $(\RR-\II)f$ is
an eigenfunction associated to $1$ and hence it is equal to a constant, say
$c\in\mathbb C$. We deduce that $\RR f=f+c$  and hence $\RR^n f=f+n c$ for
all $n\geq 1$.
Moreover, for all $n\geq 1$ and $x\in E$, by iterating \eqref{eq:VLyapRR},
    \begin{align}\label{jenv}
 \RR^{n} V(x)
\leq \vartheta \RR^{n-1} V(x)+C_1
\leq \vartheta^n V(x)+\sum_{i=0}^{n-1} \vartheta^i C_1 
\le C V(x)
    \end{align}
    which implies,  by Jensen's inequality,
    \begin{align}\label{jenw}
    \RR^n W(x)\leq \left(\RR^{n} V(x)\right)^{\nicefrac1q}
\le C W(x).
    \end{align}
    In particular, for all $n\ge0$,
    \begin{align}
    |f(x)+n c|=|\RR^n f(x)|\leq \|f\|_{B(W)}\,\RR^n W(x)
\leq \|f\|_{B(W)} C W(x),
    \end{align}
    which implies that $c=0$ and hence that $(\RR-\II)f=0$, so that $f$ is
    an eigenfunction associated to $1$ and hence is constant. 
We have shown that 
$\ker\bigpar{(\RR-\II)^2}=\ker\bigpar{\RR-\II}=\set{c1:c\in\bbC}$.
This implies that the algebraic multiplicity of 1 
in the finite-dimensional space $F_\rho$ is 1, and it follows 
that~\ref{ergo3} holds true, which completes the proof that $\RR$ is \slqc. 
    
We set, as in \eqref{gth}, $\theta:=\sup \Re(\sigma(\RR)\setminus\{1\})$. 
By Theorem~\ref{T1} (with $D:=B(W)$), 
$\theta<1$ and,
for every $\delta\in (0,1-\theta)$, there exists a constant $C_\delta$ 
such that, for all $f\in B(W)$,
     \begin{align}
    \label{eqT1firstExample}
    \E \left(\left|\tfm_n- \nu f\right|^2 \mid \fm_0\right)
    \le C_\delta  \,\tfm_0 V\,\left(\frac{\fm_0(E)+1}{\fm_0(E)+n}\right)^{2\delta\wedge 1} \norm{f}_{B(W)}^2,\qquad \forall n\geq 1
.    \end{align}
    But $\fm_0(E)=1$ and,
 by \eqref{moo} and \eqref{erland}, 
    \begin{align}\label{EMV}
 \E\bigsqpar{\tfm_0 V}
=\frac{1}T\sum_{i=0}^{T-1}\E V(Z_i)
=\frac{1}T\sum_{i=0}^{T-1}\E V(X_i)
<+\infty      
  .  \end{align}
Hence, \eqref{eqT1firstExample} yields,
up to a change of $C_\delta$,
    \begin{align}
    \label{eqT1firstExample2}
    \E \left(\left|\tfm_n- \nu f\right|^2 \right)
    \le C_\delta \, n^{-2\delta\wedge 1} \norm{f}_{B(W)}^2.
    \end{align}
If furthermore
$\delta<\nicefrac12$, then for all $f\in B(W)$,
    \begin{align}
    \label{eqT1firstbisExample}
    n^\delta\left|\tfm_n- \nu f\right|\xrightarrow[n\to+\infty]{a.s.} 0.
    \end{align}
 
  For all $n\geq 1$ and $k\in\{0,\ldots,T-1\}$, we have
    	\begin{align}
        \label{eq:decompT}
    	\frac{1}{nT+k}\sum_{i=0}^{nT+k-1} \delta_{Z_i}= \frac{n T}{nT+k}\tfm_{n-1}+\frac{1}{nT+k}\sum_{i=nT}^{nT+k-1}\delta_{Z_i},
    	\end{align}
and thus, for all $f\in B(W)$ such that $\|f\|_{B(W)}\leq 1$,
        \begin{align}
\left|\frac{1}{nT+k}\sum_{i=0}^{nT+k-1} f(Z_i)-\frac{nT}{nT+k}\tfm_{n-1}f\right|
        &\leq \frac{1}{nT+k}\sum_{i=nT}^{nT+k-1} W(Z_i)
\label{eq:deltaTmn}
\\
        &\le \frac{(n+1)T}{nT+k}\tfm_{n}W-\frac{nT}{nT+k}\tfm_{n-1}W.\label{eq:deltaTmn2}
        \end{align}
Hence, 
\begin{align}\label{kling}
&\bigabs{\bmu_{nT+k-1} f-\nu f}
=\left|\frac{1}{nT+k}\sum_{i=0}^{nT+k-1} f(Z_i)-\nu f\right|
\notag\\&\qquad
\leq  
\frac{k}{nT+k}|\nu f|
+\frac{nT}{nT+k}|\tfm_{n-1}f-\nu f|
+\frac{(n+1)T}{nT+k}|\tfm_{n}W-\nu W |
\notag\\&\qquad\qquad
+\frac{nT}{nT+k}|\tfm_{n-1}W-\nu W|+\frac{T}{nT+k}\nu W
\notag\\
&\qquad
\leq  
\frac{1}{n}|\nu f|
+|\tfm_{n-1}f-\nu f|
+2|\tfm_{n}W-\nu W |
+|\tfm_{n-1}W-\nu W|+\frac{1}{n}\nu W
.\end{align}

We now obtain \eqref{eq:L2conv} from 
\eqref{kling} and~\eqref{eqT1firstExample2}
by Minkowski's inequality.
Similarly, \eqref{eq:ASconv} follows from \eqref{kling} and 
\eqref{eqT1firstbisExample},
which concludes the proof of
Proposition~\ref{prop:examplequasicomp}(a). 
    
\medskip
{\bf We next show that Theorem~\ref{T2} applies to~$\fm$ which implies Proposition~\ref{prop:examplequasicomp}(b).} 
We have proved that $\RR$ is \aslqc{} operator on $D=B(W)$;
moreover, by \refL{lem:QC} and \eqref{eq:VLyapRR}, $r_e(\RR)\le\gthv\xq$ where 
$\gthv$ is given by \eqref{gthv}.
We now assume that $\gthv\xq<\half$, and thus $r_e(\RR)<\half$. This means
that we may take $\rho<\half$ in \eqref{eq:decompSpec},
 which entails that the set
\begin{align}
\set{\gl\in\gs(\RR):\Re\gl\ge\rho } 
=
\set{\gl\in\gs(\RR\restr_{F_\rho}):\Re\gl\ge\rho } 
\end{align}
is finite.
Since in addition we have $\E((1+\fm_0(E))\tfm_0 V)<+\infty$, it follows
from Remark~\ref{rem:m0rand} that
one of the cases (1), (2) or (3) in Theorem~\ref{T2} applies to $\fm$
with $D=B(W)$.
(The case depends on whether $\gth$ is $<\half$, $=\half$, or $>\half$.) 
Moreover, in cases (2) and (3), we have $\kk<\infty$ by \refR{Rkk}.

In addition, by~\eqref{eq:deltaTmn}, for  $n\geq 1$ and $0\le k<T$,
\begin{align}
&\E \left|\frac{1}{nT+k}\sum_{i=0}^{nT+k-1} f(Z_i)-\frac{nT}{nT+k}\tfm_{n-1}f\right|
\leq
  \E \Bigsqpar{\frac{1}{nT+k}\sum_{i=nT}^{(n+1)T-1}W(Z_i)}
\notag\\&\qquad
=\E\Bigsqpar{\frac{T}{nT+k}\bigpar{\fm_n W-\fm_{n-1}W}}
=\frac{T}{nT+k}\E \bigsqpar{\RAn_{Y_n}W}
\end{align}
and thus
\begin{align}\label{carl}
&\E \left|\frac{1}{nT+k}\sum_{i=0}^{nT+k-1} f(Z_i)-\tfm_{n-1}f\right|
\leq
\frac{T}{nT+k}\E \bigsqpar{\RAn_{Y_n}W}
+\frac{k}{nT+k}\E|\tfm_{n-1}f|
\end{align}
Both $\mathbb E\bigl[\tfm_{n-1}W\bigr]$ and 
$\mathbb E\bigl[\RAn_{Y_n}W\bigr]$ are uniformly bounded in $n$
by~\eqref{lmv4a} in Lemma~\ref{Lmwqq} and~\eqref{palmc} in its proof. 
Hence, \eqref{carl} yields
\begin{align}
\Eq \bigabs{\bmu_{nT+k-1}f-\tfm_{n-1}f}=
\E \left|\frac{1}{nT+k}\sum_{i=0}^{nT+k-1} f(Z_i)-\tfm_{n-1}f\right|=O(1/n).
\end{align}
and, in particular, as $n\to\infty$, for any fixed $\ga<1$,
\begin{align}\label{theodor}
n^{\alpha} \bigpar{\bmu_{nT+k-1}f-\tfm_{n-1}f}\pto0  
.\end{align}

If conclusion (1) of Theorem~\ref{T2} holds for $\fm$,
this implies that $(\xfrac nT)^{\nicefrac12}\left(\bmu_n f-\nu f\right)$ has the
same limit in distribution
as $\xpar{ n/T}^{\nicefrac12} (\tfm_{\lfloor (n+1)/T\rfloor-1}f-\nu f)$, 
which equals the limit in \eqref{t11},
and similarly for conclusion~(2). 

If conclusion (3) of Theorem~\ref{T2} holds for
  $\fm$, we use~\eqref{eq:deltaTmn2} which entails 
\begin{align}
\left|\bmu_{nT+k-1}f-\tfm_{n-1} f\right|
 &\le \frac{(n+1)T}{nT+k}\tfm_{n}W-\frac{nT}{nT+k}\tfm_{n-1}W
+\frac{k}{nT+k}\abs{\tfm_{n-1}f}
\notag\\
 &\le \frac{(n+1)T}{nT+k}\tfm_{n}W-\frac{nT-k}{nT+k}\tfm_{n-1}W
\notag\\
 &= \frac{T+k}{nT+k}\tfm_{n}W+\frac{nT-k}{nT+k}\bigpar{\tfm_nW-\tfm_{n-1}W}
\notag\\
 &\le \frac{2}{n}\tfm_{n}W+\bigabs{\tfm_nW-\tfm_{n-1}W}
.\end{align}
In particular, setting $\alpha_n:=
\xfrac{n^{1-\Re\lambda_1}}{\log^{\kappa-1}n}$, we get 
\begin{align}
\left|\alpha_n\bigpar{\bmu_{nT+k-1} f-\nu f}
  -\sum_{j=1}^p  n^{\ii \Im\lambda_j}\Lambda_j \right|
&\leq \left|\alpha_n \left(\tfm_{n-1} f-\nu f\right)-\sum_{j=1}^p n^{\ii \Im \lambda_j}\Lambda_j\right|\notag\\
&\qquad+\ga_n\bigabs{\tfm_{n}W-\tfm_{n-1}W}
+\frac{2\alpha_n }{n}\tfm_{n}W.\label{eq:interstep}
\end{align}
The first term on the right-hand side converges to $0$ \as{} and in
$L^2$ as \ntoo{} according to conclusion (3) of Theorem~\ref{T2} for $\fm$;
it is easy to see that we may replace $\tfm_n$ by $\tfm_{n-1}$ in
\eqref{hildegran} because $\ga_n/\ga_{n-1}\to1$.
Similarly,
the third term in the right-hand side converges to $0$ \as{} and in
$L^2$ according to Theorem~\ref{T1}\ref{T1a} for $\fm$ and the fact that
$\alpha_n=o(n)$; 
note that $\E|\tfm_n W|^2=O(1)$ by taking the expectation
in \eqref{eqT1first} combined with \eqref{EMV}.
It remains to consider the second term in the right-hand
side, for which we observe that 
\begin{align}
\alpha_n\left(\tfm_{n}W-\tfm_{n-1}W\right)
&=\alpha_n\left(\tfm_{n}W-\nu W\right)-\sum_{j=1}^p  n^{\ii \Im\lambda_j}\Lambda_j
\notag  \\
&\qquad-\bigg(\alpha_n\left(\tfm_{ n-1}W-\nu W\right)-\sum_{j=1}^p  n^{\ii \Im \lambda_j}\Lambda_j\bigg),
\end{align}
where both terms go to $0$ \as and in $L^2$  as $n\to+\infty$ 
according to conclusion~(3) of Theorem~\ref{T2} for $\fm$. 
Consequently, the \lhs{} of \eqref{eq:interstep} converges to 0 \as{} and in
$L^2$.  
Finally, $\ga_{nT+k-1}/\ga_n\to T^{1-\Re\gl_1}$ as \ntoo{} and $0\le k<T$, and
it follows easily that
\begin{align}
  \left|\alpha_N\bigpar{\bmu_{N} f-\nu f}
  -\sum_{j=1}^p  N^{\ii \Im\lambda_j}T^{1-\gl_j}\Lambda_j \right|
\to0
\end{align}
\as{} and in $L^2$ as \Ntoo.

\medskip
{\bf We conclude by proving Proposition~\ref{prop:examplequasicomp}(c).} In order to do so,
we note that Lemma~\ref{lem:QC} applied to the transition probability kernel 
$P$ of $X$, using \eqref{eq:VLyap}, implies that the corresponding operator
$\PPP$ is 
quasi-compact on $B(W)$ with $r_e(\PPP)<r(\PPP)=1$.
Hence, there exists $\rho<1$ and a decomposition as in
\eqref{eq:decompSpec}, and it follows that the spectrum $\gs(\PPP)$ has only
finitely many points $\gl$ with $|\gl|>\rho$, and these points all have
$|\gl|\le1$. In particular, 1 is isolated in $\gs(\PPP)$ and thus
\begin{align}
  \label{etapp}
\eta:=\inf\bigset{|1-s|:s\in\gs(\PPP)\setminus\set1}>0.
\end{align}

We have
$\RR=\frac1T\sum_{i=0}^{T-1}{\bf P}^i$,
and thus the spectral mapping theorem 
\cite[Theorem VII.4.10]{Conway} 
shows that the spectrum of $\RR$ is given by
 \begin{align}\label{sm}
\gs(\RR) =  \left\{\frac1T\sum_{i=0}^{T-1} s^i: s\in\sigma(\PPP)\right\},
 \end{align}
and thus
 \begin{align}\label{sm1}
\gs(\RR)\setminus\set1=
 \left\{\frac1T\sum_{i=0}^{T-1} s^i: s\in\sigma({\PPP})\setminus\{1\}\right\}.
 \end{align}
 For every $s\in\sigma({\PPP})\setminus\{1\}$ we have  $|s|\le1$ 
and $|1-s|\ge\eta$,
and thus
\begin{align}\label{slim}
 \left|\frac1T\sum_{i=0}^{T-1} s^i\right|
=\lrabs{\frac{1}{T}\frac{1-s^T}{1-s}}
\le \frac{2}{\eta T}
\xrightarrow[T\to+\infty]{}0.
 \end{align}
In particular, if we choose  $T_0$ such that $T_0>4/\eta$, 
then for every $T\ge T_0$, we have by \eqref{sm1} and \eqref{slim},
\begin{align}
  \gth
= \sup\bigset{\Re\gl:\gl\in\gs(\RR)\setminus\set1}
\le \sup\bigset{|\gl|:\gl\in\gs(\RR)\setminus\set1}
<\half, 
\end{align}
and thus case (1) applies in part (b).


 This concludes the proof of Proposition~\ref{prop:examplequasicomp}.
\end{proof}

\bigskip
\begin{proof}[Proof of Lemma~\ref{lem:QC}]
The proof relies on \cite[Theorem~XIV.3]{HeHe}. 
Fix $q>1$ and set as usual $W:=V^{\nicefrac1q}$.  Jensen's inequality 
and the assumption \eqref{lqc}
entail that 
\begin{align}
	\label{eq:WLyap}
\RR W\leq \left(\RR V\right)^{\nicefrac1q}\leq \left(\vartheta V+C\right)^{\nicefrac1q}\leq \vartheta^{\nicefrac1q} W+C^{\nicefrac1q}.
\end{align}
In particular, this shows that $\RR$ acts as a bounded operator on $B(W)$;
we regard in the rest of the proof $\RR$ as an operator on $B(W)$.
By induction similar to \eqref{jenv},
\eqref{eq:WLyap} also implies that $\RR^n W\leq C W$ for some constant $C>0$
and all $n\geq 0$. Thus 
$\norm{\RR^n}_{B(W)}\le C$ and
by the spectral radius formula
\cite[Proposition~VII.3.8]{Conway}, 
the spectral radius $\gs(\RR)$ 
of $\RR$ 
is at most~$1$. 
Since $1\in B(W)$ and  we have $\RR 1= 1$, we deduce
that $1$ is an eigenvalue of  $\RR$. 
We can
thus conclude that the spectral radius of $\RR$, 
as a bounded operator on $B(W)$, equals 1.

To apply \cite[Theorem~XIV.3]{HeHe}, we consider the Banach space $(B(W),
\|\cdot\|_{B(W)})$, endowed with the continuous norm $\|\cdot\|_{B(V)}$.
 We check that
\begin{romenumerate}
	\item\label{kli} $\RR(\{f\in B(W): \|f\|_{B(W)}\leq 1\})$ is 
totally bounded 
in $( B(W), \triplenorm{\cdot})$; 
	\item\label{klii} 
there exists a constant $M>0$ such that, for all $f\in  B(W)$, $ \triplenorm{\RR f}\leq M \triplenorm{f}$;
	\item\label{kliii} 
for any $\varepsilon>0$, there exists a constant $C_\varepsilon>0$ such that
	\begin{align}
	\|\RR f\|_{B(W)}\leq (\vartheta^{\nicefrac1q}+\varepsilon) \|f\|_{B(W)}+C_\varepsilon \triplenorm{f}.
	\end{align}
\end{romenumerate}
Once this is proved, the conclusion of Lemma~\ref{lem:QC} immediately
follows from~\cite[Theorem~XIV.3]{HeHe}.

We first prove \ref{kli}. 
Recall that a set in a metric space is totally bounded 
if for every $\eps>0$ there is a finite $\eps$-net in it, i.e., a finite
subset $F$ such that every point in the set has distance at most $\eps$ to
$F$. (This is also called precompact, and in a complete metric space it is
equivalent to relatively compact. 
Thus \ref{kli} says that $\RR$ is a compact operator $B(W)\to B(V)$.
See \eg{} \cite[I.6.14--15]{Dunford-Schwartz}.)
Let 
\begin{align}
U:=
\set{f\in\BW:\norm{f}_{\BW}\le1}
=\set{f\in \bbC^E:|f(x)|\le W(x),\ \forall x\in E}
\end{align}
be the unit ball of $B(W)$. Since $\RR$ is bounded on $B(W)$,
$\RR(U)\subseteq C U$ for some constant $C$, and it suffices to show that
$U$ is totally bounded for the norm $\triplenorm{\cdot}$.

Let $\eps>0$. 
Fix $M>0$, and let $K_M:=\{x\in E: V(x)\leq M\}$; recall that this set is
finite.
Consider first the restrictions to $K_M$.
$U_M:=\set{f\restr_{K_M}:f\in U}$ is a bounded set in the finite-dimensional
space $\bbC^{K_M}$, and thus it is relatively compact. (In fact, it is compact.)
Hence, there exists 
a finite set $\set{f_i}_{i=1}^N\subset U_{M}$ 
such that for every $f\in U$ there exists an $f_i$ with
\begin{align}\label{annak}
\max_{x\in K_M}|f(x)-f_i(x)|<\eps.
\end{align}
Extend every $f_i$ to a function on $E$, still denoted $f_i$,  
by $f_i(x):=0$ for $x\notin K_M$. If $f\in U$ and $x\notin K_M$, then
for every $i\in\set{1,\dots,N}$,
\begin{align}
  \frac{|f(x)-f_i(x)|}{V(x)} 
=  \frac{|f(x)|}{V(x)} 
\le \frac{W(x)}{V(x)} = V(x)^{\nicefrac1q-1}
\le M^{\nicefrac1q-1}.
\end{align}
By choosing $M$ large enough, this is less than $\eps$.
Hence, if $f_i$ is chosen to satisfy \eqref{annak}, then
$|f(x)-f_i(x)|/V(x)\le\eps$ for every $x\in E$, and thus
$\triplenorm{f-f_i}\le\eps$. Hence $\set{f_i}_1^N$ is a finite $\eps$-net in
$U$. Consequently, \eqref{kli} holds.

The property \ref{klii} is a consequence of~\eqref{lqc}, which indeed
implies that, for all $f\in B(W)$,
\begin{align}
 \triplenorm{\RR f}=\sup_{x\in E} \frac{|\RR f(x)|}{V(x)}\leq \sup_{x\in E} \frac{\|f/V\|_\infty\,\RR V(x)}{V(x)}\leq  \triplenorm{f}\,(\vartheta+C).
\end{align}

We now prove \ref{kliii}: Since $\inf_{x\notin K_M} W(x)\ge M\to+\infty$
when $M\to+\infty$, we deduce from~\eqref{eq:WLyap} that, for any
$\varepsilon>0$, there exists $M(\varepsilon)>0$ and a constant
$c_\varepsilon>0$ such that  
\begin{align}
\RR W\leq (\vartheta^{\nicefrac1q}+\varepsilon)W+c_\varepsilon \etta_{K_{M(\varepsilon)}}(x),\ \forall x\in E.
\end{align}
Hence, for all $x\notin K_{M(\varepsilon)}$,
\begin{align}
|\RR f(x)|\leq  \|f/W\|_\infty\,\RR W(x)\leq \|f\|_{B(W)}\,(\vartheta^{\nicefrac1q}+\varepsilon)\,W(x).
\end{align}
But, according to~\eqref{lqc}, for all $x\in K_{M(\varepsilon)}$,
\begin{align}
|\RR f(x)|&\leq  \|f/V\|_\infty\,\RR V(x)
\leq\triplenorm{f}\,(\vartheta+C)\,V(x)
\notag\\&
  \leq  \triplenorm{f}\,(\vartheta+C)\,\max_{y\in K_{M(\varepsilon)}}\frac{V(y)}{W(y)}\,W(x).
\end{align}
Setting $C_\varepsilon= (\vartheta+C)\,\max_{y\in K_{M(\varepsilon)}}\frac{V(y)}{W(y)}$ and using the two previous inequalities, we deduce that, for all $x\in E$, 
\begin{align}
\frac{|\RR f(x)|}{W(x)}\leq  (\vartheta^{\nicefrac1q}+\varepsilon)\|f\|_{B(W)}+C_\varepsilon  \triplenorm{f},
\end{align}
which concludes the proof of \ref{kliii} and hence of Lemma~\ref{lem:QC}.
\end{proof}

\appendix
\section{Kernels and the definition of the \mvpp}
\label{Akernels}

We use the notation introduced in \refSS{S:not}.

\subsection{Kernels}
Recall that given two measurable spaces $(S,\cS)$ and $(T,\cT)$, a
\emph{kernel} from $S$ to $T$ is a map $s\mapsto\beta_s$ from $S$ to the set
$\cmp(T)$
of positive measures on $(T,\cT)$
that is measurable; 
in other words, $s\mapsto\beta_s(B)$ is $\cS$-measurable
for every fixed set $B\in\cT$. 
See \eg{}
\cite[pp.~20--21]{Kallenberg}
or \cite[Section 1.3]{Kallenberg-RM}
for a detailed discussion; we summarize a few facts that we need.

A probability kernel is the special case when each $\beta_s$ is a probability
measure on $T$.

A signed kernel is defined in the same way, with $\beta_s$ a signed measure on
$T$. 

If $\ga\in\cP(S)$ and $\gb$ is a probability kernel from $S$ to $T$, then
$\ga\otimes\gb$ 
denotes the probability measure  on $S\times T$ given by
\begin{align}
  \label{ker}
\ga\otimes\gb(A):=\int_S\dd\ga(s)\int_T\etta_A(s,t)\dd\gb_s(t).
\end{align}
This means that if $(X,Y)$ is a random variable in $S\times T$ with the
distribution $\ga\otimes\gb$, then $X$ has distribution $\ga$, and the
conditional distribution of $Y$ given $X=x$ is $\gb_x$ (for \aex{} $x$);
hence \eqref{ker} formalizes the notion of choosing randomly first $X$ with
distribution $\ga$, and then $Y$ with distribution $\gb_X$.

The construction \eqref{ker} generalizes to the case where
$\ga$ is a probability kernel from a third  space $U$ to $S$; then
$\ga\otimes\gb$ is a probability kernel from $U$ to $S\times T$.

\subsection{The \mvpp}
The definition of the \mvpp{} in \refS{Smodel} uses a family 
$(\RA_x)_{x\in  E}$ of random (signed) measures in $\cmr(E)$. Only their
distributions matter, so letting $\cR_x:=\cL(\RA_x)$, the distribution of
$\RA_x$,
it is equivalent to start with a family $\cR=(\cR_x,x\in E)$ of probability
distributions in $\cmr(E)$, 
or equivalently a map
$\cR:E \to \cP(\cmr(E))$;
we may then define, for each $x\in E$,
$R_x$ as a random measure in $\cmr(E)$ with distribution $\cR_x$, 
and $R\nn_x$ as a sequence of independent copies of $\RA_x$.

Our basic assumption 
is that 
$\cR=(\cR_x, x\in E)$
is a probability kernel from $E$ to $\cmr(E)$,
which we call the \emph{replacement kernel}.
(We abuse notation and use the same name also for
the corresponding family  $(\RA_x)_x$ of random measures.)

\begin{remark}\label{RER}
  The assumption that $\cR$ 
is a probability kernel from $E$ to $\cmr(E)$
implies that its expectation $\ER$ defined in \eqref{ER}
is a signed kernel from $E$ to $E$,
provided that \eqref{fin2} holds. 

It is also easy to see that the assumption that $\cR$ is a kernel implies
that $\BBB_x(f,g)$ in \eqref{BBB} is a measurable function of $x$; hence also
$\BB_x(f)$ and $\BC_x(f)$ in \eqref{BB} are  measurable.
\end{remark}

Let us now try to formalize the definition of the \mvpp, 
starting from a given replacement kernel $\cR$
and a given deterministic
$\fm_0\in\cmpp(E)$.
Our aim is to define random variables $Y_n\in E$ and
$\RYn\in\cmr(E)$ for all $n\ge1$ satisfying the description in \refS{Smodel};
then $\fm_n$ is 
given by
\begin{align}\label{mnA}
  \fm_n:=\fm_0+\sum_{i=1}^n \RYi
.\end{align}
Equivalently, we want to construct the joint distribution of 
all $(Y_n,\RYn)$, $n\ge1$, as a probability measure on 
$(E\times\cmr(E))^\infty$.
We will  achieve this 
 using the construction \eqref{ker} twice.
However, we have (so far) only been able to do so
assuming one of the following assumptions (or both).
\begin{romenumerate}
\item \label{Ai} 
$\RA_x$ is always a positive measure, so there are no subtractions in
  the urn, or
\item \label{Aii}
$E$ is a Borel space
(see \eg{} \cite[Appendix A]{Kallenberg}).
\end{romenumerate}
The reasons for the technical assumption \ref{Aii} will be discussed below.

\textbf{\ref{Ai}}:
Consider first the simple case when $\RA_x$ always is a positive measure, 
\ie, $\RA_x\in\cmp(E)$.
In this case, there is no need to consider signed measures.
Write $\cX:=E\times\cmp(E)$.
Let $n\ge0$ and assume that we have constructed
the distribution $\mu_n$ of $(Y_i,\RYi)_1^n$, 
as a probability measure on $\cX^n$. 
(This assumption is void for $n=0$.) 
We write an element of $\cX^n$ as
$(y_i,r_i)_1^n$;
then
we can realize
$Y_i$ and $\RYi$ for $i\le n$ as the coordinate functions 
$y_i$ and $r_i$
on the probability space $(\cX^n,\mu_n)$.
By \eqref{mnA}, $\fm_n$ then is given by the function $m_n:\cX^n\to\cmpp(E)$
defined by 
\begin{align}\label{ekmn}
  m_n\bigpar{(y_i,r_i)_1^n} :=\fm_0+\sumin r_i.
\end{align}
Thus, the normalized measure $\tfm_n$ is given by 
the function $\gamma_n:\cX^n\to\P(E)$
defined by 
\begin{align}\label{ekgamma}
  \gamma_n(\xi_n):=m_n(\xi_n)/m_n(\xi_n)(E).
\end{align}
Nota that $\gamma_n$ is a probability kernel from $\cX^n$ to $E$.

We want $Y_{n+1}$ to be a random element of $E$ such that,
conditioned on the history up to time $n$, 
$Y_{n+1}$ has the distribution $\tfm_n$. 
In other words, conditioned on $(Y_i,\RYi)_1^n=\xi_n\in\cX^n$,
$Y_{n+1}$ has the conditional distribution $\gamma_n(\xi_n)$.
This means that
\begin{align}\label{ekY}
  \bigpar{(Y_i,\RYi)_1^n,Y_{n+1}}\sim \mu_n\otimes\gamma_n,
\end{align}
 and we
may take this as a formal definition of (the distribution of) 
$Y_{n+1}$.

Next, the replacement kernel $\cR$ is now assumed to be a probability kernel
from $E$ to $\cmp(E)$. 
We may (trivially) regard it as a kernel from $\cX^n\times E$ by
letting $\cR_{(\xi_n,x)}:=\cR_x$.
Hence, \eqref{ker} defines the probability measure
$(\mu_n\otimes\gamma_n)\otimes\cR$
on $\cX^n\times E\times\cmp(E)=\cX^{n+1}$.
We want $R\nnz_{Y_{n+1}}$ to have the conditional distribution,
given the previous history, $\cR_{Y_{n+1}}$, and thus  
\begin{align}\label{ekYR}
((Y_i,\RYi)_1^n,Y_{n+1},R\nnz_{Y_{n+1}})
\sim
\mu_n\otimes\gamma_n\otimes\cR.
\end{align}
(Note that $\otimes$ is associative: 
$(\mu_n\otimes\gamma_n)\otimes\cR=\mu_n\otimes(\gamma_n\otimes\cR)$, so we
may omit the brackets.) 
We may take \eqref{ekYR} as a formal definition  of 
$\RYni$. 

In other words, our formal construction is
\begin{align}\label{ekmu}
  \mu_{n+1}:=\mu_n\otimes\gamma_n\otimes\cR \in \cP\bigpar{\cX^{n+1}}.
\end{align}
This completes the inductive step, and starting from the trivial probability
measure $\mu_0$ on a one-point space, we obtain recursively
a probability measure $\mu_n$ on $\cX^n$ for every $n\ge1$.
Finally, since $\mu_n$ are obtained recusively by composing with the probability
kernels $\gamma_n\otimes\cR$, 
the Ionescu Tulcea theorem \cite[Theorem 6.17]{Kallenberg} now shows the
existence of a probability space and  infinite sequences
$Y_n$ and $\RYn$ with the desired distribution; this defines also $\fm_n$ by
\eqref{mnA}. 
Equivalently, the Ionescu Tulcea theorem shows the existence of a
probability measure on $\cX^\infty$ with the desired projection $\mu_n$ to
$\cX^n$ for each $n$. 
This completes the construction in the special case when $\RA_x\in\cmp(E)$.
It follows from the construction that $\bigpar{\fm_n,Y_n,\RAn_{Y_n}}_{n\ge1}$ 
is a Markov chain.

\textbf{\ref{Aii}}:
Consider now the general case, when $\RA_x\in\cmr(E)$ is a signed measure,
but we assume that the urn is tenable. Assume now also that $E$ is a Borel
space. 
We may now define $\cX:=E\times\cmr(E)$, and try to argue as above.
The only problem is that $\gamma_n$ defined by \eqref{ekgamma} is not a
probability kernel, since $m_n(\xi_n)$ is not a positive measure for all
$\xi_n\in\cX^n$. (We will even have $m_n(\xi_n)(E)=0$ for some $\xi_n$, and then
$\gamma_n(\xi_n)$ is not even defined.) We thus have to modify the
definition of $\gamma_n$.
Consider again some $n\ge0$ and assume that we have constructed
$\mu_n\in\cP(\cX^n)$.
Note that for a Borel space $E$, $\cmpp(E)$ is a measurable subset of
$\cmr(E)$, as may easily be verified.
Let $\gY_n:=m_n\qw(\cmpp(E))$, where $m_n:\cX^n\to\cmr(E)$ is the function
defined in \eqref{ekmn}; thus $\gY_n$ is a measurable subset of $\cX^n$. 
We assume that the urn is tenable, 
which means that $\fm_n$ 
\as{} satisfies
$\fm_n\in\cmpp(E)$. In other words,  $m_n(\xi)\in\cmpp(E)$ for
$\mu_n$-\aex{} $\xi$; equivalently, $\mu_n(\gY_n)=1$.

We may now modify \eqref{ekgamma} and
define a probability kernel $\gamma_n$ from $\cX^n$ to $E$
by
\begin{align}\label{thgamma}
  \gamma_n(\xi_n):=
  \begin{cases}
    \tm_n(\xi_n):=m_n(\xi_n)/m_n(\xi_n)(E),& \text{if }\xi_n\in \gY_n,\\
    \nu,&\text{if }\xi_n \in\cX^n\setminus\gY_n,
  \end{cases}
\end{align}
where $\nu$ is an arbitrary, fixed probability measure on $E$.
Then the construction proceeds as above.
(The choice of $\nu$ does not affect $\mu_{n+1}$, since $\mu_n(\gY)=1$.)
This completes the construction in case \ref{Aii}, when $E$ is a Borel space.

What happens when $E$ is not a Borel space?
In some cases it might be possible to modify the construction above;
for example if (for each $n\ge1$)
there exists a measurable subset $\gY_n$ of $m_n\qw(\cmpp(E))$
such that $\mu_n(\gY_n)=1$.
However, we will see in \refE{Emeas-} that in general no such $\gY_n$ exists.
In general, unless \ref{Ai} ar \ref{Aii} above holds, we have to assume that
the process $\fm_n$ is defined by some external construction.
(See \refE{Emeas-} for an example where a construction is trivial.)

\begin{example}\label{Emeas-}
  Let $E:=\setoi^\cA$ for some uncountable set $\cA$.
Let $Z$ be a random element of $E$, with some distribution $\nu_Z\in\cP(E)$, and
let 
\begin{align}\label{meas-1}
R_x:=-\gd_x+2\gd_Z, \qquad x\in E; 
\end{align}
also, let $\fm_0:=\gd_{x_0}$ for some $x_0\in E$.
This describes an urn with balls (corresponding to point masses) labelled by
elements of $E$; we start with a single ball $x_0$, and in each step we
remove one randomly chosen ball, and add two new balls with label $Z_n$,
where $(Z_n)\xoo$ are i.i.d. This process is obviously well defined and tenable.
Nevertheless, we will see that there is no measurable set $\gY_1$ such that
the construction \eqref{thgamma} works for $n=1$.
(In particular, $\cmp(E)$ is not a measurable subset of $\cmr(E)$.)
Note that necessarily $Y_1=x_0$, and thus $\RYx1=-\gd_{x_0}+2\gd_{Z_1}$.
Hence, the distribution $\mu_1$ of $(Y_1,\RYx1)$ is the product measure
$\gd_{x_0}\times\cL(\RYx1)$.
Suppose that $\gY_1\subseteq \cX=E\times\cmr(E)$ is measurable and such that
$\mu_1(\gY_1)=1$ and $m_1(y,r)=\fm_0+r\in\cmpp(E)$ for every $(y,r)\in \gY_1$.
We will show  that this leads to a contradiction.

Let $\gL\subseteq\cmr(E)$ be a non-empty measurable set.
Recall that the \gsf{} on $\cmr(E)$ is generated by the mappings
$\mu\mapsto\mu(B)$ for $B\in\cE$, where $\cE$ is the \gsf{} on $E$.
It is well known that this implies that
there exists a countable
family $(B_i)\subset\cE$ such that $\gL$ belongs to the \gsf{} generated by
the mappings $\mu\mapsto\mu(B_i)$, $i\in\bbN$. (Because the union of these
\gsf{} over all countable families $(B_i)$ is a \gsf.)

Similarly, since the product \gsf{} $\cE$ is generated by the coordinate
maps $(x_a)_{a\in\cA}\mapsto x_a$ for $a\in \cA$, for each $B\in\cE$ there is a
countable subset $\cA_B\subset\cA$ 
and a (measurable) set $\tB\subseteq\setoi^{\cA_B}$
such that
\begin{align}\label{labb}
B=\tB_i\times \setoi^{\cA\setminus\cA_B}.  
\end{align}
Fix a coordinate $a'\in\cA\setminus\bigcup_i\cA_{B_i}$. 
Define, for $j\in\setoi$, the elements $z^j=(z^j_a)\in E$ by 
\begin{align}
  z^j_a:=
  \begin{cases}
j,& a=a',\\
    0,&a\neq a'.
  \end{cases}
\end{align}

Take a  signed measure $\gl\in\gL$, and for $N\ge0$, let
$\gl_N:=\gl+N(\gd_{z^0}-\gd_{z^1})$.
For each $B_i$, we have $a'\notin\cA_{B_i}$, and thus, by \eqref{labb},
$z^0\in B_i\iff z^1\in B_i$. Consequently, for every $N\ge0$,
\begin{align}
  \gl_N(B_i)=\gl(B_i)+N\bigpar{\indic{z^0\in\cA_{B_i}}-\indic{z^1\in\cA_{B_i}}}
=\gl(B_i)
\end{align}
for every $B_i$. Since $\gL$ is in the \gsf{} generated by the maps
$\mu\mapsto\mu(B_i)$, and $\gl\in\gL$, it follows that $\gl_N\in\gL$.
On the other hand, if $B:=\set{(x_a)\in E:x_{a'}=1}$,
then $B\in\cE$ and $\gl_N(B)=\gl(B)-N$; hence, if $N$ is large enough, 
$\gl_N(B)<1$ and thus $\gl_N+\fm_0\notin\cmpp(E)$.

We have shown that  there is no nonempty measurable set 
$\gL\subseteq\cmr(E)$ such that 
\begin{align}\label{samuel}
\gl\in\gL\implies \gl+\fm_0\in\cmpp(E).  
\end{align}
However, if $\gY_1$ is as above, then the section
$\gL:=\set{r\in \gY_1:(x_0,r)\in\gY_1}$ is measurable,
satisfies \eqref{samuel},
and also $\P(\RYx1\in\gL)=1$, a contradiction.

Note that the proof shows that $\cmpp(E)$ is not a measurable
subset of $\cmr(E)$, and, moreover, that it does not contain any 
non-empty measurable subset. (The same holds for  $\cmp(E)$.)
\end{example}

\section{Some functional analysis}\label{AppFA}
In this appendix we state some general results on spectra of operators in
Banach spaces; these are used in our examples in \refS{Sex}.
The results are simple and have presumably been known for a long time, but
since we have not found references to the results in the form that we need,
we give full proofs for completeness.

Recall that if $T$ is a bounded operator on $\cX$, and $T^*$ is its adjoint
acting on the dual space $\cX^*$, then
\cite[Proposition VII.6.1]{Conway}
\begin{align}\label{gs*}
  \gs(T^*)=\gs(T).
\end{align}
Our first lemma deals with the situation when we instead
consider $T^*$ as acting on a subspace of $\cX^*$.

\begin{definition}\label{def:Kchapeau}
If $K$ is a compact subset of $\bbC$, define 
$K\sphatx$ as the union of $K$ and all
bounded connected components of $\bbC\setminus K$; in other words,
its complement $\bbC\setminus K\sphatx$ is
the unbounded component of $\bbC\setminus K$.
($K\sphatx$ is known as the polynomially convex hull of $K$,
see \cite[Definition VII.5.2 and Proposition VII.5.3]{Conway}.)
In particular, if $T$ is a bounded operator on a Banach space
and $\rhooo(T)$ denotes the unbounded component of 
the resolvent set $\rho(T)=\bbC\setminus\gs(T)$, then 
\begin{align}\label{sphat}
  \gs(T)\sphatx:=\bbC\setminus\rhooo(T).	
\end{align}
\end{definition}

We let $\innprod{x^*,x}$ denote the pairing of elements $x^*\in \cX^*$ and
$x\in \cX$, for any Banach space~$\cX$.

\begin{lemma}\label{Lgs1}
  Let $T$ be a bounded operator on a complex Banach space $\cX$, and 
suppose that $\cY\subseteq \cX^*$ is a 
closed subspace of the dual space $\cX^*$
such that the adjoint operator $T^*$ maps $\cY$ into itself.
\begin{romenumerate}
\item \label{Lgs1a}
Then
\begin{align}\label{lgs1a}
\gs\bigpar{T^*|_\cY}\subseteq\gs(T)\sphatx.
\end{align}

\item \label{Lgs1b}      
Suppose further that $\cY$ is norm-determining, \ie, that
if $x\in \cX$, then 
\begin{align}\label{normdet}
\norm{x}=\sup\set{\innprod{x^*,x}:x^*\in\cY,\,\norm{x^*}=1}.  
\end{align}
Then also
\begin{align}\label{lgs1b}
\gs(T)\subseteq\gs\bigpar{T^*|_\cY}\sphatx
\end{align}
and thus
\begin{align}\label{lgs1b=}
\gs\bigpar{T^*|_\cY}\sphatx=\gs(T)\sphatx.
\end{align}
\end{romenumerate}
\end{lemma}

\begin{proof}\pfitemref{Lgs1a}
As said above,
	the spectrum $\gs(T^*)$ of $T^*$ as an operator on $\cX^*$ equals $\gs(T)$,
	and the resolvent is simply given by $(z-T^*)\qw=((z-T)\qw)^*$, 
	$z\in\rho(T)=\bbC\setminus\gs(T)$.
	
	We first show that this resolvent maps $\cY$ into itself, at least when
	$z\notin\gs(T)\sphatx$.
	To do so, we take $y\in \cY$ and let $x\xx\in \cX\xx$ be such that
    $x\xx\perp \cY$, \ie, 
	$\innprod{x\xx,y'}=0$ for every $y'\in \cY$.
	Consider the function 
	\begin{align}\label{gsa}
		g(z):=\innprod{x\xx,(z-T^*)\qw y}, \qquad z\in\rhoT=\rho(T^*).
	\end{align}
	The function $g$ is analytic on $\rho(T)$,
see \cite[Theorem VII.3.6]{Conway}. Furthermore, if $|z|>\norm{T}$, then
	$(z-T^*)\qw = \sumko z^{-k-1}(T^*)^k$ with an absolutely convergent sum, and
	thus, because $T^*(\cY)\subseteq \cY$,
	\begin{align}\label{gsb}
		(z-T^*)\qw y= \sumko z^{-k-1}(T^*)^k y \in \cY.
	\end{align}
	Consequently, \eqref{gsa} and \eqref{gsb} imply that if
    $|z|>\norm{T}$, then 
	$g(z)=0$. 
By analytic continuation, $g(z)=0$ in the unbounded connected component
$\rhooo(T)$ of $\rhoT$.

	This holds for any $x\xx\perp \cY$, and thus, by definition of $g$ in~\eqref{gsa},
	it follows that $(z-T^*)\qw y\in \cY$ for all $z\in \rhooo(T)$.
	In other words, for all $z\in\rhooo(T)$, 
we have $(z-T^*)\qw: \cY\to \cY$, which means that it is
	the inverse of the restriction $(z-T^*)\restr_\cY$. Hence, for all $z\in\rhooo(T)$,
	$z$ belongs to the resolvent set $\rho(T^*\restr_\cY)$; in other words,
	$\rhooo(T)\subseteq \rho(T^*\restr_\cY)$, and thus 
\eqref{lgs1a} holds by \eqref{sphat}.

\pfitemref{Lgs1b}
  The canonical embedding $\cX\to \cX\xx$ 
induces a linear map $\cX\to \cY^*$,
which is an isometric embedding by the assumption \eqref{normdet}.
Hence, we may regard $\cX$ as a subspace of $\cY^*$.
We may thus apply part \ref{Lgs1a} with $\cX$
and $\cY$,  and also $T$ and $T^*$, interchanged.
This yields \eqref{lgs1b}, and \eqref{lgs1b=} then easily follows from
\eqref{lgs1a} and \eqref{lgs1b}.
\end{proof}

\begin{corollary}\label{Cgs}
  Let $T$ be a bounded operator on a complex Banach space $\cX$, and 
suppose that $\cY\subseteq \cX^*$ is a 
closed subspace of the dual space $\cX^*$
such that the adjoint operator $T^*$ maps $\cY$ into itself.
Suppose further that $\cY$ is norm-determining.
Then
\begin{romenumerate}
\item   \label{Cgsa}
 $T$ is \aslqc{} operator on $\cX$ if and only if
$T^*$ is \aslqc{} operator on $\cY$.

\item   \label{Cgsb}
$T$ is a small operator on $\cX$ if and only if
$T^*$ is a small operator on $\cY$.
\end{romenumerate}
\end{corollary}

\begin{proof}
\pfitemref{Cgsa}
Suppose that $T$ is  \aslqc{}  operator.
Let $\gth:=\sup\Re\bigpar{\gs(T)\setminus\set1}$ and note that $\gth<1$ as
in \refL{Lth}\ref{Ltha}.
We then have
\begin{align}\label{tretton}
  U:=\bigset{\gl:\Re\gl>\gth}\setminus\set1\subset \rho(T),
\end{align}
which implies, since the set $U$ is connected and unbounded,
\begin{align}\label{kasper}
\bigset{\gl:\Re\gl>\gth}\setminus\set1\subset \rhooo(T),  
\end{align}
and thus
\begin{align}\label{jesper}
  \gs(T)\sphat\subset \bigset{\gl:\Re\gl\le\gth}\cup\set1.
\end{align}
Hence \refL{Lgs1} yields 
\begin{align}\label{melchior}
\gs(T^*\restr_{\cY})\sphat
=\gs(T)\sphat
\subset \set{\gl:\Re\gl\le\gth}\cup\set1,  
\end{align}
which implies \ref{ergo2}  for $T^*\restr_{\cY}$,
and also that 1 is isolated in $\gs(T^*\restr_{\cY})$
if 1 belongs to this spectrum at all.
It remains to show only that 1 is an eigenvalue of $T^*\restr_{\cY}$ and
that the corresponding spectral projection 
$\Pi_1(T^*\restr_{\cY})$ has rank~1.

We can regard $T^*$ as an operator on $\cX^*$ or on its subspace $\cY$.
In both cases we have, see \cite[Equation~VII.6.9]{Conway},
	\begin{align}\label{e1tx}
		\Pi_1(T^*)=\frac{1}{2\pi\ii}\oint_\gG(z-T^*)\qw\dd z
	\end{align}
where we choose $\gG$ to be a small circle around 1
inside $\rhooo(T)$, cf.\ \eqref{kasper}.
By the proof of \refL{Lgs1}, if $z\in\gG$, then
$(z-T^*)\qw$ maps $\cY$ into itself, and its restriction to $\cY$ is
$(z-T^*\restr_\cY)\qw$. Hence, \eqref{e1tx} shows that $\Pi_1(T^*)$ maps
$\cY$ into itself, and that its restriction to $\cY$ is $\Pi_1(T^*\restr_\cY)$.

Moreover, $(z-T^*)\qw=\bigpar{(z-T)\qw}^*$ for $z\in \gG$, and thus
by \eqref{e1tx} and the same formula for $T$, 
we have	$\Pi_1(T^*)=\Pi_1(T)^*$. 
By Assumption \ref{ergo3}, $\Pi_1(T)$ has rank 1, and 
	is thus given by 
	\begin{align}\label{qc}
		\Pi_1(T)x = \innprod{x_0^*,x}x_0  
	\end{align}
	for some non-zero $x_0\in \cX$ and $x^*_0\in \cX^*$ with
    $\innprod{x_0^*,x_0}=1$. 
	It follows that, for any $x^*\in \cX^*$,
	\begin{align}\label{qa}
		\Pi_1(T^*)(x^*)
		=   \Pi_1(T)^*(x^*)
		=\innprod{x^*,x_0}x^*_0.
	\end{align}
Since $\cY$ is norm-determining, there exists $y\in \cY$ such that
	$\innprod{y,x_0}\neq0$. 
Since $\Pi_1(T^*):\cY\to\cY$, we have     $\Pi_1(T^*)(y)\in \cY$,
and \eqref{qa} then shows that $x^*_0\in \cY$.
	
We have shown that $\Pi_1(T^*\restr_\cY)$ is the rank 1 operator defined
 by \eqref{qa} restricted to
$x^*\in\cY$.
In particular, $x_0^*\in \cY$ is an eigenvector with $T^*x_0^*=x^*_0$.
Hence, \ref{ergo3} in \refD{Dsmall} holds for
$T^*\restr_\cY$, which concludes the proof that
$T^*\restr_\cY$ is \slqc{} if $T$ is.

The converse follows, as in the proof of \refL{Lgs1}\ref{Lgs1b},
by interchanging the roles of $\cX$ and $\cY$, noting that $\cX$ always is
norm-determining as a subspace of $\cY^*$.

\pfitemref{Cgsb}
Now suppose that $T$ is small. 
This means that in the proof of \ref{Cgsa}, we have
$\gth<\half$. Hence, \eqref{melchior} shows that $T^*$ is small.
The converse follows as above.
\end{proof}

In the following lemma, we compare the spectra of the ``same'' operator in
two different spaces. 
When necessary, we use subscripts such as $T_\cX$ to denote the space
where we consider the operator.

\begin{lemma}\label{LXY}
  Let $\cX$ and $\cY$ be two complex Banach spaces and suppose that
$\cY\subseteq\cX$ with a continuous, but not necessarily isometric, inclusion.
Suppose that $T$ is a bounded operator on $\cX$ such that
$T(\cX)\subseteq\cY$. 
\begin{romenumerate}
  
\item \label{LXYa}
Then
\begin{align}\label{lxy}
  \gs(T_\cX)=
  \begin{cases}
      \gs(T_\cY), &\cY=\cX,
\\
    \gs(T_\cY)\cup\set0, &\cY\subsetneq\cX.
  \end{cases}
\end{align}
(We do not make any claims on whether $0\in\gs(T_\cY)$ or not.)
\item \label{LXYb}
If $\gl\neq0$ is an isolated point in  $\gs(T_\cX)$,
then $\Pi_\gl(T_\cY)$ equals the restriction of $\Pi_\gl(T_\cX)$ to $\cY$.
(Thus we can use the notation $\Pi_\gl$ for both without confusion.)
Moreover,
$\Pi_\gl\cX=\Pi_\gl\cY\subseteq\cY$.

\item \label{LXYc}
$T_\cX$ is \slqc{} if and only if $T_\cY$ is \slqc.
$T_\cX$ is small if and only if $T_\cY$ is small.
\end{romenumerate}
\end{lemma}

\begin{proof}
\pfitemref{LXYa}
  Note first that by the closed graph theorem, $T:\cX\to\cY$ is a bounded
  operator. 
Hence, the restriction $T_\cY$ to $\cY$ is a bounded operator on $\cY$,
and the spectra $\gs(T_\cX)$ and $\gs(T_\cY)$ are both defined.

If $\cY=\cX$, then the norms on $\cX$ and $\cY$ are equivalent, again by the
closed graph theorem, and thus $\gs(T_\cX)=\gs(T_\cY)$.

Assume in the sequel that $\cY\neq\cX$. In particular, since
$T(\cX)\subseteq\cY$, $T$ is not onto $\cX$, and thus $T_\cX$ is not
invertible; hence $0\in\gs(T_\cX)$.

Suppose that $\gl\in\rho(T_\cX)$. This means that the resolvent 
$R_\gl:=(\gl-T)\qw$ exists as a bounded operator on $\cX$.
We have
\begin{align}\label{gsk}
  I = (\gl-T)R_\gl = \gl R_\gl - T R_\gl.
\end{align}
Hence, 
if $y\in \cY$, then,
using again $T(\cX)\subseteq\cY$,
\begin{align}\label{gok}
  \gl R_\gl y = y + T R_\gl y \in \cY.
\end{align}
Since $0\notin\rho(T_\cX)$, as remarked above, we have $\gl\neq0$.
Hence \eqref{gok} implies $R_\gl y\in\cY$, and thus $R_\gl:\cY\to\cY$.
It follows immediately that the restriction of $R_\gl$ to $\cY$ is an inverse to $\gl-T_\cY$, and
thus $\gl\in\rho(T_\cY)$.
We have shown that
\begin{align}\label{vide}
\rho(T_\cX)\subseteq\rho(T_\cY)\setminus\set0
.\end{align}

Conversely, suppose that $\gl\in\rho(T_\cY)$, and let 
$R'_\gl:=(\gl-T_\cY)\qw:\cY\to\cY$
denote the corresponding resolvent.
Since $T:\cX\to\cY$, we may define the operator $Q:=I+R'_\gl T$ on $\cX$.
For any $x\in\cX$, we then have, since $Tx\in\cY$,
\begin{align}
  (\gl-T)Qx = (\gl-T)x+(\gl-T)R'_\gl Tx
= \gl x-Tx+ Tx =\gl x
\end{align}
and
\begin{align}
Q  (\gl-T)x &= (\gl-T)x+ R'_\gl T(\gl-T)x
= (\gl-T)x+ R'_\gl (\gl-T)Tx
= \gl x-Tx+ Tx 
\notag\\&
=\gl x.
\end{align}
Hence, if also $\gl\neq0$, then
$\gl\qw Q$ is an inverse of $\gl-T$ on $\cX$, and thus $\gl\in\rho(T_\cX)$.
Consequently, $\rho(T_\cY)\setminus\set0\subseteq\rho(\cX)$.
Thus equality holds in \eqref{vide}, and thus \eqref{lxy} holds.

\pfitemref{LXYb}
Let $\gG$ be a sufficiently small circle around $\gl$, 
such that $\gG$ and its interior are disjoint from
$\gs(T_\cX)\setminus\set\gl$.
Then the spectral projections $\Pi_\gl(T_\cX)$ and $\Pi_\gl(T_\cY)$ are both
obtained by integrating the respective resolvents along $\gG$, as in
\eqref{e1tx}.  
If $\gl'\in\gG$, then, as shown in the proof of \ref{LXYa},  
$(\gl'-T_\cY)\qw$ is
the restriction of $(\gl'-T_\cX)\qw$ to $\cY$;
hence it follows that
$\Pi_\gl(T_\cY)$ is the restriction of the projection $\Pi_\gl(T_\cX)$ to $\cY$.
Consequently, 
\begin{align}\label{yxa}
\Pi_\gl\cY=\xpar{\Pi_\gl\cX}\cap\cY.  
\end{align}
Moreover, $T$ maps $\Pi_\gl\cX$ into itself, and  the restriction of $T$
to $\Pi_\gl\cX$ is invertible 
(since its spectrum is \set\gl, and $\gl\neq0$), 
and thus onto.
Since $T:\cX\to\cY$, it follows that $\Pi_\gl\cX\subseteq\cY$.
Combined with \eqref{yxa}, this yields $\Pi_\gl\cX=\Pi_\gl\cY$ as
asserted. 

\pfitemref{LXYc}
An immediate consequence of \ref{LXYa} and \ref{LXYb}.
\end{proof}

\begin{lemma}\label{LXZ}
  Let $\cN$ be a closed subspace of a complex Banach space
$\cX$,  and let $\cZ:=\cX/\cN$.
Suppose that $T$ is a bounded operator on $\cX$ such that
$Tn=0$ for every $n\in \cN$.
Then $T$ can also be regarded as an operator on $\cZ$,
and the following holds.
\begin{romenumerate}
\item \label{LXZa}
If $\cN\neq\set0$, then $  \gs(T_\cX)= \gs(T_\cZ)\cup\set0$.
(If $\cN=\set0$, then trivially $\gs(T_\cX)= \gs(T_\cZ)$.)
\item \label{LXZb}
If $\gl\neq0$ is an isolated point in  $\gs(T_\cX)$,
then 
$\Pi_\gl\cN=\set0$, and thus $\Pi_\gl(T_\cX)$ induces an operator
on $\cZ=\cX/\cN$; this induced operator equals $\Pi_\gl(T_\cZ)$.
Moreover, the quotient map $\cX\to\cZ$ is a bijection
$\Pi_\gl(T_\cX)\cX\to\Pi_\gl(T_\cZ)\cZ$.

\item \label{LXZc}
$T_\cX$ is \slqc{} if and only if $T_\cZ$ is \slqc.
$T_\cX$ is small if and only if $T_\cZ$ is small.

\end{romenumerate}
\end{lemma}

\begin{proof}
  That $T$ can be regarded as an operator on the quotient space $\cZ$
is well known.
Moreover, $\cZ^*$ is identified with the closed subspace 
$\set{x^*\in X^*:x^*(\cN)=0}$ of $\cX^*$. 

If $x^*\in X^*$ and $n\in N$, then $\innprod{T^*x^*,n}=\innprod{x^*,Tn}=0$;
thus $T^*x^*\in \cZ^*$ for every $x^*\in\cX^*$.
Hence, we can apply \refL{LXY} to $T^*$ on the spaces 
$\cX^*$ and $\cZ^*\subseteq\cX^*$.

\pfitemref{LXZa}
If $\cN\neq\set0$, then \eqref{lxy} yields
$\gs(T^*_{\cX^*})=\gs(T^*_{\cZ^*})\cup\set0$,
and thus $\gs(T_{\cX})=\gs(T_{\cZ})\cup\set0$
by \eqref{gs*}.

\pfitemref{LXZb}
By \eqref{gs*}, $\gl$ is an isolated point of $\gs(T^*_{X^*})$.
Recall also that (by the argument in the proof of \refC{Cgs})
$\Pi_\gl(T)^*=\Pi_\gl(T^*)$, for any of the spaces $\cX$ and $\cZ$. 
\refL{LXY}\ref{LXYb} thus shows that
$\Pi_\gl(T_\cX)^*:\cX^*\to\cZ^*$. 
Hence, if $n\in \cN$, then for any  $x^*\in \cX^*$
we have $\innprod{x^*,\Pi_\gl n}= \innprod{\Pi_\gl^*x^*, n}=0$,
and thus $\Pi_\gl n=0$. Hence $\Pi_\gl\cN=\set0$ as claimed.

Moreover, if $\pi:\cX\to\cZ$ is the quotient mapping, then
$\pi^*:\cZ^*\to\cX^*$ is the inclusion mapping, and \refL{LXY}\ref{LXYb} 
shows also that
$\Pi_\gl(T_\cX)^*\pi^*=\pi^*\Pi_\gl(T_\cZ)^*$.
Hence, by taking adjoints, 
\begin{align}\label{piz}
\pi\Pi_\gl(T_\cX)=\Pi_\gl(T_\cZ)\pi,
\end{align}
which shows that $\Pi_\gl(T_\cX)$ induces $\Pi_\gl(T_\cZ)$ on $\cZ$.
Furthermore, \eqref{piz} also implies
\begin{align}
  \pi\Pi_\gl(T_\cX)\cX=\Pi_\gl(T_\cZ)\pi\cX=\Pi_\gl(T_\cZ)\cZ,
\end{align}
and thus $\pi$ maps $\Pi_\gl(T_\cX)\cX$ onto $\Pi_\gl(T_\cZ)\cZ$.
Moreover, $\pi$ is injective on $\Pi_\gl(T_\cX)\cX$,
since $\pi x=0$ for some $x\in\Pi_\gl(T_\cX)\cX$ means that $x\in \cN$, and thus
$x=\Pi_\gl(T_\cX)x=0$ as shown above.

\pfitemref{LXZc}
An immediate consequence of \ref{LXZa} and \ref{LXZb}.
\end{proof}

We end this appendix with a standard definition.

\begin{definition}\label{Dre}
  Let $T$ be a bounded operator in a complex Banach space $B$.
Let $r(T)$ denote the spectral radius of $T$.
Furthermore, consider all decompositions $B=F\oplus H$ as a direct sum of
two closed 
$T$-invariant subspaces such that $\dim(F)<\infty$, and define
the \emph{essential spectral radius} of $T$ by
\begin{align}\label{re}
  r_e(T):=\inf\bigset{r(T\restr_H): B=F\oplus H \text{ as above}}.
\end{align}
\end{definition}
\begin{remark}\label{Rre}
It is easily seen that the definition \eqref{re} is equivalent to
\cite[Definition~XIV.1]{HeHe}.
There are several other, equivalent, definitions; for example, $r_e(T)$
equals the spectral radius of $T$ in the Banach algebra $\cB(B)/\cK(B)$, 
where $\cB(B)$ is the Banach algebra of bounded linear operators and
$\cK(B)$ is the ideal of compact operators. 
For this, and the relation to the \emph{essential spectrum}
(which has several, non-equivalent, versions), see \eg{}
\cite[\S1.4]{EdmundsEvans} 
and 
\cite[p.~243]{Kato}.
\end{remark}

\begin{remark}\label{Rqc}
Taking $F=\set0$ and $H=B$ in \eqref{re} 
shows that $r_e(T)\le r(T)$ for every $T$.
An operator $T$ is \emph{quasi-compact} if $r_e(T)<r(T)$.
(See \cite[Definition II.1]{HeHe} for another, equivalent, definition.)
\end{remark}

\section{A technical lemma}\label{AppC}

We state an elementary lemma that is used in the proof of \refT{T2}.
\begin{lemma}
	\label{lem:tecLemma}
	Let $\alpha\in\mathbb R$ and $k\geq 0$. Then, as \ntoo,
	\begin{align}\label{skar1}
	\sum_{j=1}^n j^{-1-\mathrm i\alpha} \log^k(\nicefrac nj)
	=\begin{cases}
	\displaystyle (1+o(1))\frac{\log^{k+1} n}{k+1}&\text{ if }\alpha=0,\\
	\displaystyle O(\log^k n)&\text{ if }\alpha\neq 0.
	\end{cases}
	\end{align}
\end{lemma}

\begin{proof}
We first approximate the sum by an integral.
  Let $g_n(x):=x^{-1-\iia}\log^k(n/x)$, $x\ge1$.
Then, assuming in the sequel $n\ge2$,
\begin{align}\label{skar2}
  g_n'(x)=(-1-\iia)x^{-2-\iia}\log^k(n/x)-kx^{-2-\iia}\log^{k-1}(n/x)
=O\bigpar{(\log^k n)x\qww}, \qquad x\ge1.
\end{align}
Hence, for $j\ge1$,
\begin{align}\label{skar3}
  g_n(j)-\int_j^{j+1}g_n(x)\dd x 
=   \int_j^{j+1}\bigpar{g_n(j)-g_n(x)}\dd x 
=	O\bigpar{(\log^k n)j\qww}.
\end{align}
Consequently, with the change of variables $x=n/y$,
\begin{align}\label{skar4}
  	\sum_{j=1}^n j^{-1-\mathrm i\alpha} \log^k(\nicefrac nj)&
=\sumjn g_n(j)
= g_n(n) + \sum_{j=1}^{n-1}
 \Bigpar{\int_j^{j+1}g_n(x)\dd x + O\bigpar{(\log^k n)j\qww}}
\notag\\&
=\int_1^{n}g_n(x)\dd x + O\bigpar{\log^k n}
=\int_1^{n} \frac{\log^k(n/x)}{x^{1+\iia}}\dd x+ O\bigpar{\log^k n}
\notag\\&
=n^{-\iia}\int_1^{n} \frac{\log^k y}{y^{1-\iia}}\dd y+ O\bigpar{\log^k n}
.\end{align}
It thus suffices to consider the final integral in \eqref{skar4}.

If $\ga=0$, then
	\begin{align}\label{skar5}
\int_1^{n} \frac{\log^ky}{y^{1-\iia}}\dd y
=
\int_1^{n} \frac{\log^ky}{y}\dd y
=\lrsqpar{\frac{\log^{k+1}y}{k+1}}^{n}_1
=\frac{\log^{k+1} n}{k+1},
	\end{align}
and thus \eqref{skar1} follows in this case.

If $\alpha\neq 0$, we use integration by parts and get
\begin{align}\label{skar6}
  \int_1^n\log^k(y)y^{\iia-1}\dd y&
=\Bigsqpar{\log^k(y)\frac{y^{\iia}}{\iia}}_1^n
-\frac{k}{\iia}\int_1^n\log^{k-1}(y) y^{\iia-1}\dd y 
\notag\\&
=O\bigpar{\log^k(n)}+\int_1^nO\bigpar{\log^{k-1}(n)}\frac{\ddx y}{y}
=O\bigpar{\log^k(n)}.
\end{align}
Hence, \eqref{skar1} follows from \eqref{skar4} in this case too.
\end{proof}

\begin{remark}
  It is possible to show that for $\ga\neq0$, the sum in \eqref{skar1} is
  asymptotic to $\zeta(1+\iia)\log^k n$.
Moreover, for any $\ga$,
an asymptotic expansion with an arbitrary number of terms may be
obtain by singularity analysis as in similar examples in 
\cite[Section 3.1]{FillFK}.
\end{remark}

\newcommand\AAP{\emph{Adv. Appl. Probab.} }
\newcommand\JAP{\emph{J. Appl. Probab.} }
\newcommand\JAMS{\emph{J. \AMS} }
\newcommand\MAMS{\emph{Memoirs \AMS} }
\newcommand\PAMS{\emph{Proc. \AMS} }
\newcommand\TAMS{\emph{Trans. \AMS} }
\newcommand\AnnMS{\emph{Ann. Math. Statist.} }
\newcommand\AnnPr{\emph{Ann. Probab.} }
\newcommand\CPC{\emph{Combin. Probab. Comput.} }
\newcommand\JMAA{\emph{J. Math. Anal. Appl.} }
\newcommand\RSA{\emph{Random Structures Algorithms} }
\newcommand\DMTCS{\jour{Discr. Math. Theor. Comput. Sci.} }

\newcommand\AMS{Amer. Math. Soc.}
\newcommand\Springer{Springer-Verlag}
\newcommand\Wiley{Wiley}

\newcommand\vol{\textbf}
\newcommand\jour{\emph}
\newcommand\book{\emph}
\newcommand\inbook{\emph}
\def\no#1#2,{\unskip#2, no. #1,} 
\newcommand\toappear{\unskip, to appear}

\newcommand\arxiv[1]{\texttt{arXiv:#1}}
\newcommand\arXiv{\arxiv}

\def\nobibitem#1\par{}

\end{document}